\documentclass[aop]{imsart}

\RequirePackage{amsthm,amsmath,amsfonts,amssymb,dsfont,mathtools,xcolor}
\RequirePackage[numbers,sort&compress]{natbib}
\usepackage[shortlabels]{enumitem}
\mathtoolsset{showonlyrefs,showmanualtags}
\usepackage[hidelinks]{hyperref}
\RequirePackage{graphicx}

\startlocaldefs
\theoremstyle{plain}
\newtheorem{theorem}{Theorem}
\newtheorem{proposition}{Proposition}[section]
\newtheorem{lemma}{Lemma}[section]
\theoremstyle{remark}
\newtheorem{remark}{Remark}[section]

\numberwithin{equation}{section}
\renewcommand{\complement}{{c}}
\newcommand{\E}{\mathds{E}}
\renewcommand{\P}{\mathds{P}}

\renewcommand{\Pr}{\mathbf{P}}

\newcommand{\R}{\mathbb{R}}
\newcommand{\T}{\mathbb{T}}
\newcommand{\C}{\mathbb{C}}
\newcommand{\Z}{\mathbb{Z}}
\newcommand{\N}{\mathbb{N}}
\newcommand{\dd}{{\rm d}}
\newcommand{\fstop}{\; \text{.}}
\newcommand{\comma}{\; \text{,}\;\;}
\newcommand{\semicolon}{\; \text{;}\;\;}
\newcommand{\restr}[1]{\rvert_{#1}}

\newcommand{\tonde}[1]{\left(#1\right)}
\newcommand{\quadre}[1]{\left[#1\right]}
\newcommand{\ttonde}[1]{\big(#1\big)}

\newcommand{\abs}[1]{\left\lvert#1\right\rvert}

\newcommand{\emparg}{{\,\cdot\,}}

\newcommand{\eqdef}{\coloneqq}
\newcommand{\defeq}{\eqqcolon}
\newcommand{\car}{\mathds{1}}
\newcommand{\scalar}[2]{\left\langle #1\, \middle \vert\, #2 \right\rangle}

\newcommand{\norm}[1]{\left\lVert#1\right\rVert}
\newcommand{\tnorm}[1]{\big\lVert#1\big\rVert}
\newcommand{\ttnorm}[1]{\|#1\|}
\newcommand{\set}[1]{\left\{#1\right\}}							
\newcommand{\tset}[1]{\big\{#1\big\}}
 
\newcommand{\eps}{\varepsilon}

\newcommand{\purple}[1]{{\color{black}{#1}}}

\newcommand{\cA}{\ensuremath{\mathcal A}} 
 
\newcommand{\cC}{\ensuremath{\mathcal C}} 
\newcommand{\cD}{\ensuremath{\mathcal D}} 
 
\newcommand{\cF}{\ensuremath{\mathcal F}}

\newcommand{\cM}{\ensuremath{\mathcal M}} 
\newcommand{\cN}{\ensuremath{\mathcal N}}

\newcommand{\cQ}{\ensuremath{\mathcal Q}} 
\newcommand{\cR}{\ensuremath{\mathcal R}} 
\newcommand{\cS}{\ensuremath{\mathcal S}} 
\newcommand{\cT}{\ensuremath{\mathcal T}} 
\newcommand{\cU}{\ensuremath{\mathcal U}}

\newcommand{\cY}{\ensuremath{\mathcal Y}}

\newcommand{\bbY}{{\ensuremath{\mathbb Y}}}

\endlocaldefs

\begin{document}

\begin{frontmatter}
\title{Tiny fluctuations of the averaging process\\ around its degenerate steady state}
\runtitle{Tiny fluctuations of the averaging process}

\begin{aug}
\author[A]{\fnms{Federico}~\snm{Sau}\ead[label=e1]{federico.sau@units.it}},
\address[A]{
University of Trieste\printead[presep={,\ }]{e1}}

\end{aug}

\begin{abstract}
We analyze nonequilibrium  fluctuations of the averaging process on $\T_\eps^d$, a continuous degenerate Gibbs sampler running over the edges of the discrete $d$-dimensional torus.  We show that, if we start from a smooth deterministic non-flat interface, recenter,  blow-up  by a non-standard CLT-scaling factor $\theta_\eps=\eps^{-(d/2+1)}$, and rescale diffusively,   Gaussian fluctuations emerge in the limit  $\eps\to 0$. These fluctuations are purely dynamical, zero at times $t=0$ and $t=\infty$, and non-trivial for $t\in (0,\infty)$. We fully determine the correlation matrix of the limiting  noise, non-diagonal as soon as $d\ge 2$.
The main technical challenge in this stochastic homogenization procedure lies in  a LLN for a  weighted space-time average of  squared discrete gradients. We accomplish this through a Poincaré inequality with respect to the underlying randomness of the edge updates, a tool from Malliavin calculus in Poisson space. This inequality, combined with sharp gradients' second moment estimates, yields quantitative variance bounds without prior knowledge of the limiting mean. 
Our method avoids higher (e.g., fourth) moment bounds, which seem inaccessible with the present techniques. 
\end{abstract}



\end{frontmatter}
\setcounter{tocdepth}{1}
\tableofcontents

\section{Introduction, model and main results}
The averaging process on a graph is a continuous-space Markov chain, which is commonly interpreted  as  an opinion dynamics,  a distributed algorithm, or an interface moving through a randomized sequence of deterministic local updates (see, e.g., \cite{boyd_et_al_randomized_2006,aldous_lecture_2012,aldous2013interacting,movassagh_repeated2022} and references therein). Its  dynamics goes as follows. Attach i.i.d.\ Poisson clocks to edges, and assign real values to vertices; at the arrival times of these  clocks, update the values with their average. As time runs, the averaging process converges to a flat configuration, and one major problem in the field is that of quantifying the speed of convergence to its degenerate equilibrium in terms of characteristic features of the underlying graph  \cite{chatterjee2020phase,quattropani2021mixing,caputo_quattropani_sau_cutoff_2023}.  	

In this paper, we examine scaling limits of the averaging process on the discrete $d$-dimensional torus $\T_\eps^d$,  $d\ge 1$, subjected to a diffusive space-time rescaling.  In this setting, the averaging process, say $u_t^\eps \in \R^{\T_\eps^d}$, evolves by replacing, at rate $\eps^{-2}$ for each nearest neighbor pair $x,y \in \T_\eps^d$,  the values $(u_{t^-}^\eps(x), u_{t^-}^\eps(y))$ with
\begin{equation} (u_t^\eps(x),u_t^\eps(y))\eqdef (\tfrac12\tonde{u_{t^-}^\eps(x)+u_{t^-}^\eps(y)},\tfrac12\tonde{u_{t^-}^\eps(x)+u_{t^-}^\eps(y)})\fstop
\end{equation}
Alternatively, in the language of stochastic homogenization,  $u_t^\eps$
is the  solution to the discrete parabolic problem  
$\partial_t u_t^\eps =\nabla_*^\eps\cdot \tonde{a_\eps(t,\emparg)\,\nabla^\eps u_{t^-}^\eps}$ on $\T_\eps^d$,  in which the coefficient field $a_\eps=(a_\eps(t,x))_{t\ge 0,\, x\in \T_\eps^d}$ is random, time-dependent, and formally given by
\begin{equation}\label{eq:a-eps-coefficient-field}
	a_\eps(t,x)= \frac{\eps^2}2\begin{pmatrix}
		\dd N_t^{\eps,1}(x)&&\\
		&\ddots &\\
		&& \dd N_t^{\eps,d}(x)
	\end{pmatrix}\comma\qquad t \ge 0\comma x \in \T_\eps^d\fstop
\end{equation}
Here, $\dd N_t^{\eps,i}(x)$ stands for the increments of the Poisson process of intensity $\eps^{-2}$ attached to the pair  $x$, $x+\eps e_i\in \T_\eps^d$. This homogenization problem is then degenerate, in the sense that the coefficient field $a_\eps$ is a linear combination of Dirac deltas and,  thus, does not satisfy any ellipticity conditions: there is no  $C>0$ satisfying neither $a_\eps\ge C$, nor $a_\eps \le C$ (in the matrix sense).

The first ergodic theorems in stochastic homogenization for linear elliptic and parabolic problems date back to the seminal works of Kozlov \cite{kozlov_averaging_1979} and Papanicolau and Varadhan \cite{papanicolau_varadhan_boundary_1981}. Since then,  there has been an intense  activity in considerably extending these qualitative results. In fact, even quantitative features of the solutions, such as regularity   and fluctuations limit theorems, are by now well understood in the uniform ellipticity context (see, e.g., \cite{armstrong_kuusi_mourrat_quantitative_2019,josien_otto_annealed_2022}, \cite{gu_mourrat_scaling_2016,duerinckx_gloria_otto_structure_2020}, and references therein), while the only degenerate examples are limited to supercritical percolation clusters \cite{armstrong_dario_elliptic_2018,dario_optimal_2021}, rigid inclusions \cite{duerinckx_gloria_quantitative_2022} and log-normal coefficient fields \cite{gloria_qi_quantitative_2024,clozeau_gloria_qi_quantitative_2024}.

In this article, we move one step forward by investigating dynamical fluctuations of the random $u_t^\eps$ in this time-dependent degenerate setting (cf.\ \eqref{eq:a-eps-coefficient-field}). This result is the content of Theorem \ref{th:flu}, which could be regarded as the functional central limit theorem (FCLT) after the (quantitative) law of large numbers (LLNs) established in \cite{sau_concentration_2023}. As recently worked out in some examples of  random walks in time-dependent random environment (see, e.g., \cite{rhodes_homogenization_2008,biskup_rodriguez_limit_2018}, all concerned with LLNs of random walks' probabilities),  our FCLT  also draws upon the fundamental idea that adding a mixing dynamics  potentially tempers spatial degeneracy. In our case, this dynamics consists in  having zero conductances which suddenly take the value $+\infty$ at the occurrence of a Poisson mark. As in most degenerate instances, this mixing mechanisms plays a crucial role in recovering on large scales regularity estimates which would deterministically hold true in the uniform ellipticity setting.

Our approach is probabilistic, based on martingales,   as developed by Holley and 	Stroock \cite{holley_generalized_1978} and successfully exploited within the context of interacting particle systems	 (see, e.g., the monographs \cite{de_masi_mathematical_1991,kipnis_scaling_1999}). Nonetheless, in contrast with most systems studied in that realm,  the averaging process $u_t^\eps$ has no  truly random ergodic states (they all consist of deterministic flat configurations). This lack of microscopic fluctuations at equilibrium prevents us to employ methods based on relative entropy, not having a clear notion of local equilibrium to compare the law of $u_t^\eps$ with (see, e.g.,  \cite{jara2018nonequilibrium,jara_landim_stochastic_2023} and references therein). 

We overcome this obstacle by combining the aforementioned martingale-based approach with two main ingredients,  typical of infinite-dimensional stochastic analysis. On the one hand, we control a key quantity ---  squared discrete gradients of $u_t^\eps$ --- by expressing its expected value as an infinite series of iterated integrals. On the other hand, inspired by a series of recent works in stochastic homogenization, e.g., \cite{duerinckx_otto_higher_2020,duerinckx_gloria_multiscale_concentration_2020,duerinckx_gloria_multiscale_constructive_2020}, we leverage the Poisson nature of the updates by operating with tools from Malliavin calculus (in Poisson space).  These are the two building blocks in the proof of our second main result, Theorem \ref{th:LLN}, a LLN for weighted space-time averages of squared discrete gradients.
Theorem \ref{th:LLN} is also an 	 essential step in the characterization of the limiting Gaussian process from Theorem \ref{th:flu}, and reveals  the smallness of the fluctuations $u_t^\eps-\E^\eps[u_t^\eps]$, as well as the non-diagonal correlation structure of the limiting driving noise (despite the form of the coefficient field $a_\eps$ in \eqref{eq:a-eps-coefficient-field} with i.i.d.\ diagonal entries).

In conclusion, as originally proposed in \cite{aldous_lecture_2012,aldous2013interacting}, the averaging process (together with a long list of companions \textquotedblleft random averages\textquotedblright\ models, e.g., \cite{ferrari_fontes_fluctuations_1998,balazs_rassoul-agha_seppalainen_random_2006,como_scaling2011,lanchier_critical_2012,haggstrom2014further,banerjee2020rates,fontes_machado_zuaznabar_scaling_2023}, with either	 site or edge updates) is a natural nonequilibrium  system of stochastic moving interfaces, whose quantitative features are still largely unexplored. These include, for instance,  a regularity theory and worst-case mixing times   on geometric settings other than the torus. The latter problem becomes particularly  interesting in light of the works \cite{chatterjee2020phase,caputo_quattropani_sau_cutoff_2023}, which show that on the simplest examples of expander graphs the averaging process' degenerate local update rule  dramatically affects the timescales relevant to relaxation, if compared to the homogenized dynamics. 
We believe the combination of Markov chains'  and stochastic homogenization techniques that  we develop here to  be fruitful also  in this context.

\smallskip 

We now present the model and our main results. Before that, let us introduce some general notation.
All throughout the article,  $\T^d\eqdef (\R/\Z)^d$, with $d\ge 1$, is the $d$-dimensional torus,  while $\T_\eps^d$ denotes its lattice discretization with mesh size $\eps \in (0,1)$. For simplicity, we  always assume  $\eps^{-1}\in \N$. Moreover, with a slight abuse of notation, we often identify the set $\T_\eps^d$ with the  undirected graph obtained by connecting vertices of the form $x$ and $y=x\pm \eps e_i\in \T_\eps^d$ with an edge, where $e_1,\ldots, e_d\in \R^d$ denotes the canonical basis of $\R^d$. Hence, letting $\abs{\emparg}$ denote the usual Euclidean distance on $\R^d$, we will refer to sites $x, y\in \T_\eps^d$ satisfying $\abs{x-y}=\eps$ as nearest neighbor vertices. Finally, we shall implicitly imply  \textquotedblleft holds true for all $\eps \in (0,1)$ with $\eps^{-1}\in \N$\textquotedblright\ whenever $\eps$ appears in a statement without any other indications.
\subsection{Averaging process} The \textit{averaging process} on $\T_\eps^d$ is the continuous-time Markov process $(u_t^\eps)_{t\ge 0}$ which evolves on $\R^{\T_\eps^d}$ by updating, independently and at rate $\eps^{-2}$, the values at nearest neighbor vertices with their average value. A possible rigorous description of the model goes through the introduction of a family of i.i.d.\ Poisson processes \begin{equation}\label{eq:poisson-process}
	N^\eps=(N_t^{\eps,i}(x))_{t \ge 0,\,x\in \T_\eps^d,\, i=1,\ldots, d}
\end{equation} of intensity $\eps^{-2}$, with the arrivals of $N_\emparg^{\eps,i}(x)$ corresponding to an update among the two nearest neighbor vertices $x$ and $x+\eps e_i\in \T_\eps^d$. We write $\P^\eps$ and $\E^\eps$ for the associated probability law and expectation, respectively. 	Hence, for every starting configuration $u_0^\eps\in \R^{\T_\eps^d}$, $(u_t^\eps)_{t\ge 0}$ is defined as the unique \textit{c\`{a}dl\`{a}g} solution to the following finite system of SDEs driven by Poisson processes: for every $x\in \T_\eps^d$ and $t>0$, 
\begin{align}
	\dd u_t^\eps(x)&= \sum_{i=1}^d\dd N_t^{\eps,i}(x) \quadre{\tfrac12\tonde{u_{t^-}^\eps(x)+u_{t^-}^\eps(x+\eps e_i)}-u_{t^-}^\eps(x)}\\
	&\quad + \sum_{i=1}^d \dd N_t^{\eps,i}(x-\eps e_i)\quadre{\tfrac12\tonde{u_{t^-}^\eps(x)+u_{t^-}^\eps(x-\eps e_i)}-u_{t^-}^\eps(x)}\\
	&= \sum_{i=1}^d \dd N_t^{\eps,i}(x)\quadre{\tfrac12 \tonde{u_{t^-}^\eps(x+\eps e_i)-u_{t^-}^\eps(x)}}\\
	&\quad + \sum_{i=1}^d \dd N_t^{\eps,i}(x-\eps e_i)\quadre{\tfrac12 \tonde{u_{t^-}^\eps(x-\eps e_i)-u_{t^-}^\eps(x)}}\comma	
	\label{eq:SDE-poisson}
\end{align} 
with $\dd N_t^{\eps,i}(x)=N_t^{\eps,i}(x)-N_{t^-}^{\eps,i}(x)$ denoting the increments of the Poisson process. Further, note that the spacing $\eps$ in $\T_\eps^d$ and the intensity $\eps^{-2}$ of the Poisson processes  correspond to a diffusive space-time rescaling in  \eqref{eq:SDE-poisson}.

Since $\T_\eps^d$ is a finite graph, it readily follows that, as a strong solution to the system above, $(u_t^\eps)_{t\ge 0}$ exists unique for all initial conditions $u_0^\eps\in \R^{\T_\eps^d}$,  and is a Markov process.  Moreover, due to the connectedness of $\T_\eps^d$,  $u_t^\eps$ approaches, as $t\to \infty$ and for all $u_0^\eps\in \R^{\T_\eps^d}$, the flat profile with value $\langle u_0^\eps\rangle_\eps\in \R$, the spatial average of $\T_\eps^d$,
\begin{equation}
	u_t^\eps(x) \xrightarrow{t\to \infty} \langle u_0^\eps\rangle_\eps\eqdef \eps^d \sum_{y\in \T_\eps^d} u_0^\eps(y)\comma
\end{equation}
$\P^\eps$-a.s.\ and uniformly over $x\in \T_\eps^d$. In particular, we observe that, despite $u_t^\eps$ is, generically, a  random element of $\R^{\T_\eps^d}$ for all times $t> 0$, the corresponding unique steady state is deterministic and identically equal to $\langle u_0^\eps\rangle_\eps$. Hence, the averaging dynamics is trivial at equilibrium,  registering, in particular, no fluctuations.	

In this work, we analyze asymptotics as $\eps \to 0$ of {out-of-equilibrium fluctuations} for this model,  which, due to the degeneracy of the steady state, lacks a non-trivial notion of \emph{local equilibrium} --- a fundamental feature in the theory of hydrodynamic limits for interacting particle systems (see, e.g., \cite{kipnis_scaling_1999}). In words, we shall find a scaling factor $\theta_\eps\to \infty$ as $\eps\to0$ and a non-trivial process $(\cY_t)_{t\ge 0}$ describing well, in the diffusive regime as $\eps\to0$, the nonequilibrium fluctuations encoded in the centered fields 
\begin{equation}\label{eq:fluctuation-fields}
	\cY_t^\eps \eqdef \theta_\eps\set{ \eps^d\sum_{x\in \T_\eps^d}\tonde{u_t^\eps(x)-\E^\eps[u_t^\eps(x)]}\delta_x}\comma
\end{equation}  after taking spatial averages against suitable test functions.	Note that we consider a deterministic initial configuration $u_0^\eps\in \R^{\T_\eps^d}$, so to observe only {dynamical fluctuations}.

Maybe surprisingly at first, it turns out that,  instead of the most standard CLT-scaling $\eps^{-d/2} $, 	the averaging process' nonequilibrium fluctuations require an extra factor $\eps^{-1}$, thus, \begin{equation}\label{eq:theta-eps}
	\theta_\eps=\eps^{-(d/2+1)}\fstop
\end{equation}	  In this sense, fluctuations are unusually small in our context. 	Next to this, the limiting fluctuations $\cY_t$ become {Gaussian}, have rather explicit space-time correlations, and satisfy $\cY_0=\lim_{t\to \infty}\cY_t=0$. 

\subsection{Tiny nonequilibrium fluctuations}
In order to best describe the limit $\cY_t$, let us further manipulate the equations \eqref{eq:SDE-poisson} defining $u_t^\eps$.	
By   adopting the standard  notation, for all $g\in \R^{\T_\eps^d}$,  $x\in \T_\eps^d$ and $i=1,\ldots, d$,
\begin{align}\label{eq:grad}
	\begin{aligned}
		\nabla^{\eps,i} g(x)&\eqdef \eps^{-1}\tonde{g(x+\eps e_i)-g(x)}\comma &&\nabla^\eps g(x)\eqdef (\nabla^{\eps,i}g(x))_{i=1,\ldots, d}\comma
		\\ \nabla_*^{\eps,i}g(x)&\eqdef \eps^{-1}\tonde{g(x)-g(x-\eps e_i)}\comma &&\nabla_*^\eps g(x)\eqdef (\nabla_*^{\eps,i}g(x))_{i=1,\ldots, d}\comma
	\end{aligned}
\end{align}
\begin{equation}\label{eq:laplacian}
	\Delta_\eps g(x)\eqdef \sum_{i=1}^d \nabla_*^{\eps,i}\nabla^{\eps,i}g(x)
	=\sum_{i=1}^d \eps^{-2}\tonde{g(x+\eps e_i)+g(x-\eps e_i)-2 g(x)}\comma
\end{equation}
and letting ${\rm diag}(a)\in \R^{d\times d}$  denote the diagonal matrix having  $a=(a^i)_{i=1,\ldots, d}\in \R^d$ as diagonal elements, 
the system \eqref{eq:SDE-poisson} reads as
\begin{align}
	\dd u_t^\eps(x)&= \frac{\eps}{2}	 \sum_{i=1}^d\tset{ \dd  N_t^{\eps,i}(x){ \nabla^{\eps,i}u_{t^-}^\eps(x)}-\dd  N_t^{\eps,i}(x-\eps e_i){ \nabla^{\eps,i}u_{t^-}^\eps(x-\eps e_i)}}\\
	&= \frac{\eps^2}2\sum_{i=1}^d \nabla_*^{\eps,i}\big(\dd  N_t^{\eps,i}\, \nabla^{\eps,i}u_{t^-}^\eps\big)(x)\\
	&= \frac{\eps^2}2\, \nabla_*^\eps \cdot \tonde{{\rm diag}(\dd N_t^{\eps,\emparg})\nabla^\eps u_{t^-}^\eps}(x)\comma
\end{align}
or, equivalently, by passing to the compensated Poisson processes \begin{equation}\label{eq:poisson-process-compensated}
	\bar N_t^{\eps,i}(x)\eqdef N_t^{\eps,i}(x)-\eps^{-2}t\comma\qquad t\ge 0\comma x \in \T_\eps^d\comma i=1,\ldots, d\comma
\end{equation} as
\begin{equation}\label{eq:SDE-poisson-compensated}
	\dd u_t^\eps(x)=\tfrac12 \Delta_\eps u_t^\eps(x)\, \dd t+ \tfrac{\eps^2}2 \nabla_*^\eps\cdot \tonde{{\rm diag}(\dd \bar N_t^{\eps,\emparg})\nabla u_{t^-}^\eps}(x)\fstop
\end{equation}
Hence, next to a smoothening Laplacian term $\frac12\Delta_\eps u_t^\eps$, the stochastic dynamics is driven by an i.i.d.\ diagonal noise which multiplies the discrete gradients $\nabla^\eps u_{t^-}^\eps$ of the random averaging process.  The gradient $\nabla_*^\eps$ accounts for the conservative nature of the noise.

Our main task is to establish, after centering and scaling with $\theta_\eps$,  a homogenization principle for this noise term, in which two main effects equally contribute: on the one side, the random gradients $\nabla^\eps u_{t^-}^\eps$ get replaced by their deterministic counterparts; on the other side, the stochasticity lost in the previous step is superseded by the appearance of a non-diagonal structure of the limiting noise. More precisely, we show that the limiting  fluctuations $\cY_t$ solve (in a weak sense) the following {\rm SPDE} on $\T^d$
\begin{equation}\label{eq:SPDE}
	\dd \cY_t = \tfrac12 \Delta \cY_t - \tfrac12\nabla\cdot \tonde{\xi \,\nabla u_t}\comma
\end{equation} 
where:
\begin{itemize}
	\item $\nabla u_t$ is the gradient of the deterministic solution $u_t$ to the heat equation  on $\T^d$
	\begin{equation}\label{eq:heat-equation}
		\partial_t u= \tfrac12 \Delta u\semicolon
	\end{equation}
	\item  $\xi=(\xi^{i,j})_{i,j=1,\ldots, d}$ is a matrix-valued space-time white noise whose covariance is informally given, for some $\mathfrak a=\mathfrak a(d)\in (0,1]$, by
	\begin{equation}\label{eq:noise-structure}
			\E\big[\xi_t^{i,j}(x)\xi_s^{k,\ell}(y)\big]= \delta_0(t-s)\delta_0(x-y)\, \car_{i=k}\car_{j=\ell}\tonde{\tonde{1+\mathfrak a}\car_{i=j}+\frac{1-\mathfrak a}{d-1}\,\car_{i\neq j}}\comma
	\end{equation}
	for all  $x,y\in \T^d$, $t,s\ge 0$, and $i,j,k,\ell=1,\ldots, d$;
	\item a weak solution to \eqref{eq:SPDE} is meant to be  a solution to the corresponding martingale problem (cf.\ \eqref{eq:mart-problem}, \eqref{eq:mart-2}).
\end{itemize}

The precise statement of this convergence result is the content of the following theorem, before which we need to introduce some notation.
For all $k\in \N_0\cup\set{\infty}$, let $\cC^k(\T^d)$ be the space of $k$-continuous differentiable functions on $\T^d$. We  write $\cC(\T^d)=\cC^0(\T^d)$.
For all $\alpha\in \R$, $H^\alpha(\T^d)$ denotes the $\alpha^{th}$-order fractional Sobolev space on $\T^d$, defined as the closure of $\cC^\infty(\T^d)$ with respect to the following Hilbertian (semi)norm 
\begin{equation}\label{eq:H-norm}
	\norm{f}_{H^\alpha(\T^d)}\eqdef \bigg(\sum_{m\in \Z^d}\tonde{1+|m|^2}^\alpha  |\scalar{f}{\phi_m}_{L^2(\T^d)}|^2\bigg)^\frac12\comma\qquad f \in \cC^\infty(\T^d)\fstop
\end{equation} Here, $\phi_m(x)\eqdef e^{2\pi im\cdot x}$, for all $x\in \T^d$ and $m\in \Z^d$. It is well-known that $H^\alpha(\T^d)$ is a Hilbert space, and that,  for $\alpha = k\in \N_0$, $H^\alpha(\T^d)$ coincides with the usual Sobolev space $W^{2,k}(\T^d)$. Moreover, $H^\alpha(\T^d)\hookrightarrow H^\beta(\T^d)$ for all $\alpha>\beta$, and    $\cY_t^\eps\in H^{-\alpha}(\T^d)$ for all $\alpha > d/2$. For more details and properties, we refer to, e.g., \cite[Section C]{jara2018nonequilibrium}.	Finally, we write $\cD([0,\infty);H^{-\alpha}(\T^d))$ (resp.\ $\cC([0,\infty);H^{-\alpha}(\T^d))$) for the space of \textit{c\`{a}dl\`{a}g} (resp.\ continuous) trajectories taking values in  $H^{-\alpha}(\T^d)$ (see,  e.g., \cite[\S2--3]{billingsley_convergence_1999}). 
\begin{theorem}[Nonequilibrium fluctuations]\label{th:flu}
	Fix $\alpha>3+d/2$ and $u_0\in \cC^2(\T^d)$.	Then, there exists $\mathfrak a=\mathfrak a(d)\in (0,1]$ such that, when initializing the averaging process $(u_t^\eps)_{t\ge 0}$ with $u_0^\eps = u_0\restr{\T_\eps^d}$, the following convergence in law 	
	\begin{equation}
		(\cY_t^\eps)_{t\ge 0} \xRightarrow{\eps\to 0} (\cY_t)_{t\ge 0}\comma\quad \text{in}\ \cD([0,\infty);H^{-\alpha}(\T^d))\comma
	\end{equation}
	holds true for the corresponding fluctuation fields $(\cY_t^\eps)_{t\ge 0}$ given in \eqref{eq:fluctuation-fields} with $\theta_\eps= \eps^{-(d/2+1)}$, where $(\cY_t)_{t\ge 0}$ is the unique process in $\cC([0,\infty);H^{-\alpha}(\T^d))$ solving the {\rm SPDE} \eqref{eq:SPDE}--\eqref{eq:noise-structure}, with:
	\begin{itemize}
		\item initial condition $\cY_0=0$;
		\item $(u_t)_{t\ge 0}$ therein being the solution to \eqref{eq:heat-equation} with initial condition $u_0\in \cC^2(\T^d)$.
	\end{itemize}  
\end{theorem}
\begin{remark}[Constant $\mathfrak a$]\label{remark:constant-a}
	Since our proofs do not require it, we did not try to extract the precise numerical value of the constant $\mathfrak a$ appearing in Theorem \ref{th:flu}. We just mention that
	\begin{equation}
		\mathfrak a=1\quad \text{for}\ d=1\comma\qquad \mathfrak a=\frac{\pi}{3\pi-4}\quad \text{for}\ d=2\comma
	\end{equation}	
	while we have some more abstract expressions for $d\ge 3$ guaranteeing $\mathfrak a\in (0,1)$. For more details, see Remark \ref{remark:formula-a}.
\end{remark}
\subsection{A law of large numbers for the squared gradients}\label{sec:sketch-proof+LLN} In order to prove Theorem \ref{th:flu}, we follow a probabilistic approach, taking advantage of the Markovian nature of the averaging process $u_t^\eps$. In formulas, recalling \eqref{eq:SDE-poisson-compensated} and  that $\cY_0^\eps=0$,  this corresponds to decompose the fields $\cY_t^\eps$ in \eqref{eq:fluctuation-fields}, when tested against a test function $f\in \cC^\infty(\T^d)$, as 
\begin{equation}\label{eq:martingale-decomposition}
	\cY_t^\eps(f)= \int_0^t \cY_s^\eps(\tfrac12 \Delta_\eps f)\, \dd s + \cM_t^\eps(f)\comma\qquad t \ge 0\comma
\end{equation}
where, for $\theta_\eps=\eps^{-(d/2+1)}$,
\begin{equation}\label{eq:martingale-eps}
	\begin{aligned}
		\cM_t^\eps(f)&\eqdef 	\int_0^t {\theta_\eps\,\eps^{d+2}}	\sum_{x\in \T_\eps^d}  \tonde{\tonde{{\rm diag}(\dd \bar N_s^{\eps,\emparg})\nabla u_{s^-}^\eps}(x)}\cdot \tonde{\tfrac12\nabla^\eps f(x)}\\
		&= \int_0^t \eps^{d/2+1} \sum_{i=1}^d\sum_{x\in \T_\eps^d} {\tonde{\dd \bar N_s^{\eps,i}(x)\,\nabla^{\eps,i}u_{s^-}^\eps(x)}}\tonde{\tfrac12 \nabla^{\eps,i}f(x)}\comma\qquad t \ge 0\comma
	\end{aligned}
\end{equation}
is a square-integrable martingale (with respect to the filtration generated by the process $u_t^\eps$). The martingale $\cM_t^\eps(f)$ has jumps and may be described through its predictable quadratic variation, referred to as $\langle \cM^\eps(f)\rangle_t$ and explicitly  given by
\begin{equation}\label{eq:predictable-q-v}
	\langle \cM^\eps(f)\rangle_t = \int_0^t  \eps^d\sum_{x\in \T_\eps^d}\sum_{i=1}^d \tonde{\nabla^{\eps,i}u_s^\eps(x)}^2 \tonde{\tfrac12 \nabla^{\eps,i}f(x)}^2\dd s\comma\qquad t \ge 0\fstop
\end{equation} 
This derivation uses elementary properties of the compensated Poisson processes in \eqref{eq:poisson-process-compensated}, such as $\E[\bar N_t^{\eps,i}(x)]=0$ and $\langle \bar N^{\eps,i}(x)\rangle_t = \eps^{-2}t$, as well as their independence. 

In view of the decomposition in \eqref{eq:martingale-decomposition}, the proof of Theorem \ref{th:flu} boils down to proving the following three claims:
\begin{enumerate}[(i)]
	\item the sequence $((\cY_t^\eps)_{t\ge 0})_\eps$ is tight in $\cD([0,\infty);H^{-\alpha}(\T^d))$, for $\alpha > 3+d/2$;
	\item all limiting processes $(\cY_t)_{t\ge 0}$ are continuous;
	\item all limiting processes $(\cY_t)_{t\ge 0}$ are solutions to the following martingale problem
	\begin{equation}\label{eq:mart-problem}
		\cY_t(f)=\int_0^t \cY_s(\tfrac12 \Delta f)\, \dd s  + \cM_t(f)\comma \qquad t\ge 0\comma f\in \cC^\infty(\T^d)\comma
	\end{equation}
	where $(\cM_t(f))_{t\ge 0}$ is a martingale (with respect to  the filtration generated by $(\cY_t)_{t\ge 0}$) whose predictable quadratic variation is deterministic and compatible with the noise structure described in \eqref{eq:SPDE}--\eqref{eq:noise-structure}.
\end{enumerate}
By the form of the limiting martingale and standard uniqueness results due to Holley and Stroock \cite{holley_generalized_1978}, (ii) and (iii) together would uniquely characterize the limit as the one described in Theorem \ref{th:flu}.

While the proofs of the steps (i) and (ii) above are not particularly problematic, establishing the  third step on the identification of the limiting martingale represents the major challenge of the proof of Theorem \ref{th:flu}. More specifically, thanks to classical martingale convergence theorems, our main task reduces to prove convergence in probability of the corresponding predictable quadratic variations. Reading out the explicit expression  of these quadratic variations from the right-hand side of \eqref{eq:predictable-q-v}, this means proving a weak law of large numbers for  weighted space-time averages of squares of discrete gradients of $u_t^\eps$. Most of our analysis is devoted to the proof of this result, which, for later reference, we now state in form of a theorem.

\begin{theorem}[LLN for squared discrete gradients]\label{th:LLN}
	Fix $u_0\in \cC^2(\T^d)$, and let $(u_t^\eps)_{t\ge 0}$ and $(u_t)_{t\ge 0}$ denote, respectively, the averaging process and the solution to the heat equation \eqref{eq:heat-equation},  started from $u_0^\eps =u_0\restr{\T_\eps^d}$ and $u_0\in \cC^2(\T^d)$. Then,  for all $g=(g^i)\in (\cC([0,\infty)\times \T^d))^d$ and $t\ge 0$, we have
	\begin{equation}
		\E^\eps\big[\tonde{\varGamma^\eps_t(g)-\varGamma_t(g)}^2\big]\xrightarrow{\eps\to 0} 0\comma
	\end{equation}
	where
	\begin{align}\label{eq:Gamma-eps}
		\varGamma_t^\eps(g)&\eqdef \sum_{i=1}^d\int_0^t \eps^d \sum_{x\in \T_\eps^d}\tonde{\nabla^{\eps,i}u_s^\eps(x)}^2 g_s^i(x)\, \dd s\comma\\
		\label{eq:Gamma}
		\varGamma_t(g)&\eqdef \sum_{i,j=1}^d \tonde{\tonde{1+\mathfrak a}\car_{i=j}+\frac{1-\mathfrak a}{d-1}\,\car_{i\neq j}} \int_0^t \int_{\T^d}  \tonde{\nabla^j u_s(x)}^2 g_s^i(x)\, \dd x \dd s\comma
	\end{align}
	and $\mathfrak a=\mathfrak a(d)\in (0,1]$ is the constant in Theorem \ref{th:flu} and Remark \ref{remark:constant-a}.
\end{theorem}
\subsection{Proof ideas and techniques}
The starting point in the proof of the LLN in Theorem \ref{th:LLN} is a variance estimate (Theorem \ref{th:var}). We derive this by tools from Malliavin calculus in Poisson space (see, e.g., \cite{last_penrose_poisson_2011,last_penrose_lectures_2018} and references therein). More specifically, we will resort to one specific tool in that context:  the Poincaré inequality. In words, this inequality allows to estimate the variance of a functional $F=F(N^\eps)$ of the underlying Poisson process in terms of (a norm of) its derivative $DF$ with respect to an infinitesimal variation of  $N^\eps$. The functional  $DF$ is referred to as the Malliavin derivative. 

As we will show, for the averaging process, $DF$ comes with a nice probabilistic dynamical interpretation in terms of an evolving discrepancy (Section \ref{sec:discrepancy}). We exploit this point of view to turn this Poincaré inequality into an effective bound. In order to achieve this, we derive some new second-moment estimates for the discrete gradients of the averaging process, and combine them with properties of the averaging process earlier developed in \cite{aldous_lecture_2012,quattropani2021mixing,sau_concentration_2023}.  Such improved estimates are key in the proof of Theorem \ref{th:LLN}, for controlling the variance of $\varGamma_t^\eps(g)$, and are sharp enough to determine the limit of its mean (Theorem \ref{th:mean}).

\subsubsection{Poincaré inequality in Poisson space}\label{sec:poincare-ineq-poisson-space}
In our setting, this Poincaré inequality may be rigorously stated as follows. Recall the  Poisson process $N^\eps$ in \eqref{eq:poisson-process} of the edge updates, 
which we can describe, more precisely, through its intensity measure $\gamma_\eps$ on $\bbY_\eps\eqdef [0,\infty)\times \T_\eps^d \times \set{1,\ldots, d}$ (endowed with the corresponding product Borel $\sigma$-field),  given by the  product measure 
\begin{equation}\gamma_\eps=\eps^{-2}\dd t \otimes \nu_{\T_\eps^d}\otimes \nu_{\set{1,\ldots, d}}\fstop
\end{equation} Here,  $\dd t$ stands for the Lebesgue measure on $[0,\infty)$, whereas, for a discrete set $S$, $\nu_S$ denotes the counting measure on $S$. Letting $\mathbf N(\bbY_\eps)$ be the space of integer-valued $\sigma$-finite measures on $\bbY_\eps$, we then interpret $N^\eps$ as a random element in the measurable space $\mathbf N(\bbY_\eps)$,  endowed with  $\cN^\eps$,  the smallest $\sigma$-field making the mappings $m\in \mathbf N(\bbY_\eps)\mapsto m(A)\in \R$ measurable, for all measurable sets $A\subset \bbY_\eps$. Now, $\P^\eps$ and $\E^\eps$ stand for the  law and expectation of $N^\eps$, respectively.

Then,  the  \textit{Poincaré inequality in Poisson space} reads, for every measurable function  $F:\mathbf N(\bbY_\eps)\to \R$, as
\begin{align}
	{\rm Var}_\eps(F)&\eqdef \E^\eps\big[\big(F(N^\eps)-\E^\eps[F(N^\eps)]\big)^2\big]\le \int_{\bbY_\eps} \E^\eps\big[ \tonde{D_{(t,x,i)}F(N^\eps)}^2\big]\, \dd\gamma_\eps(t,x,i)\\
	&\qquad = \eps^{-(d+2)}\int_0^\infty \eps^d \sum_{x\in \T_\eps^d}\sum_{i=1}^d \E^\eps\big[\tonde{D_{(t,x,i)}F(N^\eps)}^2\big]\, \dd t\fstop
	\label{eq:poincare-ineq-F}
\end{align}
In this formula, for all $(t,x,i)\in \bbY_\eps$, $D_{(t,x,i)}$ is the  difference operator defined, for all measurable $F:\mathbf N(\bbY_\eps)\to \R$, as
\begin{equation}
	D_{(t,x,i)}F(m)\eqdef F(m+\delta_{(t,x,i)})-F(m)\comma\qquad m \in \mathbf N(\bbY_\eps)\comma
\end{equation}
with $\delta_{(t,x,i)}$ denoting the Dirac delta at $(t,x,i)\in \bbY_\eps$. Roughly speaking, $D_{(t,x,i)}$ plays the role of derivative at the \textquotedblleft point\textquotedblright\ $m\in \mathbf N(\bbY_\eps)$ in the \textquotedblleft direction\textquotedblright\ $(t,x,i)\in \bbY_\eps$. In our context, if we keep in mind the interpretation of $N^\eps$ as identifying the edge updates for the averaging process,  $D_{(t,x,i)}F$ measures the effect on $F$ of the addition of an  update at time $t$ between the nearest neighbors $x$ and $x+\eps e_i\in \T_\eps^d$.

\subsubsection{Application of Poincaré inequality}
In proving Theorem \ref{th:LLN}, we will apply the Poincaré inequality \eqref{eq:poincare-ineq-F} to functionals $F$ of the form $\varGamma_t^\eps(g)$ given as  in \eqref{eq:Gamma-eps}, for some  $t> 0$, $g=(g^i)_{i=1,\ldots, d}\in (\cC([0,\infty)\times\T^d))^d$, and initial condition $u_0^\eps\in \R^{\T_\eps^d}$. This yields
\begin{align}\label{eq:poincare-ineq-Gamma-eps}
	\begin{aligned}
		{\rm Var}_\eps(\varGamma_t^\eps(g))&\le \eps^{-(d+2)}\int_0^\infty\eps^d\sum_{x\in \T_\eps^d}\sum_{i=1}^d \E^\eps\big[\tonde{D_{(s,x,i)}\varGamma_t^\eps(g)}^2\big]\, \dd s\\
		&=\eps^{-(d+2)}\int_0^t \eps^d\sum_{x\in \T_\eps^d}\sum_{i=1}^d \E^\eps\big[\tonde{D_{(s,x,i)}\varGamma_t^\eps(g)}^2\big]\,\dd s\comma
	\end{aligned}
\end{align} 
where the identity uses the fact that $\varGamma_t^\eps(g)$ does not depend on the history after time $ t$. Moreover, since $(\varGamma_t^\eps(g))_{t\ge 0}$ is an additive functional of $(u_t^\eps)_{t\ge 0}$, $D_{(s,x,i)}\varGamma_t^\eps(g)$ may be further simplified, for $s\le t$, as
\begin{equation}
	D_{(s,x,i)}\varGamma_t^\eps(g)= \sum_{j=1}^d\int_s^t \eps^d \sum_{y\in \T_\eps^d} \set{\tonde{\nabla^{\eps,j} E_{(s,x,i)}u_r^\eps(y)}^2-\tonde{\nabla^{\eps,j} u_r^\eps(y)}^2} g_r^j(y)\, \dd r\comma
\end{equation}
where, for $r\ge s$, $E_{(s,x,i)} u_r^\eps$ is the averaging process following the  edge updates $N^\eps$, conditioned to be equal, at time $r=s$, to
\begin{equation}
	y\in \T_\eps^d\longmapsto	E_{(s,x,i)}u_s^\eps(y)\eqdef \begin{dcases}
		\tfrac12\tonde{u_s^\eps(x)+u_s^\eps(x+\eps e_i)} &\text{if}\ y=x\ \text{or}\ y=x+\eps e_i\\
		u_s^\eps(y) &\text{else}\fstop
	\end{dcases}
\end{equation}
Analogously to  $D_{(s,x,i)}$ (which acts on functionals $F$), the operator $E_{(s,x,i)}$ acts on $(u_r^\eps)_{r\ge 0}$ by imposing an extra averaging-update at vertices $x$ and $x+\eps e_i\in \T_\eps^d$ at time $s\ge 0$. Furthermore, by linearity of the averaging process, $E_{(s,x,i)}u_r^\eps$ may be further decomposed as
\begin{equation}\label{eq:u+w}
	E_{(s,x,i)}u_r^\eps = u_r^\eps + w_r^{\eps,(s,x,i)}\comma\qquad r \ge s\comma
\end{equation}
namely, a sum of the original process $u_r^\eps \in \R^{\T_\eps^d}$ and a \textquotedblleft discrepancy\textquotedblright\ $w_r^{\eps,(s,x,i)}\in \R^{\T_\eps^d}$, both evolving as two averaging processes with two different initializations, sharing the same edge updates. In particular, they are not independent.	In view of this, we may recast  an estimate of the right-hand side of \eqref{eq:poincare-ineq-Gamma-eps}  into a bound on the evolution of such a discrepancy, and a central part of our analysis deals with establishing precise quantitative estimates on it and its discrete gradients (see Section \ref{sec:proof-overview-lemmas} for more details).

\subsubsection{Ultracontractivity  and gradient estimates for the averaging process}
Next to the Poincaré inequality in Poisson space, our analysis builds on two functional analytic  properties of the averaging flow on the discrete torus:  ultracontractivity and second-moment gradient estimates.

\smallskip \noindent
\textit{Ultracontractivity.}
By ultracontractivity we mean a bound of the following form: \purple{for  some $C=C(d)>0$},
\begin{equation}\label{eq:ultracontractivity}
	\E^\eps\big[\norm{u_t^\eps}_{L^2(\T_\eps^d)}^2\big]\le \langle u_0^\eps\rangle_\eps^2 + C\norm{u_0^\eps}_{L^1(\T_\eps^d)}^2\ttonde{\tonde{t\vee \eps^2}^{-d/2}\vee 1}\comma\qquad t\in (0,1]\fstop
\end{equation}
Here,  $\norm{\emparg}_{L^p(\T_\eps^d)}$, $p\in [1,\infty]$, denotes the $L^p$-norm on $\T_\eps^d$ with respect to the uniform measure $\eps^d$, i.e., for all	 $g\in \R^{\T_\eps^d}$, 
\begin{equation}\label{eq:Lp-space}
	\norm{g}_{L^p(\T_\eps^d)}\eqdef \bigg(\eps^d\sum_{x\in \T_\eps^d}\abs{g(x)}^p\bigg)^\frac1p\ \text{for}\ p\in [1,\infty)\comma\quad \norm{g}_{L^\infty(\T_\eps^d)}\eqdef \sup_{x\in \T_\eps^d}\abs{g(x)}\fstop
\end{equation}
We write $\scalar{\emparg}{\emparg}_{L^2(\T_\eps^d)}$ for the inner product in $L^2(\T_\eps^d)$.

Estimates of type \eqref{eq:ultracontractivity} for parabolic equations are classical, dating back to Nash's work (\cite{nash_continuity1958}, see also \cite{fabes_new1986}), 	and are well-known also in the graph context (see, e.g., \cite{diaconis1996nash}). For the averaging process,  they were first derived   only recently in \cite{quattropani2021mixing},    and further refined on $\T_\eps^d$, establishing the stronger inequality \cite[Theorem 2.1]{sau_concentration_2023}
\begin{equation}\label{eq:ultracontractivity-H1}
	\E^\eps\big[\norm{u_t^\eps}_{H^1(\T_\eps^d)}^2\big]\le C\norm{u_0^\eps}_{L^1(\T_\eps^d)}^2 \tonde{t\vee \eps^2}^{-(d/2+1)}\comma\qquad t\ge 0\fstop
\end{equation}
On the left-hand side above, $\norm{\emparg}_{H^1(\T_\eps^d)}$ reads, 
for all $g\in \R^{\T_\eps^d}$, as
\begin{equation}\label{eq:H1-space}
	\norm{g}_{H^1(\T_\eps^d)}\eqdef \bigg(\eps^d \sum_{x\in \T_\eps^d}\sum_{i=1}^d \abs{\nabla^{\eps,i }g(x)}^2\bigg)^\frac12\fstop
\end{equation}
In order to see why \eqref{eq:ultracontractivity-H1} is stronger than \eqref{eq:ultracontractivity}, it suffices to integrate over time the following   identity due to Aldous and Lanoue \cite[Eq.\ (2.7)]{aldous_lecture_2012}:
\begin{equation}\label{eq:AL-identity}
	\frac{\dd}{\dd t}\,\E\big[\norm{u_t^\eps}_{L^2(\T_\eps^d)}^2\big] = -\frac12\,\E\big[\norm{u_t^\eps}_{H^1(\T_\eps^d)}^2\big]\comma\qquad t\ge0\fstop
\end{equation}

Finally, we observe that the bounds in \eqref{eq:ultracontractivity} and \eqref{eq:ultracontractivity-H1} are effective for all $t=O(1)$, as well as for very small times $t\ll 1$. We emphasize that such estimates for $t\ll 1$ will play a key role in Lemma \ref{lemma:varLambda-3}, when proving  Theorem \ref{th:LLN}.

\smallskip \noindent
\textit{Gradient estimates.}
The second important and new ingredient is a thorough analysis of the dynamics of discrete gradients of the averaging process on the torus on a diffusive space-time scale. This will be carried out  in Sections \ref{sec:gradient-estimates} and \ref{sec:mean}. In a nutshell, this approach goes through: 
\begin{enumerate}[(a)]
	\item first, solving explicitly the SDEs governing the evolution of $(\nabla^{\eps,i}u_t^\eps)_{t\ge 0} \subset \R^{\T_\eps^d}$ (Lemma \ref{lemma:mild-solution-gradients});
	\item  then, reducing asymptotics of the second moments $\E\big[(\nabla^{\eps,i}u_t^\eps(x))^2\big]$ to the analysis of  infinitely-many iterated integrals  (Propositions \ref{pr:second-moments}, \ref{pr:mean-2} and \ref{pr:mean-1}).
\end{enumerate}   This latter step is what allows us to determine the non-diagonal form of the limiting correlations (cf.\ \eqref{eq:noise-structure}). Next to asymptotics, we  extract some new pseudo-contractivity and small-time estimates for gradients' second moments of independent interest (Lemmas \ref{lemma:H1-contractivity}--\ref{lemma:short-time}). 

It is worth to mention that similar  estimates were recently obtained in \cite[Proposition 2.2]{sau_concentration_2023}, aiming at capturing the correct order of magnitude of $\E\big[(\nabla^{\eps,i}u_t^\eps(x))^2\big]$ for a  wider time window, including the regime $t\gg 1$ as $\eps\to0$. However, while the approach in the aforementioned work only provides bounds, in the present article,  we are able to determine exact asymptotics by restricting to times $t=O(1)$. Finally, 
our analysis relies solely on second moments, not involving higher (e.g., fourth) moment estimates, whose control seems to be out of reach with the present techniques.

\subsection{Structure of the paper}
The rest of the paper is organized as follows. In Section \ref{sec:gradient-estimates}, we collect some technical results concerning discrete gradients of the averaging process which will be used all throughout. In Sections \ref{sec:variance} and \ref{sec:mean} we prove Theorem \ref{th:LLN} on the LLN for $\varGamma_t^\eps$ defined in \eqref{eq:Gamma-eps}. More specifically, in the former section, we prove that variances vanish as $\eps \to 0$, while in the latter one we establish the convergence of the means. Section \ref{sec:proof-th-1} is devoted to the proof of Theorem \ref{th:flu}. In \hyperref[sec:app]{Appendix}, we prove a few auxiliary results employed in Sections \ref{sec:variance} and \ref{sec:mean}.
\section{Discrete gradients}\label{sec:gradient-estimates}
In this section, we derive a number of technical results on the random discrete gradients of $u_t^\eps$ on $\T_\eps^d$, of use for the subsequent sections.	Before starting, let us introduce some notation.

After recalling that $\cN^\eps$ is the $\sigma$-field associated to $N^\eps$ (cf.\ Section \ref{sec:poincare-ineq-poisson-space}), we define $\cF^\eps\eqdef \cN^\eps\otimes \sigma(u_0^\eps)$ as the product  $\sigma$-field associated to the Poisson process and the averaging process' initial conditions. We write $(\cF_t^\eps)_{t\ge 0}$ for the corresponding filtration, with respect to which $(u_t^\eps)_{t\ge 0}\subset \R^{\T_\eps^d}$ is adapted.
In particular, $u_0^\eps\in \R^{\T_\eps^d}$ is $\cF_0^\eps$-measurable. 
Furthermore, recalling the definitions of $\nabla^{\eps,i}$, $\nabla_*^{\eps,i}$, $\nabla^\eps$, $\nabla_*^\eps$ and $\Delta_\eps$ from \eqref{eq:grad}--\eqref{eq:laplacian}, we introduce the random walk $(X_t^\eps)_{t\ge 0}$ on $\T_\eps^d$, with semigroup $(P_t^\eps)_{t\ge 0}$ and infinitesimal generator $\tfrac12 \Delta_\eps$. Moreover, we let $(p_t^\eps(x,y))_{t\ge 0,\,x, y\in \T_\eps^d}$ denote the corresponding transition probabilities. In what follows, we will repeatedly use that all discrete gradients commute, namely,
\begin{equation}\label{eq:commutation-gradients}
	\nabla^{\eps,i}\nabla^{\eps,j} = \nabla^{\eps,j}\nabla^{\eps,i}\comma \quad \nabla^{\eps,i}\nabla_*^{\eps,j}=\nabla_*^{\eps,j}\nabla^{\eps,i}\comma\quad \nabla_*^{\eps,i}\nabla_*^{\eps,j}=\nabla_*^{\eps,j}\nabla_*^{\eps,i}\comma
\end{equation}
for all $i,j=1,\ldots, d$, from which we readily obtain the following well-known intertwining relations for the semigroup and the gradients on $\T_\eps^d$: 
\begin{equation}\label{eq:intertwining}
	P_t^\eps \nabla^{\eps,i}=\nabla^{\eps,i}P_t^\eps\comma\quad P_t^\eps \nabla_*^{\eps,i}=\nabla_*^{\eps,i}P_t^\eps\comma\qquad t \ge 0\comma i=1,\ldots, d\fstop
\end{equation}
We remark that $\Delta_\eps$ and $P_t^\eps$ are symmetric operators on  $L^2(\T_\eps^d)$ (cf.\ \eqref{eq:Lp-space}) and that, by translation and rotational invariance of the random walk's dynamics on $\T_\eps^d$, we have $p_t^\eps(x,y)=p_t^\eps(y,x)=p_t^\eps(y-x,0)=p_t^\eps(x-y,0)$.	 
Finally, in order to lighten the notation,  we  write $\P=\P^\eps$ and $\E=\E^\eps$ all throughout.

\subsection{Stochastic dynamics  of discrete gradients}

As a first step, we write the random discrete gradients of $u_t^\eps$ in a mild-solution form. For this purpose, recall the definition of compensated Poisson processes $\bar N^\eps=(\bar N_t^{\eps,i}(x))_{t\ge 0,\, x\in \T_\eps^d,\, i=1,\ldots, d}$, as well as the equations \eqref{eq:SDE-poisson-compensated} solved by $u_t^\eps$. 
\begin{lemma}[Mild solution]\label{lemma:mild-solution-gradients} Fix $u_0^\eps\in \R^{\T_\eps^d}$. Then, 
	$\P$-a.s., we have
	\begin{equation}\label{eq:grad-u}
		\nabla^{\eps,i}u_t^\eps(x)= P_t^\eps\tonde{\nabla^{\eps,i}u_0^\eps}(x) + \sum_{j=1}^d \sum_{y\in \T_\eps^d}\int_0^t b_{t-s}^{\eps,i,j}(x,y)\, \nabla^{\eps,j}u_{s^-}^\eps(y)\, \dd \bar N_s^{\eps,j}(y)\comma
	\end{equation}
	for all  $t> 0$,  $x\in \T_\eps^d$ and $i=1,\ldots, d$, 
	where
	\begin{equation}\label{eq:bt-eps}
		b_t^{\eps,\purple{i,j}}(x,y)\eqdef \tfrac12\tonde{p_t^\eps(x+\eps e_i,y)+p_t^\eps(x-\eps e_j,y)-p_t^\eps(x+\eps e_i-\eps e_j,y)-p_t^\eps(x,y)}\fstop
	\end{equation}
\end{lemma}
\begin{proof}
	Passing to the discrete gradients in \eqref{eq:SDE-poisson-compensated}, we obtain
	\begin{align}\label{eq:SDE-gradients}
		\dd	(\nabla^{\eps,i} u_t^\eps)(x)& = 	
		\tfrac12 \nabla^{\eps,i}\Delta_\eps u_t^\eps(x)\, \dd t+ \tfrac{\eps^2}2 \nabla^{\eps,i}\big(\nabla_*^\eps\cdot \tonde{{\rm diag}(\dd \bar N_t^{\eps,\emparg})\nabla u_{t^-}^\eps}\big)(x)\\
		&=\tfrac12 \Delta_\eps (\nabla^{\eps,i}u_t^\eps)(x)+
		\frac{\eps^2}2\sum_{j=1}^d \nabla_*^{\eps,j}\nabla^{\eps,i}\big(\dd \bar N_t^{\eps,j}\,\nabla^{\eps,j}u_{t^-}^\eps\big)(x)\comma
	\end{align}
	where for the second step we used the commutation relations \eqref{eq:commutation-gradients}. Note that the above finite system of SDEs (with compensated-Poisson noise) has the form $\dd f_t= \tfrac12 \Delta_\eps f_t\,\dd t+ F_t$, hence, recalling that $P_t^\eps=\exp\tonde{\tfrac{t}2\Delta_\eps}$,  $\nabla^{\eps,i}u_t^\eps$  admits the following explicit form, $\P$-a.s., 
	\begin{align}
		\nabla^{\eps,i}u_t^\eps\purple{(x)}&= P_t^\eps\big(\nabla^{\eps,i}u_0^\eps\big)(x)+ \int_0^t P_{t-s}^\eps\bigg(\frac{\eps^2}2\sum_{j=1}^d \nabla_*^{\eps,j}\nabla^{\eps,i}\tonde{\dd \bar N_s^{\eps,j}\,\nabla^{\eps,j}u_{s^-}^\eps}\bigg)(x)\\
		&= P_t^\eps\big(\nabla^{\eps,i}u_0^\eps\big)(x) + \sum_{j=1}^d \frac{\eps^2}{2}\int_0^t  \nabla_*^{\eps,j}\nabla^{\eps,i}P_{t-s}^\eps\big(\dd \bar N_s^{\eps,j}\,\nabla^{\eps,j}u_{s^-}^\eps\big)(x)\comma
	\end{align} 
	where the last step employs the
	intertwinings in \eqref{eq:intertwining}. The desired claim now follows by recalling \eqref{eq:bt-eps} and observing that, for all $g\in \R^{\T_\eps^d}$, we have
	\begin{align}
		\eps^2\,\nabla_*^{\eps,j}\nabla^{\eps,i} P_t^\eps g(x)&= \eps\tonde{\nabla^{\eps,i}P_t^\eps g(x)-\nabla^{\eps,i}P_t^\eps g(x-\eps e_j)}\\
		&= {P_t^\eps g(x+\eps e_i)-P_t^\eps g(x)-P_t^\eps g(x-\eps e_j+\eps e_i)+P_t^\eps g(x-\eps e_j)}\\
		&= \sum_{y\in \T_\eps^d} 2b_t^{\eps,\purple{i,j}}(x,y)\, g(y)\fstop
	\end{align}
	This concludes the proof of the lemma.
\end{proof}
\subsection{Second moments of discrete gradients}
Let us iterate the  expression from Lemma \ref{lemma:mild-solution-gradients} to obtain formulas  for the expectation of $\nabla^{\eps,i}u_t^\eps(x)\, \nabla^{\eps,i}v_t^\eps(x)$, where $(u_t^\eps)_{t\ge 0}$ and $(v_t^\eps)_{t\ge 0}$ are two copies of the averaging process starting from $u_0^\eps$ 	and $v_0^\eps\in \R^{\T_\eps^d}$, respectively,  driven by the same Poisson noise. For this and subsequent results, the following quantity defined in terms of $b_t^{\eps,i,j}(x,y)$ (see \eqref{eq:bt-eps})
\begin{equation}\label{eq:B}
	q_t^{\eps,i,j}(x,y)\eqdef 	\big(\eps^{-1}b_t^{\eps,i,j}(x,y)\big)^2\comma\qquad t\ge 0\comma x,y\in \T_\eps^d\comma i,j=1,\ldots, d\comma
\end{equation}
will play a key role, whose main properties are collected in the next lemma and proved in   \hyperref[sec:app]{Appendix}. 
\begin{lemma}\label{lemma:Q1} For all $t\ge 0$, $x,y \in \T_\eps^d$ and $i,j=1,\ldots, d$, we have 
	\begin{equation}\label{eq:B-symmetry}
		q_t^{\eps,i,j}(x,y)=q_t^{\eps,i,j}(y,x)\fstop
	\end{equation}
	Moreover, the following quantity
	\begin{equation}\label{eq:B-invariance}
		\cQ_\eps(t)\eqdef	\sum_{j=1}^d \sum_{y\in \T_\eps^d} q_t^{\eps,i,j}(x,y)
	\end{equation}
	depends only on $d\ge 1$ and $t\ge 0$ (thus, not on $x\in \T_\eps^d$ and $i=1,\ldots, d$), 
	and satisfies
	\begin{equation}\label{eq:B-12}
		\int_0^\infty \cQ_\eps(t)\, \dd t =\frac12\fstop
	\end{equation}
	
\end{lemma}

We  state a notational remark. 
\begin{remark}[Notation]\label{rmk:notation}
	Here and all throughout, we adopt the following standard notation: for all  $t> 0$ and $k\in \N$, 
	\begin{equation}
		[0,t]_>^k\eqdef \set{(s_1,\ldots, s_k)\in [0,t]^k: s_1>\ldots >s_k}\comma
	\end{equation}
	while, for all $h\in \cC([0,t]_>^k)$, 
	\begin{equation}
		\int_{[0,t]_>^k} \dd s_1\cdots \dd s_k\, h(\emparg)\eqdef \int_0^t\int_0^{s_1}\cdots \int_0^{s_{k-1}} h(\emparg)\, \dd s_k\cdots \dd s_1\fstop
	\end{equation}
	Moreover, within this context, we implicitly identify $s_0\eqdef t$ and, whenever $x\in \T_\eps^d$ is a fixed vertex and $i=1,\ldots, d$ a fixed direction, $y_0\eqdef x$ and $j_0=i$. We employ analogous notations when replacing $0$ (resp.\ $t$) by $-\infty$ (resp.\ $\infty$), in $[0,t]_>^k$.
\end{remark}
\begin{proposition}\label{pr:second-moments}
	For all $u_0^\eps, v_0^\eps\in \R^{\T_\eps^d}$, $t> 0$, $x\in \T_\eps^d$ and $i=1,\ldots, d$,  we have
	\begin{equation}\label{eq:second-moment-mixed}
		\E\big[\nabla^{\eps,i}u_t^\eps(x)\, \nabla^{\eps,i}v_t^\eps(x)\big] = \sum_{k=0}^\infty \Pi_t^{\eps,i,k}(u_0^\eps,v_0^\eps)(x)\comma
	\end{equation}
	the above series being absolutely convergent.
	Here, each summand is defined, for $k=0$, as	
	\begin{equation}\label{eq:Pi-eps-0}
		\Pi_t^{\eps,i,0}(u_0^\eps,v_0^\eps)(x)\eqdef P_t^\eps(\nabla^{\eps,i}u_0^\eps)(x)\, P_t^\eps(\nabla^{\eps,i}v_0^\eps)(x)\comma
	\end{equation} while, for $k\ge 1$, as	(identifying $s_0=t$, $y_0=x$ and $j_0=i$, cf.\ Remark \ref{rmk:notation})
	\begin{align}\label{eq:Pi-eps-k}
		&\qquad 	\Pi_t^{\eps,i,k}(u_0^\eps,v_0^\eps)(x)\\
		&  \eqdef \int_{[0,t]_>^k} \dd s_1\cdots \dd s_k \sum_{y_1,\ldots, y_k\in \T_\eps^d}\sum_{j_1,\ldots, j_k=1}^d \tonde{\prod_{\ell=1}^k q_{s_{\ell-1}-s_\ell}^{\eps,j_{\ell-1},j_\ell}(y_{\ell-1},y_\ell)} \Pi_{s_k}^{\eps,j_k,0}(u_0^\eps,v_0^\eps)(y_k)\comma
	\end{align}
	where $q^\eps$ is given in \eqref{eq:B}.
	In particular, for $u_0^\eps=v_0^\eps$, we have
	\begin{equation}\label{eq:second-moment}
		\E\big[\big(\nabla^{\eps,i} u_t^\eps(x)\big)^2\big] = \sum_{k=0}^\infty \Pi_t^{\eps,i,k}(u_0^\eps)(x)\comma\quad \text{with}\ \Pi_t^{\eps,i,k}(u_0	^\eps)\eqdef\Pi_t^{\eps,i,k}(u_0^\eps,u_0^\eps)\fstop	
	\end{equation}
\end{proposition}
\begin{proof}The two proofs of \eqref{eq:second-moment-mixed} and \eqref{eq:second-moment} are essentially the same; for notational convenience, let us discuss in detail the one of \eqref{eq:second-moment}, and only quickly comment on that of \eqref{eq:second-moment-mixed}.
	Recall \eqref{eq:grad-u} and note that
	$P_t^\eps\big(\nabla^{\eps,i}u_0^\eps\big)$ is $\cF_0^\eps$-measurable, while
	\begin{equation}
		\E\bigg[\int_0^t b_{t-s}^{\eps,i,j}(x,y)\, \big(\dd \bar N_s^{\eps,j}\,\nabla^{\eps,j}u_{s^-}^\eps\big)(y)\bigg|\cF_0^\eps\bigg] = 0\comma
	\end{equation}
	because $s\mapsto b_{t-s}^{\eps,i,j}(x,y)\, \nabla^{\eps,j} u_{s^-}^\eps(y)$ is predictable and $s\mapsto \bar N^{\eps,j}_s(y)$ is a martingale with respect to $(\cF_s^\eps)_{s\ge 0}$. Hence, 
	the left-hand side of \eqref{eq:second-moment} reads as
	\begin{align}
		&\E\big[\big(\nabla^{\eps,i} u_t^\eps(x)\big)^2\big]\\
		&\qquad= \tonde{P_t^\eps(\nabla^{\eps,i}u_0^\eps)(x)}^2 + \E\bigg[\bigg(\sum_{j=1}^d\sum_{y\in \T_\eps^d}\int_0^t b_{t-s}^{\eps,i,j}(x,y)\,\big(\dd \bar N_s^{\eps,j}(y)\, \nabla^{\eps,j} u_{s^-}^\eps(y)\big)\bigg)^2\bigg]\\
		&\qquad= \tonde{P_t^\eps(\nabla^{\eps,i}u_0^\eps)(x)}^2 + \sum_{j=1}^d\sum_{y\in \T_\eps^d}\int_0^t \big(b_{t-s}^{\eps,i,j}(x,y)\big)^2\, \eps^{-2}\, 	\E\big[\big(\nabla^{\eps,j}u_s^\eps(y)\big)^2\big]\, \dd s\\
		&\qquad = \tonde{P_t^\eps(\nabla^{\eps,i}u_0^\eps)(x)}^2 + \sum_{j=1}^d\sum_{y\in \T_\eps^d}\int_0^t q_{t-s}^{\eps,i,j}(x,y)\, 	\E\big[\big(\nabla^{\eps,j}u_s^\eps(y)\big)^2\big]\, \dd s		\comma
	\end{align}
	where for the second step we used the fact that $\bar N^\eps$ is a family of i.i.d.\ martingales, each with predictable quadratic variation given by $s\mapsto \eps^{-2}\,s$, whereas  the third step employs the definition in \eqref{eq:B}. Hence, letting $U^{\eps,i}_t(x)\eqdef \E\big[(\nabla^{\eps,i}u_t^\eps(x))^2\big]$ and   recalling \eqref{eq:Pi-eps-0}, we may	 rewrite the above identity as
	\begin{equation}
		U_t^{\eps,i}(x)= \Pi_t^{\eps,i,\purple{0}}(u_0^\eps)(x)+ \sum_{j=1}^d\sum_{y\in \T_\eps^d}\int_0^t q_{t-s}^{\eps,i,j}(x,y)\, U_s^{\eps,j}(y)\, \dd s\comma
	\end{equation}
	which, by non-negativity of all three functions involved,  may be iterated infinitely often, yielding \eqref{eq:second-moment}, which \textit{a priori} may be infinite. However, recalling \eqref{eq:Pi-eps-0} and \eqref{eq:B-12}, we first  obtain\purple{, for all $k\in \N_0$,}
	\begin{equation}
		\sup_{t\ge 0}\tnorm{	\Pi_t^{\eps,j,\purple{k}}(u_0^\eps)}_{L^\infty(\T_\eps^d)}\le \tnorm{\nabla^{\eps,j}u_0^\eps}_{L^\infty(\T_\eps^d)}^2 \purple{2^{-k}}<\infty\comma\qquad j=1,\ldots, d\comma
	\end{equation}
	and, then,
	\begin{equation}
		\E\big[(\nabla^{\eps,i}u_{\purple{t}}^\eps(x))^2\big]\le \max_{j=1,\ldots, d}\norm{\nabla^{\eps,j}u_0^\eps}_{L^\infty(\T_\eps^d)}^2\sum_{k=0}^\infty 2^{-k}=2\max_{j=1,\ldots, d}\norm{\nabla^{\eps,j}u_0^\eps}_{L^\infty(\T_\eps^d)}^2<\infty\fstop
	\end{equation} 
	This proves the desired claim when $u_0^\eps=v_0^\eps$, i.e., \eqref{eq:second-moment}. When $u_0^\eps \neq v_0^\eps$, we similarly have
	\begin{align}
		&\E\big[\nabla^{\eps,i}u_t^\eps(x)\, \nabla^{\eps,i}v_t^\eps(x)\big] \\
		&\qquad= \Pi_t^{\eps,i,0}(u_0^\eps,v_0^\eps)(x)+ \sum_{j=1}^d \sum_{y\in \T_\eps^d}\int_0^t q_{t-s}^{\eps,i,j}(x,y)\,	 \E\big[\nabla^{\eps,i}u_s^\eps(x)\, \nabla^{\eps,i}v_s^\eps(x)\big]\,\dd s\comma
	\end{align}
	which, thanks to
	\begin{equation}
		\sup_{t\ge 0}	\tnorm{\Pi_t^{\eps,j,0}(u_0^\eps,v_0^\eps)}_{L^\infty(\T_\eps^d)}\le \tnorm{\nabla^{\eps,j}u_0^\eps}_{L^\infty(\T_\eps^d)}\tnorm{\nabla^{\eps,j}v_0^\eps}_{L^\infty(\T_\eps^d)}<\infty\comma\qquad j=1,\ldots,d\comma
	\end{equation}
	and the identity in \eqref{eq:B-12}, may be iterated, yielding \eqref{eq:second-moment-mixed}. 
\end{proof}

\subsection{Technical estimates}
All lemmas in this section are required for the proof of Lemma \ref{lemma:varLambda-2} below. We start with the following upper bound on integrals of second moments of discrete gradients, which may be thought of as an annealed pseudo-contractivity bound in $H^1(\T_\eps^d)$.	
\begin{lemma}\label{lemma:H1-contractivity} Recall the definition of $\norm{\emparg}_{H^1(\T_\eps^d)}$ from \eqref{eq:H1-space}. Then, for all $u_0^\eps\in \R^{\T_\eps^d}$ and $t> 0$, we have
	\begin{equation}\label{eq:ub-H1-norm}
		\E\big[\norm{u_t^\eps}_{H^1(\T_\eps^d)}^2\big]\le 2 \norm{u_0^\eps}_{H^1(\T_\eps^d)}^2\fstop
	\end{equation}
\end{lemma}
\begin{proof}
	By definition of $\norm{\emparg}_{H^1(\T_\eps^d)}$ and Proposition \ref{pr:second-moments}, we have
	\begin{equation}
		\E\big[\norm{u_t^\eps}_{H^1(\T_\eps^d)}^2\big] = \sum_{k=0}^\infty\eps^d \sum_{x\in \T_\eps^d}\sum_{i=1}^d \Pi_t^{\eps,i,k}(u_0^\eps)(x)\fstop
	\end{equation}
	Hence, the desired claim in \eqref{eq:ub-H1-norm} follows from 
	\begin{equation}\label{eq:estimate-Pi-k}
		\eps^d \sum_{x\in \T_\eps^d}\sum_{i=1}^d \Pi_t^{\eps,i,k}(u_0^\eps)(x)\le 2^{-k}\norm{u_0^\eps}_{H^1(\T_\eps^d)}^2\comma \qquad k \in \N_0\fstop
	\end{equation}
	We start with the case $k=0$. Recalling \eqref{eq:Pi-eps-0}, we have
	\begin{align}
		\eps^d\sum_{x\in \T_\eps^d}\sum_{i=1}^d \Pi_t^{\eps,i,0}(u_0^\eps)(x)&= \sum_{i=1}^d\eps^d\sum_{x\in \T_\eps^d} \tonde{P_t^\eps(\nabla^{\eps,i} u_0^\eps)(x)}^2\\
		&\le \sum_{i=1}^d \eps^d \sum_{x\in \T_\eps^d} P_t^\eps(\nabla^{\eps,i}u_0^\eps)^2(x)\\
		&=\sum_{i=1}^d \eps^d \sum_{x\in \T_\eps^d} (\nabla^{\eps,i}u_0^\eps(x))^2 = \norm{u_0^\eps}_{H^1(\T_\eps^d)}^2\comma
		\label{eq:estimate-Pi-0}
	\end{align}
	where the second step follows by Jensen inequality, while the third step used the symmetry of $P_t^\eps$ in $L^2(\T_\eps^d)$. This proves \eqref{eq:estimate-Pi-k} for $k=0$. For the general case $k\ge 1$, first recall \eqref{eq:Pi-eps-k}. Then, 
	\begin{align}
		&\eps^d \sum_{x\in \T_\eps^d}\sum_{i=1}^d \Pi_t^{\eps,i,k}(u_0^\eps)(x) 
		\\
		&\qquad= \int_{[0,t]_>^k}\dd s_1\cdots \dd s_k{\sum_{y_0,\ldots, y_{k-1}\in \T_\eps^d}\sum_{j_0,\ldots, j_{k-1}=1}^d \tonde{\prod_{\ell=1}^k q_{s_{\ell-1}-s_\ell}^{\eps,j_{\ell-1},j_\ell}(y_{\ell-1},y_\ell)}}\\
		&\qquad\qquad \times \tonde{\eps^d\sum_{y_k\in \T_\eps^d}\sum_{j_k=1}^d  \Pi_{s_k}^{\eps,j_k,0}(u_0^\eps)(y_k)}\\
		&\qquad = \int_{[0,t]_>^k}\dd s_1\cdots \dd s_k\tonde{ \prod_{\ell=1}^k \cQ_\eps(s_\ell-s_{\ell-1})} \tonde{\eps^d\sum_{y_k\in \T_\eps^d}\sum_{j_k=1}^d  \Pi_{s_k}^{\eps,j_k,0}(u_0^\eps)(y_k)}\\
		&\qquad\le \norm{u_0}_{H^1(\T_\eps^d)}^2\int_{[0,t]_>^k} \dd s_1\cdots \dd s_k\tonde{\prod_{\ell=1}^k\cQ_\eps(s_\ell-s_{\ell-1})}\\
		&\qquad\le \norm{u_0^\eps}_{H^1(\T_\eps^d)}^2 \tonde{\int_0^\infty \cQ_\eps(s)\, \dd s}^k\\
		&\qquad=2^{-k}\norm{u_0^\eps}_{H^1(\T_\eps^d)}^2\comma
		\label{eq:estimate-2-k-H1}
	\end{align}
	where the second step follows from \eqref{eq:B-symmetry}--\eqref{eq:B-invariance}, the third step used the inequality in \eqref{eq:estimate-Pi-0}, the fourth step used the non-negativity of $\cQ_\eps$, whereas the fifth step follows from \eqref{eq:B-12}. This proves \eqref{eq:estimate-Pi-k} for $k\ge 1$, thus, concludes the proof of the lemma.
\end{proof}

The following lemma presents a second bound on the expectation of discrete gradients. 
\begin{lemma}\label{lemma:W11-W1oo}
	For all $u_0^\eps, v_0^\eps\in \R^{\T_\eps^d}$ and $t> 0$,  we have
	\begin{equation}\label{eq:estimate-W11-W1oo}
\purple{\eps^d\sum_{x\in \T_\eps^d}\sum_{j=1}^d\abs{\E\big[\nabla^{\eps,j}u_t^\eps(x)\, \nabla^{\eps,j}v_t^\eps(x)\big]}}
			\le 2\,\big(\max_{i=1,\ldots, d}\norm{\nabla^{\eps,i}u_0^\eps}_{L^\infty(\T_\eps^d)}\big)\bigg(\sum_{i=1}^d\norm{\nabla^{\eps,i}v_0^\eps}_{L^1(\T_\eps^d)}\bigg)\fstop
	\end{equation}
\end{lemma}
\begin{proof}
	Recall \eqref{eq:Pi-eps-0}.	Since
	\begin{equation}
		\abs{\Pi_t^{\eps,i,0}(u_0^\eps,v_0^\eps)(x)}\le \norm{\nabla^{\eps,i}u_0}_{L^\infty(\T_\eps^d)} P_t^\eps\abs{\nabla^{\eps,i}v_0^\eps}(x)\comma
	\end{equation}
	we obtain
	\begin{align}
		\eps^d \sum_{x\in \T_\eps^d}\sum_{i=1}^d \abs{\Pi_t^{\eps,i,0}(u_0^\eps,v_0^\eps)(x)}&
		\le \max_{i=1,\ldots, d}\norm{\nabla^{\eps,i}u_0^\eps}_{L^\infty(\T_\eps^d)}\tonde{\eps^d\sum_{i=1}^d\sum_{x\in \T_\eps^d}
			P_t^\eps \abs{\nabla^{\eps,i}v_0^\eps}(x)}
		\\
		&= \max_{i=1,\ldots, d}\norm{\nabla^{\eps,i}u_0^\eps}_{L^\infty(\T_\eps^d)} \tonde{\sum_{i=1}^d\norm{\nabla^{\eps,i}v_0^\eps}_{L^1(\T_\eps^d)}}\comma
		\label{eq:estimate-W11-W1oo-1}
	\end{align}
	where for the last step we used the invariance of $P_t^\eps$, i.e., $\langle P_t^\eps g\rangle_\eps = \langle g\rangle_\eps$, for all $g\in \T_\eps^d$.
	Finally, recalling \eqref{eq:second-moment-mixed}, we get
	\begin{align}
		&\eps^d	\sum_{x\in \T_\eps^d}\sum_{i=1}^d	\abs{\E\big[\nabla^{\eps,i}u_t^\eps(x)\, \nabla^{\eps,i}v_t^\eps(x)\big]}
		\\
		&\qquad\le \sum_{k=0}^\infty \eps^d\sum_{x\in \T_\eps^d}\sum_{i=1}^d \abs{\Pi_t^{\eps,i,k}(u_0^\eps,v_0^\eps)(x)}\\
		&\qquad\le \max_{i=1,\ldots,k}\norm{\nabla^{\eps,i}u_0^\eps}_{L^\infty(\T_\eps^d)}\tonde{\sum_{i=1}^d \norm{\nabla^{\eps,i}v_0^\eps}_{L^1(\T_\eps^d)}}\sum_{k=0}^\infty  2^{-k}\comma
	\end{align}
	where for the last step we first  used  \eqref{eq:estimate-W11-W1oo-1} and then \eqref{eq:B-12}.		This concludes the proof of the lemma.
\end{proof}

The following estimate turns out to be useful for small times and concentrated initial conditions. 
In what follows, we let $B_\sigma(x)\subset \T^d$ denote the open ball of radius $\sigma\in (0,\frac12)$  around  $x\in \T^d$, and exploit the following classical exit-time estimate for the random walk $(X_t^\eps)_{t\ge  0}$ (see, e.g., \cite[Eq.\ (4.4)]{delloschiavo_portinale_sau_scaling_2021} and references therein): for all $t\in (0,1)$, $x\in \T_\eps^d$ and $\rho \in (0,\frac12)$, 
\begin{equation}\label{eq:exit-time}
	\sum_{\substack{y\in \T_\eps^d\\
			\abs{x-y}>\rho}}p_t^\eps(x,y) \le  \Pr_x^\eps\tonde{\sup_{s\in [0,t]}\abs{X_t^\eps-x}>\rho}\le  c_1\exp\tonde{-\frac{c_2\,\rho}{t^{1/2}\vee \eps}}\comma
\end{equation}
for some constants $c_1,c_2>0$ depending only on $d\ge 1$.
\begin{lemma}\label{lemma:short-time}
	Let the initial condition $u_0^\eps \in  \R^{\T_\eps^d}$ be vanishing outside $B_\sigma(x)$, for some $\sigma\in (0,\frac12)$ and $x\in \T_\eps^d$.	 Then, for all $t\in (0,1)$,  $\rho \in (\sigma+\eps,\frac12)$ and $n\in \N$, we have
	\begin{equation}\label{eq:small-time}
		\sum_{i=1}^d \E\big[\ttnorm{\car_{B_\rho^\complement(x)}\nabla^{\eps,i}u_t^\eps}_{L^2(\T_\eps^d)}^2\big]	\le C\norm{u_0^\eps}_{L^\infty(\T_\eps^d)}^2\eps^{-2}\bigg\{\eps^{-2}\exp\tonde{-\frac{C'(\rho-\sigma)}{n(t^{1/2}\vee \eps)}}+2^{-n}\bigg\}\comma
	\end{equation}
	for some constants $C, C'>0$ depending only on $d\ge 1$.
\end{lemma}
\begin{proof}
	Recall \eqref{eq:second-moment}. Then, for every integer $n\ge 1$,   the left-hand side of \eqref{eq:small-time} reads as
	\begin{align}
		&\sum_{i=1}^d \E\big[\ttnorm{\car_{B_\rho^\complement(x)}\nabla^{\eps,i}u_t^\eps}_{L^2(\T_\eps^d)}^2\big]=\eps^d\sum_{\substack{y\in \T_\eps^d\\
				\abs{x-y}>\rho}}\sum_{i=1}^d \E\big[\tonde{\nabla^{\eps,i}u_t^\eps(y)}^2\big]\\
		&\qquad = \eps^d\sum_{\substack{y\in \T_\eps^d\\
				\abs{x-y}>\rho}}\sum_{i=1}^d \sum_{k=0}^n \Pi_t^{\eps,i,k}(u_0^\eps)(y)+\eps^d\sum_{\substack{y\in \T_\eps^d\\
				\abs{x-y}>\rho}}\sum_{i=1}^d \sum_{k=n+1}^\infty \Pi_t^{\eps,i,k}(u_0^\eps)(y)\\
		&\qquad \le 
		\eps^d\sum_{\substack{y\in \T_\eps^d\\
				\abs{x-y}>\rho}}\sum_{i=1}^d \sum_{k=0}^n \Pi_t^{\eps,i,k}(u_0^\eps)(y) + \sum_{k=n+1}^\infty\bigg(\eps^d\sum_{y\in \T_\eps^d}\sum_{i=1}^d \Pi_t^{\eps,i,k}(u_0^\eps)(y)\bigg)\\
		&\qquad \le 
		\eps^d\sum_{\substack{y\in \T_\eps^d\\
				\abs{x-y}>\rho}}\sum_{i=1}^d \sum_{k=0}^n \Pi_t^{\eps,i,k}(u_0^\eps)(y) +
		2^{-n}\norm{u_0^\eps}_{H^1(\T_\eps^d)}^2	\comma\label{eq:short-time-final}
	\end{align}
	where the last step used  \eqref{eq:estimate-Pi-k} to bound the expression in parenthesis, and $2^{-n}=\sum_{k=n+1}^\infty 2^{-k}$. 	 
	Recall \eqref{eq:Pi-eps-0}. By Jensen inequality,  ${\rm supp}(u_0^\eps)\subset B_\sigma(x)$, and the  crude bound \begin{equation}\max_{j=1,\ldots, d}\ttnorm{\nabla^{\eps,j}u_0^\eps}_{L^\infty(\T_\eps^d)}\le 2\eps^{-1}\ttnorm{u_0^\eps}_{L^\infty(\T_\eps^d)}\comma
	\end{equation} we get, for all $s\in (0,1)$,   $j=1,\ldots, d$, and $y\in \T_\eps^d$, 
	\begin{equation}\label{eq:estimate-0}
		\Pi_s^{\eps,j,0}(u_0^\eps)(y)\le \purple{4}\eps^{-2}\norm{u_0^\eps}_{L^\infty(\T_\eps^d)}^2 \sum_{\substack{z\in \T_\eps^d\\
				\abs{x-z}<\sigma+\eps}} p_s^\eps(y,z)
		\fstop
	\end{equation}
	Thanks to the exit-time estimate \eqref{eq:exit-time}, we have, for $y\in \T_\eps^d$ satisfying $\abs{x-y}>\sigma+\eps$,  
	\begin{equation}\label{eq:estimate-0-1}
		\Pi_s^{\eps,j,0}(u_0^\eps)(y)\le 2\eps^{-2}\norm{u_0^\eps}_{L^\infty(\T_\eps^d)}^2\set{c_1\exp\tonde{-\frac{c_2\tonde{\abs{x-y}-\sigma-\eps}}{s^{1/2}\vee \eps}}}\fstop
	\end{equation}
	As a consequence,
	we obtain, for  $k=0$ and for some $C_0=C_0(d)>0$, 	 
	\begin{align}\label{eq:estimate-0-2}
		\eps^d \sum_{\substack{y\in \T_\eps^d\\
				\abs{x-y}>\rho}}\sum_{i=1}^d \Pi_t^{\eps,i,0}(u_0^\eps)(y)&\le 	C_0\,\eps^{-2}\norm{u_0^\eps}_{L^\infty(\T_\eps^d)}^2 \exp\tonde{-\frac{c_2\tonde{\rho-\sigma-\eps}}{t^{1/2}\vee \eps}}\fstop
	\end{align}
	For $k=1,\ldots, n$, recalling \eqref{eq:Pi-eps-k}, we have
	\begin{align}
		&\eps^d\sum_{\substack{y\in \T_\eps^d\\
				\abs{x-y}>\rho}}\sum_{i=1}^d	\Pi_t^{\eps,i,k}(u_0^\eps)(y)\\
		&=\eps^d\int_{[0,t]_>^k}\dd s_1\cdots \dd s_k\sum_{\substack{y_0,\ldots, y_k\in \T_\eps^d\\
				\abs{y_0-x}>\rho}}\sum_{j_0,\ldots, j_k=1}^d \tonde{\prod_{\ell=1}^kq_{s_{\ell-1}-s_\ell}^{\eps,j_{\ell-1},j_\ell}(y_{\ell-1},y_\ell)}\Pi_{s_k}^{\eps,j_k,0}(u_0^\eps)(y_k)\fstop
	\end{align}
	By splitting the summation over $y_k\in \T_\eps^d$ into $\set{y_k\in \T_\eps^d: \abs{x-y_k}<\sigma+\frac{\rho-\sigma-\eps}{2n}}$ and its complement, the estimates in \eqref{eq:estimate-0} and \eqref{eq:estimate-0-1}  yield
	\begin{align}	
		\label{eq:Pi-eps-k-ub}
			&\qquad\sum_{\substack{y\in \T_\eps^d\\
					\abs{x-y}>\rho}}\sum_{i=1}^d \Pi_t^{\eps,i,k}(u_0^\eps)(y)\le 
			2\eps^{-2}\norm{u_0^\eps}_{L^\infty(\T_\eps^d)}^2\\
			&\quad\qquad\times \tonde{I_t^{\eps,k}(\sigma+\tfrac{\rho-\sigma-\eps}{2n},\rho) + c_1\exp\tonde{-\frac{c_2\tonde{\tfrac{\rho-\sigma-\eps}{2n}-\eps}}{t^{1/2}\vee \eps}} J_t^{\eps,k}(\sigma+\tfrac{\rho-\sigma-\eps}{2n},\rho)}\comma
	\end{align}
	where, for all $0\le \rho_1\le \rho_2\le\frac12$, 
	\begin{align}\label{eq:I}
		I_t^{\eps,k}(\rho_1,\rho_2)&\eqdef \int_{[0,t]_>^k}\dd s_1\cdots \dd s_k\, \sum_{\substack{y_0,\ldots, y_k\in \T_\eps^d\\
				\abs{y_0}\ge \rho_2,\,\abs{y_k}< \rho_1}}\sum_{j_0,\ldots,j_k=1}^d \bigg(\prod_{\ell=1}^k q_{s_{\ell-1}-s_\ell}^{\eps,j_{\ell-1},j_\ell}(y_{\ell-1},y_\ell)\bigg)\comma
			\\
	\qquad	J_t^{\eps,k}(\rho_1,\rho_2)&\eqdef  \int_{[0,t]_>^k}\dd s_1\cdots \dd s_k \sum_{\substack{y_0,\ldots, y_k\in \T_\eps^d\\
				\abs{y_0}\ge \rho_2,\,\abs{y_k}\ge \rho_1}}\sum_{j_0,\ldots, j_k=1}^d \bigg(\prod_{\ell=1}^k q_{s_{\ell-1},s_\ell}^{\eps,j_{\ell-1},j_\ell}(y_{\ell-1},y_\ell)\bigg)\fstop
		\label{eq:J}
	\end{align}
	Note that we used the translation invariance of the transition probabilities of $(X_t^\eps)_{t\ge 0}$ to get rid of the dependence on $x\in \T_\eps^d$. Let us extend the above definitions to the case $k=0$ as follows: 
	\begin{equation}\label{eq:I-J-k0}
		I_t^{\eps,0}(\rho_1,\rho_2)\eqdef 0\comma\qquad J_t^{\eps,0}(\rho_1,\rho_2)\eqdef 1\fstop
	\end{equation}
	We first estimate $J_t^{\eps,k}$ in \eqref{eq:J}.
	By \eqref{eq:B-12}, we readily obtain, for all $0\le\rho_1\le\rho_2\le\frac12$,
	\begin{equation}\label{eq:J-ub}
		J_t^{\eps,k}(\rho_1,\rho_2)\le J_t^{\eps,k}(0,0)\le d\,\eps^{-d}\,2^{-k}\fstop
	\end{equation}
	We now deal with $I_t^{\eps,k}$ in \eqref{eq:I}, for all $k\le n \in \N$. 
	Define
	\begin{equation}\label{eq:lambda}
		\lambda=\lambda(\rho_1,\rho_2,n)\eqdef \frac{\rho_2-\rho_1}{2n}\fstop
	\end{equation}
	Recalling the definitions \eqref{eq:bt-eps} and \eqref{eq:B} and the exit-time estimate \eqref{eq:exit-time}, we have, for all  $t\in (0,1)$ and $0\le\rho_1\le\rho_2\le\frac12$,
	\begin{equation}\label{eq:ub-int-q-rho}
		\max_{i=1,\ldots, d}\sup_{\substack{z\in \T_\eps^d\\
				\abs{x-z}\ge \rho_2}}\int_0^t \sum_{\substack{y\in \T_\eps^d\\
				\abs{x-y}<\rho_1}}\sum_{j=1}^d q_s^{\eps,i,j}(z,y)\, \dd s\le C_1\,\eps^{-2}\,t\,\exp\tonde{-\frac{c_2\tonde{\rho_2-\rho_1}}{t^{1/2}\vee \eps}}\comma
	\end{equation} 	
	for some $C_1=C_1(d)>0$. Analogously, by splitting the summation over $y_{k-1}\in \T_\eps^d$ into $\set{y_{k-1}\in \T_\eps^d:\abs{y_{k-1}}<\rho_1+\lambda}$ and its complement, we obtain
	\begin{align}
		I_t^{\eps,k}(\rho_1,\rho_2) &= \int_{[0,t]_>^{k-1}}\dd s_1\cdots \dd s_{k-1}\sum_{\substack{y_0,\ldots, y_{k-1}\in \T_\eps^d\\
				\abs{y_0}>\rho_2}}\sum_{j_0,\ldots, j_{k-1}=1}^d \tonde{\prod_{\ell=1}^{k-1}q_{s_{\ell-1}-s_\ell}^{\eps,j_{\ell-1},j_\ell}(y_{\ell-1},y_\ell)}\\
		&\qquad \qquad \times \int_0^{s_{k-1}}\sum_{\substack{y_k\in \T_\eps^d\\
				\abs{y_k}<\rho_1}}\sum_{j_k=1}^d q_{s_{k-1}-s_k}^{\eps,j_{k-1},j_k}(y_{k-1},y_k)\, \dd s_k\\
		&\le 2^{-1}\, I_t^{\eps,k-1}(\rho_1+\lambda,\rho_2)+  C_1\,\eps^{-2}\,t\, \exp\tonde{-\frac{c_2\,\lambda}{t^{1/2}\vee \eps}}J_t^{\eps,k-1}(\rho_1+\lambda,\rho_2)\\
		&\le 2^{-1}\, I_t^{\eps,k-1}(\rho_1+\lambda,\rho_2) + C_2\, 2^{-k}\,\eps^{-(d+2)}\,  t\, \exp\tonde{-\frac{c_2\,\lambda}{t^{1/2}\vee \eps}}\comma
	\end{align}
	where for the first inequality we used \eqref{eq:B-12} and \eqref{eq:ub-int-q-rho}, while for the second one we used \eqref{eq:J-ub}. Here,  $C_2=C_2(d)>0$. Finally, noting that $\rho_2-\tonde{\rho_1+n\,\lambda}>\lambda$ and recalling \eqref{eq:I-J-k0}, we may iterate the above inequality, so as to obtain
	\begin{equation}\label{eq:I-ub}
		I_t^{\eps,k}(\rho_1,\rho_2)\le C_2\, k\,2^{-k}\, \eps^{-(d+2)}\, t\, \exp\tonde{-\frac{c_2\tonde{\rho_2-\rho_1}}{2n\tonde{t^{1/2}\vee \eps}}}\comma\qquad \text{for all}\ k\le n\fstop
	\end{equation}

	By combining the estimates in \eqref{eq:estimate-0-2},  \eqref{eq:Pi-eps-k-ub}, \eqref{eq:J-ub} and \eqref{eq:I-ub}, we get, for some $C_3, C_4, C_5>0$ depending only on $d\ge 1$,
	\begin{align}
		&\sum_{k=0}^n\Bigg(	\eps^d\sum_{\substack{y\in \T_\eps^d\\
				\abs{x-y}>\rho}}\sum_{i=1}^d \Pi_t^{\eps,i,k}(u_0^\eps)(y)\Bigg)\\
		&\qquad\le  C_4\,\eps^{-2}\norm{u_0^\eps}_{L^\infty(\T_\eps^d)}^2
		\exp\tonde{-\frac{C_3\tonde{\rho-\sigma-\eps}}{2n\tonde{t^{1/2}\vee \eps}}}
		\tonde{1+
			{\eps^{-2}\, t\sum_{k=1}^n
				k\,2^{-k} + \sum_{k=1}^n 2^{-k}  }}\\
		&\qquad\le C_5\, \eps^{-4}\norm{u_0^\eps}_{L^\infty(\T_\eps^d)}^2 \exp\tonde{-\frac{C_3\,\tonde{\rho-\sigma-\eps}}{2n\tonde{t^{1/2}\vee \eps}}}\comma
	\end{align}
	where for the last step we used $t<1$. Inserting this bound into \eqref{eq:short-time-final}, and estimating $\norm{u_0^\eps}_{H^1(\T_\eps^d)}^2\le C_6\,\eps^{-2}\norm{u_0^\eps}_{L^\infty(\T_\eps^d)}^2$, for some $C_6=C_6(d)>0$, we get the desired result.
\end{proof}

\section{Proof of Theorem \ref{th:LLN}. Variance via Poincaré inequality}\label{sec:variance}
The main goal of this section is to provide a quantitative control of the right-hand side of \eqref{eq:poincare-ineq-Gamma-eps} --- thus, of the variance of $\varGamma_t^\eps(g)$ --- and prove that it vanishes as $\eps \to 0$. We fix $t> 0$,  $g= (g^i)_{i=1,\ldots, d}\in (\cC([0,\infty)\times \T^d))^d$, and the initial conditions $u_0\in \cC(\T^d)$ and $u_0^\eps = u_0\restr{\T_\eps^d}\in \R^{\T_\eps^d}$ all throughout the section. Finally,  $\norm{\emparg}_\infty$ indicates uniform norms, e.g.,  \begin{equation}\label{eq:infty-norm}
	\norm{g}_\infty\eqdef \max_{i=1,\ldots,d}\sup_{t\ge0}\sup_{x\in \T^d}\abs{g^i(t,x)}\qquad \text{and}\qquad  \norm{u_0}_\infty\eqdef \sup_{x\in \T^d}\abs{u_0(x)}\fstop
\end{equation} 	
\begin{theorem}[Variance]\label{th:var}
	Let $\varGamma_t^\eps(g)$ be given as in \eqref{eq:Gamma-eps}. Then, for all $h\in (0,1)$,
	\begin{equation}\label{eq:th-var-estimate}
		\eps^{-(d+2)}\int_0^\infty  \eps^d\sum_{x\in \T_\eps^d}\sum_{i=1}^d \E\big[\big(D_{(s,x,i)}\varGamma_t^\eps(g)\big)^2\big]\, \dd s \le C\norm{g}_\infty^2\norm{u_0}_\infty^4 (\eps^{d/2})^{\tonde{1-h}\frac{d}{d+2}}\comma
	\end{equation}
	for some $C=C(d,h)>0$.  
\end{theorem}
We break the proof of Theorem \ref{th:var} into steps. We start by estimating a conditional version of the expectation on the left-hand side of \eqref{eq:th-var-estimate}. 
\begin{proposition}\label{pr:var}
	We have, $\P$-a.s.,  for all $s\in (0,t)$, $x\in \T_\eps^d$, $i=1,\ldots, d$, and $h\in (0,1)$,
	\begin{equation}
		\E\big[\tonde{D_{(s,x,i)}\varGamma_t^\eps(g)}^2\mid \cF_s^\eps\big]\le C\, \eps^{d+2}\norm{g}_\infty^2\norm{u_0}_\infty^2\tonde{\nabla^{\eps,i}u_s^\eps(x)}^2 (\eps^{d/2})^{\tonde{1-h}\frac{d}{d+2}}\comma
	\end{equation}
	where $C=C(d,h)>0$.
\end{proposition}
Before presenting the  proof of Proposition \ref{pr:var}, we use it to prove Theorem \ref{th:var}.
\begin{proof}[Proof of Theorem \ref{th:var}] The tower property and Proposition \ref{pr:var} yield the following bound on  the left-hand side of \eqref{eq:th-var-estimate}:
	\begin{align}
		&\eps^{-(d+2)}\int_0^\infty  \eps^d\sum_{x\in \T_\eps^d}\sum_{i=1}^d \E\big[\E\big[\big(D_{(s,x,i)}\varGamma_t^\eps(g)\big)^2\mid \cF_s^\eps\big]\big]\, \dd s\\
		&\qquad\le C\norm{g}_\infty^2 \norm{u_0}_\infty^2(\eps^{d/2})^{\tonde{1-h}\frac{d}{d+2}}\int_0^\infty \eps^d \sum_{x\in \T_\eps^d}\sum_{i=1}^d \E\big[\tonde{\nabla^{\eps,i}u_s^\eps(x)}^2\big]\, \dd s\fstop
	\end{align}
	The desired claim now follows by Aldous-Lanoue identity \eqref{eq:AL-identity}, which allows to write the above time integral as  
	\begin{equation}
		2\norm{u_0^\eps}_{L^2(\T_\eps^d)}^2\le 2\norm{u_0^\eps}_{L^\infty(\T_\eps^d)}^2\fstop 	\end{equation}
	This concludes the proof of the theorem. 
\end{proof}
The rest of this section is devoted to the proof of Proposition \ref{pr:var}.

\subsection{Dynamics of the discrepancy}\label{sec:discrepancy}
Let us adopt the following shorthand notation:
for all $s\in (0,t)$, $x\in \T_\eps^d$,  and $i=1,\ldots, d$, 
\begin{align}\label{eq:varLambda}
		&\qquad\varLambda^\eps_{(s,x,i)}\eqdef \tonde{D_{(s,x,i)}\varGamma_t^\eps(g)}^2\\
		& = \bigg(\int_s^t \eps^d \sum_{y\in \T_\eps^d}\sum_{j=1}^d \set{\tonde{\nabla^{\eps,j}u_r^\eps(y)+\nabla^{\eps,j}w_r^{\eps,(s,x,i)}(y)}^2-\tonde{\nabla^{\eps,j}u_r^\eps(y)}^2}g_r^j(y)\,\dd r\bigg)^2\comma
\end{align}
where we recall from \eqref{eq:u+w} that $(w_r^{\eps,(s,x,i)})_{r\ge 0}$ denotes the discrepancy created by the extra update at time $s$ at the vertices $x$ and $x+\eps e_i$. From now on, we shall drop  $(s,x,i)$ from the notation, and simply write, e.g., \begin{equation}\varLambda^\eps=\varLambda_{(s,x,i)}^\eps\qquad \text{and}\qquad w_r^\eps = w_r^{\eps,(s,x,i)}\fstop
\end{equation}

Next, we split the expression in \eqref{eq:varLambda} into three terms: for all $\delta \in (\eps^2,1)$,
\begin{align}
	\varLambda^\eps&= 	 \bigg(\int_s^t \eps^d \sum_{y\in \T_\eps^d}\sum_{j=1}^d\set{\tonde{\nabla^{\eps,j}w_r^{\eps}(y)}^2+2\,\nabla^{\eps,j}w_r^{\eps}(y)\, \nabla^{\eps,j}u_r^\eps(y)}g_r^j(y)\,\dd r\bigg)^2\\
	&\qquad\le 2\,\bigg(\int_s^t \eps^d\sum_{y\in \T_\eps^d}\sum_{j=1}^d\tonde{\nabla^{\eps,j}w_r^\eps(y)}^2 g_r^j(y)\, \dd r\bigg)^2\\
	&\qquad\qquad + 16\,\bigg(\int_s^{(s+\delta)\wedge t}\eps^d \sum_{y\in \T_\eps^d}\sum_{j=1}^d\nabla^{\eps,j}w_r^\eps(y)\, \nabla^{\eps,j}u_r^\eps(y)\, g_r^j(y)\, \dd r\bigg)^2\\
	&\qquad\qquad + 16\,\bigg(\int_{(s+\delta)\wedge t}^t\eps^d \sum_{y\in \T_\eps^d}\sum_{j=1}^d\nabla^{\eps,j}w_r^\eps(y)\, \nabla^{\eps,j}u_r^\eps(y)\, g_r^j(y)\, \dd r\bigg)^2\\
	&\qquad\defeq 2\,\varLambda_1^\eps + 16\,\varLambda_2^{\eps,\delta} + 16\,\varLambda_3^{\eps,\delta}\comma
	\label{eq:varLambda-1-2-3}
\end{align}
where we used twice the elementary inequality $(a+b)^2\le 2a^2+2b^2$, $a,b\in \R$. Our task is to  analyze the dynamics of the discrepancy $(w_r^\eps)_{r\ge s}$, bounding the conditional expectation of each of these three terms. For this purpose, let us observe that, while $w_r^\eps\equiv 0$ for $r<s$, for $r=s$ we have
\begin{equation}\label{eq:w-def}
	w_s^\eps = \frac{\eps}2\,
	\nabla^{\eps,i}u_s^\eps(x)\tonde{\car_x-\car_{x+\eps e_i}}\comma
\end{equation}
from which we obtain
\begin{equation}\label{eq:w-l2-norm}
	\norm{w_s^\eps}_{L^2(\T_\eps^d)}^2= \eps^d \sum_{y\in \T_\eps^d}\tonde{w_s^\eps(y)}^2= \frac12\,\eps^{d+2}\tonde{\nabla^{\eps,i}u_s^\eps(x)}^2\fstop
\end{equation}

\subsection{Proof strategy}\label{sec:proof-overview-lemmas}
We estimate the  three terms in \eqref{eq:varLambda-1-2-3} in three separate lemmas in the subsequent section. As an overview of the proof strategy, we observe that all three lemmas fundamentally exploit the time integrals. In this way, we avoid to estimate fourth moments of discrete gradients, for which we are not able to recover convergent recursive inequalities, as done for the second moments.

In this spirit, the first term $\varLambda_1^\eps$ is handled rather easily by squaring the time integral, applying Cauchy-Schwarz, and gain\purple{ing} $L^2$-norms of $w_s^\eps$ from the time integration and Aldous-Lanoue identity \eqref{eq:AL-identity}. The resulting upper bound depends on the discrepancy only via $\norm{w_s^\eps}_{L^2(\T_\eps^d)}^4\approx \eps^{2\tonde{d+2}}\ll\eps^{d+2}$ (see \eqref{eq:w-l2-norm}), which therefore suffices.
The third term $\varLambda_3^{\eps,\delta}$, although it contains both $w_r^\eps$ and $u_r^\eps$, is dealt with analogously. Nevertheless, if we were to take $\delta=0$ (as in $\varLambda_1^\eps$), this strategy would yield $\norm{w_s^\eps}_{L^2(\T_\eps^d)}\norm{u_s^\eps}_{L^2(\T_\eps^d)}^2\approx \eps^{d+2}$, which is too poor for our purposes. This explains the necessity of introducing a small burn-in time of size $\eps^2\ll \delta\ll \eps$, producing a regularization effect of the Dirac-like discrepancy $w_s^\eps$. This smoothening is quantified in terms of the ultracontractivity of the averaging process \eqref{eq:ultracontractivity}.

The second term $\varLambda_2^{\eps,\delta}$ concerns the remaining part of the time integral left over by $\varLambda_3^{\eps,\delta}$, namely, from time $r=s$ to $r=s+\delta$. This term is the most delicate. Instead of relying on Cauchy-Schwarz inequality, this time we must leverage  the fact that, for small times and when starting from highly concentrated data, $L^1$-norms display a better decay than $L^2$-norms. In order to turn this observation into effective bounds, we crucially employ: (i) a localization argument (Lemma \ref{lemma:short-time}, with $\delta^{1/2}\ll \rho\ll 1$); (ii) pseudo-contractivity estimates for discrete gradients (Lemmas \ref{lemma:H1-contractivity} and \ref{lemma:W11-W1oo}).

\subsection{Proofs} We split the proof  of Proposition \ref{pr:var} into three lemmas, one for each term in \eqref{eq:varLambda-1-2-3}.
We start with the first term in \eqref{eq:varLambda-1-2-3}. Here, $C_1=C_1(d)>0$.
\begin{lemma}[Estimate of $\varLambda_1^\eps$]\label{lemma:varLambda-1} We have
	\begin{equation}\E\big[\varLambda_1^\eps\mid \cF_s^\eps\big]\le C_1\, \eps^{2d+2}\norm{g}_\infty^2\norm{u_0^\eps}_{L^\infty(\T_\eps^d)}^2\tonde{\nabla^{\eps,i}u_s^\eps(x)}^2 \fstop
	\end{equation}
\end{lemma}
\begin{proof}Since
	\begin{equation}
		\varLambda_1^\eps\le 2\norm{g}_\infty^2\int_s^t\norm{w_r^\eps}_{H^1(\T_\eps^d)}^2\int_r^t \norm{w_{r'}^\eps}_{H^1(\T_\eps^d)}^2\, \dd r'\dd r\comma
	\end{equation}
	the tower property and Aldous-Lanoue identity \eqref{eq:AL-identity} \purple{(for the first and fourth steps)} yield
	\begin{align}
		\E\big[\varLambda_1^\eps\mid \cF_s^\eps\big] &\le 2\norm{g}_\infty^2\int_s^t \E\bigg[\norm{w_r^\eps}_{H^1(\T_\eps^d)}^2 \int_r^t -2\,\frac{\dd}{\dd r'}\, \E\quadre{\norm{w_{r'}^\eps}_{L^2(\T_\eps^d)}^2\mid \cF_r^\eps}\dd r'\,\bigg|\,  \cF_s^\eps\bigg]\,\dd r\\
		&\le 4\norm{g}_\infty^2\int_s^t\E\big[\norm{w_r^\eps}_{H^1(\T_\eps^d)}^2\norm{w_r^\eps}_{L^2(\T_\eps^d)}^2\mid \cF_s^\eps\big]\,\dd r\\
		&\le 4 \norm{g}_\infty^2\norm{w_s^\eps}_{L^2(\T_\eps^d)}^2\int_s^t \E\big[\norm{w_r^\eps}_{H^1(\T_\eps^d)}^2\mid \cF_s^\eps\big]\, \dd r\\
		&\le 8\norm{g}_\infty^2 \norm{w_s^\eps}_{L^2(\T_\eps^d)}^4\comma
	\end{align}
	where the  third step used that $r\mapsto \norm{w_r^\eps}_{L^2(\T_\eps^d)}^2$ is deterministically non-increasing for  $r\ge s$. Finally, we obtain the desired inequality by \eqref{eq:w-l2-norm}, which yields
	\begin{align}
		\norm{w_s^\eps}_{L^2(\T_\eps^d)}^4&=2^{-2}\, \eps^{2(d+2)}\tonde{\nabla^{\eps,i}u_s^\eps(x)}^4\\
		&\le \eps^{2d+2}\norm{u_s^\eps}_{L^\infty(\T_\eps^d)}^2\tonde{\nabla^{\eps,i} u_s^\eps(x)}^2\\
		&\le  \eps^{2d+2}
		\norm{u_0^\eps}_{L^\infty(\T_\eps^d)}^2\tonde{\nabla^{\eps,i}u_s^\eps(x)}^2\comma
	\end{align}
	where the last step  follows by the $\P$-a.s.\ monotonicity of $L^\infty(\T_\eps^d)$-norms for the averaging process.
\end{proof}

The estimate of the second term in \eqref{eq:varLambda-1-2-3} is the most involved one. Here, $C_2, C_2'>0$ depend only on $d\ge 1$.
\begin{lemma}[Estimate of $\varLambda_2^{\eps}$]\label{lemma:varLambda-2}
	We have, for all $\delta\in (\eps^2,1)$,	 $\rho \in (4\eps,\frac12)$ and  $n\in \N$, \begin{equation}\E\big[\varLambda_2^{\eps,\delta}\mid \cF_s^\eps\big]\le 
		C_2\, T^{\eps,\delta,\rho,n} \norm{g}_\infty^2\norm{u_0^\eps}_{L^\infty(\T_\eps^d)}^2  \tonde{\nabla^{\eps,i}u_s^\eps(x)}^2\comma
	\end{equation}
	where
	\begin{equation}
		T^{\eps,\delta,\rho,n}\eqdef	\eps^{-2}\delta^2\set{\tonde{\rho \eps}^d +\eps^{-2}\exp\tonde{-C_2'\frac{\rho}{n\,\delta^{1/2}}}+2^{-n}}\fstop
	\end{equation}
	
\end{lemma}
\begin{proof}
	By the tower property and  
	\begin{equation}
		\max_{j=1,\ldots, d}\ttnorm{\nabla^{\eps,j}u_r^\eps}_{L^\infty(\T_\eps^d)}\le 2\eps^{-1}\norm{u_r^\eps}_{L^\infty(\T_\eps^d)}\le 2\eps^{-1}\norm{u_0^\eps}_{L^\infty(\T_\eps^d)}\comma\qquad r\ge 0\comma\end{equation}
	we get
	\begin{equation}\label{eq:varLambda-2-1}
		\E\big[\varLambda_2^{\eps,\delta}\mid \cF_s^\eps\big]\le 4\eps^{-1}\norm{g}_\infty^2\norm{u_0^\eps}_{L^\infty(\T_\eps^d)}\int_s^{s+\delta}\sum_{j=1}^d\E\big[\norm{\nabla^{\eps,j}w_r^\eps}_{L^1(\T_\eps^d)}Z_r^{\eps,\delta}\, \big|\,\cF_s^\eps\big]\,\dd r\comma
	\end{equation}
	where
	\begin{equation}\label{eq:Z}
		Z_r^{\eps,\delta}\eqdef	\int_r^{s+\delta}
		\eps^d\sum_{y\in \T_\eps^d}\sum_{j=1}^d\abs{\E\big[\nabla^{\eps,j}w_{r'}^\eps(y)\, \nabla^{\eps,j}u_{r'}^\eps(y)\mid \cF_r^\eps\big]} \dd r'\comma\qquad r \in (s,s+\delta)\fstop
	\end{equation}
	We apply Lemma \ref{lemma:W11-W1oo} to the integrand in \eqref{eq:Z}, so as to obtain
	\begin{align}
		Z_r^{\eps,\delta}&\le 2\delta\,\big(\max_{j=1,\ldots,d}\norm{\nabla^{\eps,j}u_r^\eps}_{L^\infty(\T_\eps^d)}\big) \bigg(\sum_{j=1}^d\norm{\nabla^{\eps,j}w_r^\eps}_{L^1(\T_\eps^d)}\bigg)\\
		&\le 4\delta\, \big(\eps^{-1}\norm{u_0^\eps}_{L^\infty(\T_\eps^d)}\big)
		\bigg(\sum_{j=1}^d\norm{\nabla^{\eps,j}w_r^\eps}_{L^1(\T_\eps^d)}\bigg)\fstop
	\end{align} 
	By combining this estimate with \eqref{eq:varLambda-2-1}, we get
	\begin{equation}\label{eq:varLambda-2-ub1}
		\begin{aligned}
		&\E\big[\varLambda_2^{\eps,\delta}\mid \cF_s^\eps\big]\\
		&\qquad\le 16\,\eps^{-2}\delta \norm{g}_\infty^2\norm{u_0^\eps}_{L^\infty(\T_\eps^d)}^2\int_s^{s+\delta}\E\bigg[\bigg(\sum_{j=1}^d \norm{\nabla^{\eps,j}w_r^\eps}_{L^1(\T_\eps^d)}\bigg)^2\bigg|\,\cF_s^\eps\bigg]\,\dd r\fstop
		\end{aligned}
	\end{equation}
	We now split the $L^1$-norm above as follows: for any $\rho\in (4\eps,\frac12)$ and $j=1,\ldots, d$,
	\begin{align}
		\norm{\nabla^{\eps,j}w_r^\eps}_{L^1(\T_\eps^d)} &= \eps^d \sum_{\substack{y\in \T_\eps^d\\
				\abs{x-y}\le \rho}}\abs{\nabla^{\eps,j}w_r^\eps(y)} + \eps^d\sum_{\substack{y\in \T_\eps^d\\
				\abs{x-y}>\rho}}\abs{\nabla^{\eps,j}w_r^\eps(y)}\\
		&= \ttnorm{\car_{\bar B_\rho(x)}\nabla^{\eps,j}w_r^\eps }_{L^1(\T_\eps^d)} + \ttnorm{\car_{B_\rho^\complement(x)} \nabla^{\eps,j}w_r^\eps}_{L^1(\T_\eps^d)}\comma
	\end{align}
	from which we obtain, by  Cauchy-Schwarz inequality, 
	\begin{align}
		\bigg(\sum_{j=1}^d \norm{\nabla^{\eps,j}w_r^\eps}_{L^1(\T_\eps^d)}\bigg)^2
		&\le C_1' \sum_{j=1}^d \set{\ttnorm{\car_{\bar B_\rho(x)}\nabla^{\eps,j}w_r^\eps }_{L^1(\T_\eps^d)}^2+\ttnorm{\car_{B_\rho^\complement(x)} \nabla^{\eps,j}w_r^\eps}_{L^1(\T_\eps^d)}^2}\\
		& \le C_1'\sum_{j=1}^d \set{\ttnorm{\car_{\bar B_\rho(x)}}_{L^2(\T_\eps^d)}^2\ttnorm{\nabla^{\eps,j}w_r^\eps}_{L^2(\T_\eps^d)}^2 + \ttnorm{\car_{B_\rho^c(x)} \nabla^{\eps,j}w_r^\eps}_{L^2(\T_\eps^d)}^2}
		\\
		& \le C_2'\sum_{j=1}^d\set{\rho^d\, \ttnorm{\nabla^{\eps,j}w_r^\eps}_{L^2(\T_\eps^d)}^2 + \ttnorm{\car_{B_\rho^c(x)} \nabla^{\eps,j}w_r^\eps}_{L^2(\T_\eps^d)}^2}\\
		& = C_2'\bigg\{\rho^d \norm{w_r^\eps}_{H^1(\T_\eps^d)}^2 + \eps^d\sum_{\substack{y\in \T_\eps^d\\
				\abs{x-y}>\rho}}\sum_{j=1}^d \tonde{\nabla^{\eps,j}w_r^\eps(y)}^2\bigg\}
		\comma
	\end{align}
	for some constants  $C_1',C_2'>0$ depending only on $d\ge 1$. Taking expectations, the above estimate yields
	\begin{align}
		&\E\bigg[\bigg(\sum_{j=1}^d \norm{\nabla^{\eps,j}w_r^\eps}_{L^1(\T_\eps^d)}\bigg)^2\bigg|\,\cF_s^\eps\bigg]\\
		&\qquad\le C_2'\bigg\{\rho^d\, \E\big[\norm{w_r^\eps}_{H^1(\T_\eps^d)}^2\mid \cF_s^\eps\big] + \eps^d \sum_{\substack{y\in \T_\eps^d\\
				\abs{x-y}>\rho} }\sum_{j=1}^d	\E\big[\tonde{\nabla^{\eps,j}w_r^\eps(y)}^2\mid \cF_s^\eps\big]\bigg\}\\
		&\qquad \le C_2'\bigg\{2\rho^d \norm{w_s^\eps}_{H^1(\T_\eps^d)}^2 +  \eps^d \sum_{\substack{y\in \T_\eps^d\\
				\abs{x-y}>\rho} }\sum_{j=1}^d	\E\big[\tonde{\nabla^{\eps,j}w_r^\eps(y)}^2\mid \cF_s^\eps\big]\bigg\}\\
		&\qquad =C_2'\bigg\{2\rho^d\,c\,\eps^d\tonde{\nabla^{\eps,i}u_s^\eps(x)}^2
		+  \eps^d \sum_{\substack{y\in \T_\eps^d\\
				\abs{x-y}>\rho} }\sum_{j=1}^d	\E\big[\tonde{\nabla^{\eps,j}w_r^\eps(y)}^2\mid \cF_s^\eps\big]
		\bigg\}\comma
		\label{eq:varLambda-2-ub2}
	\end{align}
	where the second step follows from Lemma \ref{lemma:H1-contractivity}, while the third step used \eqref{eq:w-def}, which ensures that, for some $c=c(d)>0$, 
	\begin{equation}\label{eq:w-H1-norm}
		\norm{w_s^\eps}_{H^1(\T_\eps^d)}^2=\eps^d\sum_{y\in \T_\eps^d}\sum_{j=1}^d\tonde{\nabla^{\eps,j}w_s^\eps(y)}^2 = c\,\eps^d\tonde{\nabla^{\eps,i}u_s^\eps(x)}^2\fstop
	\end{equation}
	Observe that $w_s^\eps$ is non-zero only in a $2\eps$-neighborhood of $x\in \T_\eps^d$. Moreover, by \eqref{eq:w-def},
	\begin{equation}\norm{w_s^\eps}_{L^\infty(\T_\eps^d)}^2 =\frac{\eps^2}2\tonde{\nabla^{\eps,i}u_s^\eps(x)}^2\fstop
	\end{equation}
	Henceforth, Lemma \ref{lemma:short-time} applied to $(w_r^\eps)_{r\ge s}$ yields, for all $n\in \N$ and $r\in (s,s+\delta)$, 
	\begin{equation}\label{eq:varLambda-2-ub3}
		\sum_{j=1}^d \E\big[\ttnorm{\car_{B_\rho^\complement(x)}\nabla^{\eps,j}w_r^\eps}_{L^2(\T_\eps^d)}^2\mid \cF_s^\eps\big]	
		\le C\tonde{\nabla^{\eps,i}u_s^\eps(x)}^2 \set{\eps^{-2}\exp\tonde{-\frac{C'\, \rho}{n\,\delta^{1/2}}}+2^{-n}}\fstop
	\end{equation}
	
	By combining \eqref{eq:varLambda-2-ub1}, \eqref{eq:varLambda-2-ub2} and \eqref{eq:varLambda-2-ub3}, we get the desired estimate.
\end{proof}

We now bound the third term in \eqref{eq:varLambda-1-2-3}. Here, $C_3=C_3(d)>0$.
\begin{lemma}[Estimate of $\varLambda_3^{\eps,\delta}$]\label{lemma:varLambda-3}
	We have, for all $\delta\in (\eps^2,1)$, 	
	\begin{equation}
		\E\big[\varLambda_3^{\eps,\delta}\mid \cF_s^\eps\big]\le  C_3\,\eps^{2d+2}\delta^{-d/2}\norm{g}_\infty^2\norm{u_0^\eps}_{L^\infty(\T_\eps^d)}^2\tonde{\nabla^{\eps,i}u_s^\eps(x)}^2\fstop
	\end{equation}
	
\end{lemma}
\begin{proof}
	Since $\varLambda_3^{\eps,\delta}=0$ whenever $s+\delta \ge t$, let us assume  $s+\delta <t$. Then, by repeatedly applying Cauchy-Schwarz inequality, we get
	\begin{align}
		&\E\big[\varLambda_3^{\eps,\delta}\mid \cF_{s+\delta}^\eps\big]\\
		& \le \norm{g}_\infty^2\E\bigg[\bigg(\int_{s+\delta}^t\norm{w_r^\eps}_{H^1(\T_\eps^d)}\norm{u_r^\eps}_{H^1(\T_\eps^d)}\dd r\bigg)^2\bigg|\, \cF_{s+\delta}^\eps\bigg]\\
		&\le \norm{g}_\infty^2 \E\bigg[\bigg(\int_{s+\delta}^t \norm{w_r^\eps}_{H^1(\T_\eps^d)}^2\dd r\bigg)\bigg(\int_{s+\delta}^t\norm{u_r^\eps}_{H^1(\T_\eps^d)}^2\dd r\bigg)\bigg|\, \cF_{s+\delta}^\eps\bigg]\\
		&\le \norm{g}_\infty^2\E\bigg[\bigg(\int_{s+\delta}^t\norm{w_r^\eps}_{H^1(\T_\eps^d)}^2\bigg)^2\bigg|\,\cF_{s+\delta}^\eps\bigg]^\frac12\E\bigg[\bigg(\int_{s+\delta}^t\norm{u_r^\eps}_{H^1(\T_\eps^d)}^2\dd r\bigg)^2\bigg|\, \cF_{s+\delta}^\eps\bigg]^{\purple{\frac12}}\fstop
	\end{align}
	By estimating both expectations as already done in the proof of Lemma \ref{lemma:varLambda-1}, we  get
	\begin{align}
		\E\big[\varLambda_3^{\eps,\delta}\mid \cF_{s+\delta}^\eps\big]&\le 8\norm{g}_\infty^2 \norm{w_{s+\delta}^\eps}_{L^2(\T_\eps^d)}^2\norm{u_{s+\delta}^\eps}_{L^2(\T_\eps^d)}^2\\
		&\le 8\norm{g}_\infty^2 \norm{w_{s+\delta}^\eps}_{L^2(\T_\eps^d)}^2\norm{u_0^\eps}_{L^\infty(\T_\eps^d)}^2\fstop
	\end{align}
	By the tower property, we obtain
	\begin{equation}
		\E\big[\varLambda_3^{\eps,\delta}\mid \cF_s^\eps\big] = \E\big[\E\big[\varLambda_3^{\eps,\delta}\mid \cF_{s+\delta}^\eps\big]\mid \cF_s^\eps\big]\le 8\norm{g}_\infty^2\norm{u_0^\eps}_{L^\infty(\T_\eps^d)}^2\E\big[\norm{w_{s+\delta}^\eps}_{L^2(\T_\eps^d)}^2\mid \cF_s^\eps\big]\fstop
	\end{equation}
	Since $\langle w_s^\eps\rangle_\eps=0$ (cf.\ \eqref{eq:w-def}), the last expectation may be further estimated thanks to \eqref{eq:ultracontractivity}:
	\begin{equation}
		\E\big[\norm{w_{s+\delta}^\eps}_{L^2(\T_\eps^d)}^2\mid \cF_s^\eps\big]\le C\norm{w_s^\eps}_{L^1(\T_\eps^d)}^2\delta^{-d/2}\fstop
	\end{equation}
	Recalling \eqref{eq:w-def}, we have \begin{equation}\norm{w_s^\eps}_{L^1(\T_\eps^d)}^2= \eps^{2d+2}\tonde{\nabla^{\eps,i}u_s^\eps(x)}^2\comma
	\end{equation}
	and, thus, the desired result.	
\end{proof}

We conclude this section with the proof of Proposition \ref{pr:var}.
\begin{proof}[Proof of Proposition \ref{pr:var}]
	We  exhibit a choice of $\delta=\delta(\eps)\in (\eps^2,1)$ yielding
	\begin{equation}\label{eq:inequality-ell-1-2-3}
		\E\big[\varLambda_\ell^{\eps,\delta}\mid \cF_s^\eps\big]\le C\, \eps^{d+2}\norm{g}_\infty^2\norm{u_0^\eps}_{L^\infty(\T_\eps^d)}^2\tonde{\nabla^{\eps,i}u_s^\eps(x)}^2\, V_\ell^\eps\comma\qquad \ell \in \set{1,2,3}\comma
	\end{equation}
	for some  $V_\ell^\eps\le C' (\eps^{d/2})^{\tonde{1-h}\frac{d}{d+2}}$, for all $h\in (0,1)$ and some $C'=C'(d,h)>0$. The inequality in \eqref{eq:varLambda-1-2-3} would conclude the proof of the proposition.

	For what concerns $\ell=1$, we have $\varLambda_1^{\eps,\delta}=\varLambda_1^\eps$ and, by Lemma \ref{lemma:varLambda-1}, we obtain \eqref{eq:inequality-ell-1-2-3} with $V_1^\eps=\eps^{-(d+2)}\eps^{2d+2}=\eps^d$.	
	Consider the case $\ell=3$. Letting $\delta=\eps^{2\tonde{1-a}}$, $a\in (0,1)$, Lemma \ref{lemma:varLambda-3} yields \eqref{eq:inequality-ell-1-2-3} with
	\begin{equation}
		V_3^\eps=\eps^{-(d+2)}\ttonde{\eps^{2d+2}\delta^{-d/2}}= \eps^{-d-2+2d+2-d+ad}= \eps^{ad}\fstop
	\end{equation}
	By setting $a=\tonde{1-h}\frac12\frac{d}{d+2}$ for some small $h\in (0,1)$, we obtain the desired claim for $\ell=3$.
	Consider $\ell=2$, and let $\rho=\eps^{1-b}$, for $b=\frac12\frac{d}{d+2}\in (a,1)$, as well as $n=\lceil| \log\eps|^2\rceil$. Then, by inserting these choices into the claim of Lemma \ref{lemma:varLambda-2}, we obtain  
	\begin{align}
		V_2^\eps &= \eps^{-(d+2)}\, \eps^{-2+4-4a}\set{\eps^{2d-bd}+\eps^{-2}\exp\tonde{-C_2'\frac{\eps^{-\tonde{b-a}}}{\log^2\abs{\eps}}}+2^{-\log^2\abs{\eps}}}\\
		&= \eps^{-4a+d-bd} + \eps^{-d-2-4a}\exp\tonde{-C_2'\frac{\eps^{-\tonde{b-a}}}{\lceil|\log\eps|^2\rceil}} + 2^{-\lceil\purple{|}\log\eps|^2\rceil}\fstop
		\label{eq:last}
	\end{align}
	Since $b=\frac12\frac{d}{d+2}$ and $a=\tonde{1-h}b$, the first term in \eqref{eq:last} is smaller than $(\eps^{d/2})^{\frac{d}{d+2}}\le V_3^\eps$.	
	Observe that the choices $b>a$ and $n= \lceil\purple{|}\log\eps|^2\rceil$ ensure  that  the second and third terms in \eqref{eq:last} are bounded above, uniformly over $\eps\in (0,1)$,  by, e.g.,   $C''V_1^\eps=C'' \eps^d$,  for some $C''=C''(d,h)>0$. This proves \eqref{eq:inequality-ell-1-2-3} for $\ell=2$ and, thus, concludes the proof.
\end{proof}

\section{Proof of Theorem \ref{th:LLN}. Mean}\label{sec:mean}
In view of Theorem \ref{th:var} on the variance of $\varGamma_t^\eps(g)$, the proof of Theorem \ref{th:LLN} is complete as soon as we show   convergence of the corresponding means. This fact is summarized in the following theorem. Also in this section, we fix a test function $g=(g^i)_{i=1,\ldots, d}\in (\cC([0,\infty)\times \T^d))^d$, as well as the initial condition $u_0\in \cC^2(\T^d)$, and consider $u_0^\eps\eqdef u_0\restr{\T_\eps^d}$. 
\begin{theorem}[Mean]\label{th:mean}
	Let $\varGamma_t^\eps(g)$ and $\varGamma_t(g)$ be given as in \eqref{eq:Gamma-eps} and \eqref{eq:Gamma}, respectively. Then, 	 for all $t> 0$, we have
	$
		\E[\varGamma_t^\eps(g)] \xrightarrow{\eps\to 0}\varGamma_t(g)$.
	
\end{theorem}

We break the proof of Theorem \ref{th:mean} into  two main steps: Propositions \ref{pr:mean-2} and  \ref{pr:mean-1}. We remark that, while the proof of Theorem \ref{th:var} was crucially exploiting the time integration in the definition of $\varGamma_t^\eps(g)$, this time we establish a pointwise convergence in both time and space variables. Here, the series representation in terms of $\Pi_t^{\eps,i,k}(u_0^\eps)$ for expectations of  squared gradients (Proposition \ref{pr:second-moments}) plays a prominent role. For such expressions, we establish two claims. First, we introduce an approximation $\tilde \Pi_t^{\eps,i,k}(u_0^\eps)$ of $\Pi_t^{\eps,i,k}(u_0^\eps)$ (Proposition \ref{pr:mean-2}). Then, we prove a limit theorem for these approximations (Proposition \ref{pr:mean-1}). 
\begin{proposition}\label{pr:mean-2} Recall  \eqref{eq:Pi-eps-k}. Then, for all  $t>0$ and $i=1,\ldots, d$, we have 
	\begin{equation}\label{eq:pr-mean-2}
		\sum_{k=0}^\infty \abs{ \Pi_t^{\eps,i,k}(u_0^\eps)(x)-\tilde \Pi_t^{\eps,i,k}(u_0^\eps)(x)}\xrightarrow{\eps \to 0}0\comma
	\end{equation}
	uniformly over $x\in \T_\eps^d$.
	Here, $\tilde \Pi_t^{\eps,i,k}\eqdef \Pi_t^{\eps,i,k}$ for $k=0$, while, for $k\ge1$, 
	\begin{equation}\label{eq:tilde-Pi-eps-k}
		\begin{aligned}
		&\tilde \Pi_t^{\eps,i,k}(u_0^\eps)(x)\eqdef \sum_{j=1}^d \car_{j_k=j}\, \Pi_t^{\eps,j,0}(u_0^\eps)(x)\\
		&\qquad \times  \int_{[0,t]_>^k}\dd s_1\cdots \dd s_k\sum_{y_1,\ldots, y_k\in \T_\eps^d}\sum_{j_1,\ldots, j_k=1}^d \tonde{\prod_{\ell=1}^k q_{s_{\ell-1}-s_\ell}^{\eps,j_{\ell-1},j_\ell}(y_{\ell-1},y_\ell)} \fstop
		\end{aligned}
	\end{equation}
\end{proposition}
\begin{proposition}\label{pr:mean-1}Recall \eqref{eq:tilde-Pi-eps-k}. Then, there exists $\mathfrak a=\mathfrak a(d)\in (0,1]$ such that, 
	for all $t> 0$ and $i=1,\ldots,d$,  	we have
	\begin{equation}\label{eq:pr-mean-1}
		{\sum_{k=0}^\infty \tilde \Pi_t^{\eps,i,k}(u_0^\eps)(x)}\xrightarrow{\eps\to 0} \sum_{j=1}^d \tonde{\tonde{1+\mathfrak a}\car_{i=j}+\frac{1-\mathfrak a}{d-1}\,\car_{i\neq j}}\tonde{\nabla^j u_t(x)}^2\comma
	\end{equation}
	uniformly over $x\in \T_\eps^d$. 
\end{proposition}

We present the proofs of these two propositions in the subsequent section. Before that, we conclude this part by proving Theorem \ref{th:mean}.
\begin{proof}[Proof of Theorem \ref{th:mean}]
	From \eqref{eq:Gamma-eps} and Proposition \ref{pr:second-moments}, we have
	\begin{align}
		\E\big[\varGamma_t^\eps(g)\big]&= \sum_{i=1}^d \int_0^t \eps^d \sum_{x\in \T_\eps^d} \E\big[\tonde{\nabla^{\eps,i}u_s^\eps(x)}^2\big]g_s^i(x)\, \dd s\\
		&=\sum_{i=1}^d\int_0^t  \eps^d \sum_{x\in \T_\eps^d} \set{\sum_{k=0}^\infty \Pi_s^{\eps,i,k}(u_0^\eps)(x)}g_s^i(x)\, \dd s\fstop
	\end{align}
	By the uniform boundedness of $g$ and the estimate in Lemma \ref{lemma:H1-contractivity}, we may apply the dominated convergence theorem, so that the convergence in Theorem \ref{th:mean} reduces to that of
	the term in curly brackets above, for all $s>0$ and uniformly over $x\in \T_\eps^d$. Recalling the expression of $\varGamma_t(g)$ from \eqref{eq:Gamma}, this is precisely the content of Propositions \ref{pr:mean-2} and \ref{pr:mean-1}.	
\end{proof} 
\subsection{Proofs}\label{sec:mean-proofs} We start with two lemmas on some additional properties of the functions $q^\eps$ defined in \eqref{eq:B}. The first of this lemma is a refinement of Lemma \ref{lemma:Q1} in the multidimensional case $d\ge 2$. We defer the proof of the following lemma to \hyperref[sec:app]{Appendix}.
\begin{lemma}\label{lemma:Q2} Fix $d\ge 2$. Then, for all $x\in \T_\eps^d$, $t\ge 0$, and $i,j=1,\ldots, d$, the quantity
	\begin{equation}\label{eq:Q-i-j}
		\cQ_\eps^{i,j}(t)\eqdef	\sum_{y\in \T_\eps^d}q_t^{\eps,i,j}(x,y)
	\end{equation}
	depends only on $d\ge 2$, $t\ge  0$ and $\car_{i=j}$. Moreover, letting
	\begin{equation}\label{eq:b-c-eps}
		\int_0^\infty \cQ_\eps^{i,j}(t)\, \dd t \defeq \begin{dcases}
			\mathfrak 	b_\eps &\text{if}\ i=j\\
			\mathfrak c_\eps &\text{if}\ i\neq j\comma
		\end{dcases}
	\end{equation}
	we have $\mathfrak b_\eps,\mathfrak c_\eps\in (0,\frac12]$ and	
	the following limits 
	\begin{equation}\label{eq:b-c-limits}
		\mathfrak b\eqdef \lim_{\eps \to 0}\mathfrak b_\eps \comma\qquad \mathfrak c\eqdef \lim_{\eps\to 0}\mathfrak c_\eps
	\end{equation}
	exist in $(0,\frac12]$.
\end{lemma}
\begin{remark}\label{remark:Q-Q-i-j}
	Comparing $\cQ_\eps^{i,j}$ and $\cQ_\eps$ given, respectively, in \eqref{eq:Q-i-j} and \eqref{eq:B-invariance}, we have
	\begin{equation}\label{eq:sum-Q-i-j}
		\sum_{j=1}^d \cQ_\eps^{i,j}(t)=\cQ_\eps(t)\comma\qquad t \ge 0\comma i=1,\ldots, d\comma
	\end{equation}
	which, together with Lemma \ref{lemma:Q2} and \eqref{eq:B-12}, yields	
	\begin{equation}\label{eq:b-c-12}
		\mathfrak b_\eps +\tonde{d-1}\mathfrak c_\eps = \frac12\comma\qquad \mathfrak b+\tonde{d-1}\mathfrak c=\frac12\fstop
	\end{equation}
\end{remark}
The following result refines property \eqref{eq:B-12}  when the extremes of integration are not $t=0$ and $t=\infty$, but rather $t=T>0$ and $t=\infty$. As already done for Lemmas \ref{lemma:Q1} and \ref{lemma:Q2}, we postpone its proof to \hyperref[sec:app]{Appendix}.
\begin{lemma}\label{lemma:Q3} There exists a constant $c_3=c_3(d)>0$ satisfying, for all $T> 0$,
	\begin{equation}\label{eq:lemma-Q3}
		\int_T^\infty \cQ_\eps(t)\, \dd t\le c_3\,{\eps^3}\tonde{T\vee \eps^2}^{-3/2}\fstop
	\end{equation}
\end{lemma}

We are now ready to prove one of the two propositions of the section. All throughout, we employ the notation of Lemma \ref{lemma:Q2} introduced for the case $d\ge 2$; if $d=1$, by writing, e.g., $\cQ_\eps^{i,j}$, $\mathfrak b_\eps$, $\mathfrak c_\eps$, we actually mean $\cQ_\eps$, $\frac12$, $0$,  respectively. 
\begin{proof}[Proof of Proposition \ref{pr:mean-1}]
	In view of Lemma \ref{lemma:Q2}, we may adopt the following shorthand notation for the time integral in \eqref{eq:tilde-Pi-eps-k}: for all integers $k\ge 1$, and for all  $t\ge 0$ and $i,j=1,\ldots, d$, 
	\begin{equation}\label{eq:R-i-j-k}
		R_t^{\eps,i,j,k}\eqdef \int_{[0,t]_>^k}\dd s_1\cdots \dd s_k  \sum_{j_0,j_1,\ldots, j_k=1}^d\tonde{\prod_{\ell=1}^k \cQ_\eps^{j_{\ell-1},j_\ell}(s_{\ell-1}-s_\ell)}\car_{j_0=i}\car_{j_k=j}\comma
	\end{equation}
	with the usual convention that $s_0\eqdef t$. Similarly,  we write
	\begin{equation}\label{eq:R-i-j-k-infty}
		R_\infty^{\eps,i,j,k}\eqdef \sum_{j_0,\ldots, j_k=1}^d \tonde{\prod_{\ell=1}^k\int_0^\infty \cQ_\eps^{j_{\ell-1},j_\ell}(s)\, \dd s}\car_{j_0=i}\car_{j_k=j}\fstop
	\end{equation}
	Further, set $R_t^{\eps,i,j,k}=R_\infty^{\eps,i,j,k}= \car_{i=j}$ for $k=0$. With these definitions, we readily obtain 
	\begin{equation}\label{eq:R-i-j-k-ub}
		R_t^{\eps,i,j,k}\le R_\infty^{\eps,i,j,k}\comma\qquad k\ge  0\fstop
	\end{equation} We now claim that, for all $t> 0$ and some $C=C(d)>0$, 
	\begin{equation}\label{eq:R-i-j-k-lb}
		\big(1-Ct^{-3/2}\eps^{3/2}\big)^k\,R_\infty^{\eps,i,j,k}\le R_t^{\eps,i,j,k}\comma\qquad k\le \lfloor \eps^{-1}\rfloor\fstop
	\end{equation} 
	Indeed, for all $k\le \lfloor \eps^{-1}\rfloor$,  $R_t^{\eps,i,j,k}$ is bounded below by
	\begin{equation}\label{eq:R-i-j-k-lb2}
		\sum_{j_0,\ldots, j_k=1}^d\car_{j_0=i}\car_{j_k=j}\int_{t-\eps t}^t \dd s_1\,  \cQ_\eps^{j_0,j_1}(\purple{t-s_1})\int_{s_1-\eps t}^{s_1}\dd s_2\cdots \int_{s_{k-1}-\eps t}^{s_{k-1}}\dd s_k\, \cQ_\eps^{j_{k-1},j}(s_{k-1}-s_k)\comma
	\end{equation}
	where we used $\cQ_\eps^{j_{\ell-1},j_\ell}\ge 0$ and $k\le \lfloor \eps^{-1}\rfloor$ (guaranteeing that $s_{k-1}-\eps t\ge 0$). Thanks to Lemma \ref{lemma:Q3}, we estimate each nested integral as follows: for all $\ell=1,\ldots, k$, 
	\begin{align}
		\int_{s_{\ell-1}-\eps t}^{s_{\ell-1}} \cQ_\eps^{j_{\ell-1},j_\ell}(s_{\ell-1}-s_\ell)\, \dd s_\ell&=\int_0^\infty \cQ_\eps^{j_{\ell-1},j_\ell}(s)\, \dd s-\int_{t\eps}^\infty \cQ_\eps^{j_{\ell-1},j_\ell}(s)\, \dd s\\
		&\ge \int_0^\infty \cQ_\eps^{j_{\ell-1},j_\ell}(s)\, \dd s-c_3\,t^{-3/2}\eps^{3/2}\\
		&\ge \big(1-Ct^{-3/2}\eps^{3/2}\big)\int_0^\infty \cQ_\eps^{j_{\ell-1},j_\ell}(s)\, \dd s\comma
	\end{align}
	where the first inequality used Lemma \ref{lemma:Q3}, whereas for the last inequality we employed the non-degeneracy of the limits \eqref{eq:b-c-limits} (thus, ensuring $\liminf_{\eps\to 0}\int_0^\infty \cQ_\eps^{i,j}(s)\, \dd s>0$) and chose a sufficiently large constant  $C=C(d)>0$. Inserting these bounds into \eqref{eq:R-i-j-k-lb2}, we obtain \eqref{eq:R-i-j-k-lb}. In conclusion, by combining \eqref{eq:R-i-j-k-ub} and \eqref{eq:R-i-j-k-lb}, we get
	\begin{align}
		\sum_{k=0}^\infty \abs{R_t^{\eps,i,j,k}-R_\infty^{\eps,i,j,k}} &\le 	\sum_{k=1}^{\lfloor \eps^{-1}\rfloor}\tonde{1-\big(1-C't^{-3/2}\eps^{3/2}\big)^k} R_\infty^{\eps,i,j,k} + \sum_{k=\lfloor \eps^{-1}\rfloor +1}^\infty R_\infty^{\eps,i,j,k} \\
		&\le {1-\big(1-C't^{-3/2}\eps^{3/2}\big)^{\lfloor \eps^{-1}\rfloor}} + 2^{-\lfloor \eps^{-1}\rfloor}\comma
	\end{align}
	where the last inequality used \eqref{eq:sum-Q-i-j} and property \eqref{eq:B-12}. In conclusion, we obtain
	\begin{equation}\label{eq:R-t-R-infty}
		\sum_{k=0}^\infty \abs{R_t^{\eps,i,j,k}-R_\infty^{\eps,i,j\purple{,}k}}\xrightarrow{\eps\to 0}0\comma
	\end{equation}
	for all $t>0$ and $i,j=1,\ldots, d$.
	
	We now prove that, for all $i,j=1,\ldots, d$, we have
	\begin{equation}\label{eq:sum-R-i-j-infty}
		\sum_{k=1}^\infty R_\infty^{\eps,i,j,k} \xrightarrow{\eps\to 0}  	 \purple{\mathfrak r^{i,j}}\eqdef \frac1{1-\mathfrak b+\mathfrak c}\times \begin{dcases}
			\mathfrak b+\mathfrak c &\text{if}\ i=j\\
			2\mathfrak c &\text{if}\ i\neq j\comma
		\end{dcases}
	\end{equation}
	where $\mathfrak b, \mathfrak c\in (0,\frac12]$ are the limits in \eqref{eq:b-c-limits}.
	
	Recall the definition of $R_\infty^{\eps,i,j,k}$, $k\in \N$, from \eqref{eq:R-i-j-k-infty}.
	Thanks to Lemma \ref{lemma:Q2}, $R_\infty^{\eps,i,j,k}$ may be precisely determined by counting how many times adjacent indices in the sequence $(j_0,j_1,\ldots,j_{k-1}, j_k)\in \set{1,\ldots, d}^{k+1}$ coincide. For this purpose, let $N(j_0,\ldots, j_k)\in \set{0,1,\ldots, k}$ denote the number of adjacent indices with the same value, i.e.,
	\begin{equation}
		N(j_0,\ldots, j_k)\eqdef \sum_{\ell=1}^k \car_{j_{\ell-1}=j_\ell}\fstop
	\end{equation}
	Then, by \eqref{eq:b-c-eps}, we have
	\begin{equation}\label{eq:sum-b-c-i-j}
		R_\infty^{\eps,i,j,k}=	\sum_{j_0,\ldots, j_k=1}^d \tonde{\mathfrak b_\eps^{N(j_0,\ldots, j_k)}\mathfrak c_\eps^{k-N(j_0,\ldots, j_k)}}\car_{j_0=i}\car_{j_k=j}\fstop
	\end{equation}
	Clearly,  the above quantity depends on $i,j=1,\ldots, d$ only through $\car_{i=j}$; therefore, we may introduce the following shorthand notation: for all $k\in \N$, 
	\begin{equation}
		F_k^\eps\eqdef R_\infty^{\eps,i,j,k}\quad \text{if}\ i=j\comma\qquad G_k^\eps\eqdef R_\infty^{\eps,i,j,k}\quad \text{if}\ i\neq j\fstop
	\end{equation}
	
	When $d\ge 2$, such quantities satisfy 
	\begin{equation}
		F_1^\eps = \mathfrak b_\eps\comma\qquad G_1^\eps = \mathfrak c_\eps\comma
	\end{equation} 
	as well as the following one-step recursive formula: for $k\ge 2$, 
	\begin{eqnarray}
		F_k^\eps &=\ \mathfrak b_\eps\, F_{k-1}^\eps + \tonde{\mathfrak c_\eps+\tonde{d-2}\mathfrak c_\eps} G_{k-1}^\eps &=\ \mathfrak b_\eps\, F_{k-1}^\eps + \tonde{\tfrac12-\mathfrak b_\eps}G_{k-1}^\eps \\
		G_k^\eps &=\  \mathfrak c_\eps\, F_{k-1}^\eps + \tonde{\mathfrak b_\eps+\tonde{d-2}\mathfrak c_\eps} G_{k-1}^\eps &=\ \mathfrak c_\eps\, F_{k-1}^\eps + \tonde{\tfrac12-\mathfrak c_\eps}G_{k-1}^\eps\comma
	\end{eqnarray}
	where for the second set of identities we used relation \eqref{eq:b-c-12}.
	In other words, letting $M_\eps\eqdef\begin{pmatrix}
		\mathfrak b_\eps &\tfrac12-\mathfrak b_\eps\\
		\mathfrak c_\eps &\tfrac12-\mathfrak c_\eps
	\end{pmatrix}$, 	we just obtained
	\begin{equation}
		\begin{pmatrix}
			F_k^\eps\\
			G_k^\eps
		\end{pmatrix} = M_\eps \begin{pmatrix}
			F_{k-1}^\eps\\
			G_{k-1}^\eps
		\end{pmatrix}
		=\ldots = M_\eps^{k-1}\begin{pmatrix}
			\mathfrak b_\eps\\
			\mathfrak c_\eps
		\end{pmatrix}\comma
	\end{equation}
	from which we get
	\begin{equation}
		\sum_{k=1}^\infty M_\eps^{k-1}\begin{pmatrix}
			\mathfrak b_\eps\\
			\mathfrak c_\eps
		\end{pmatrix} = \tonde{\mathds I-M_\eps}^{-1}\begin{pmatrix}
			\mathfrak b_\eps\\
			\mathfrak c_\eps
		\end{pmatrix}\comma\qquad \text{with}\ \mathds I\eqdef \begin{pmatrix}
			1&0\\
			0&1
		\end{pmatrix}\fstop
	\end{equation}
	In conclusion,  since
	\begin{equation}
		\tonde{\mathds I-M_\eps}^{-1}
		=\frac1{1-\mathfrak b_\eps+\mathfrak c_\eps}\begin{pmatrix}
			1+2\mathfrak c_\eps &1-2\mathfrak b_\eps\\
			2\mathfrak c_\eps &2-2\mathfrak b_\eps
		\end{pmatrix}\comma
	\end{equation}
	the left-hand side of \eqref{eq:sum-R-i-j-infty} equals,  depending on whether $i=j$ or $i\neq j$,
	\begin{equation}\label{eq:candidate-a}
		\frac1{1-\mathfrak b_\eps+\mathfrak c_\eps}\times 	\begin{dcases}\mathfrak b_\eps+\mathfrak c_\eps
			&\text{if}\ i=j\\		
			2\mathfrak c_\eps &\text{if}\ i\neq j\fstop
	\end{dcases}\end{equation}
	Taking  $\eps \to 0$, \eqref{eq:b-c-limits} yields \eqref{eq:sum-R-i-j-infty} for $d\ge2$.	
	When $d=1$,   $i\neq j$ is not possible; hence, we have $F_k^\eps = \mathfrak b_\eps^k$ for all $k\in \N$. Since $\mathfrak b_\eps = \frac12$ (see Remark \ref{remark:Q-Q-i-j}), the left-hand side of \eqref{eq:sum-R-i-j-infty} equals $\sum_{k=1}^\infty F_k^\eps = \sum_{k=1}^\infty 2^{-k}=1$. Since $\mathfrak b=\frac12$ and $\mathfrak c=0$ when $d=1$, the right-hand side of \eqref{eq:sum-R-i-j-infty} equals $1$, as desired.

	Now, recall that $\nabla^ju_t = \nabla^jP_t u_0= P_t \nabla^ju_0$, for all $j=1,\ldots, d$, and similarly for the $\eps$-semigroup and corresponding gradients. By the functional CLT for $(X_t^\eps)_{t\ge 0}$, 	the assumptions $u_0\in \cC^2(\T^d)$ with $u_0^\eps = u_0\restr{\T_\eps^d}$ ensure
	\begin{equation}\label{eq:FCLT}
\purple{\mathfrak s_\eps\eqdef}	\sup_{t\ge 0}\sup_{x\in \T_\eps^d}\abs{\tonde{\nabla^{\eps,j}P_t^\eps u_0^\eps(x)}^2-\tonde{\nabla^j u_t(x)}^2}\xrightarrow{\eps\to 0}0\comma\qquad  j=1,\ldots, d\fstop
	\end{equation}

	In view of \eqref{eq:R-t-R-infty}, \eqref{eq:sum-R-i-j-infty} and \eqref{eq:FCLT}, the claim in \eqref{eq:pr-mean-1} follows by a triangle inequality. Indeed, 	recalling the definitions  \eqref{eq:tilde-Pi-eps-k}, \eqref{eq:R-i-j-k}, \eqref{eq:R-i-j-k-infty}, 	 \eqref{eq:Pi-eps-0}, and the uniform boundedness of $\Pi_t^{\eps,j,0}(u_0^\eps)$, we have,  as $\eps \to 0$,
	\begin{align}
		\sum_{k=0}^\infty \tilde \Pi_t^{\eps,i,k}(u_0^\eps)(x) &= \sum_{j=1}^d \Pi_t^{\eps,j,0}(u_0^\eps)(x)\sum_{k=0}^\infty R_t^{\eps,i,j,k}\\
		\eqref{eq:FCLT}\Longrightarrow\quad
		&\approx \sum_{j=1}^d \tonde{\nabla^j u_t(x)}^2\sum_{k=0}^\infty R_t^{\eps,i,j,k}
		\\
		\eqref{eq:R-t-R-infty}\Longrightarrow\quad	&\approx \sum_{j=1}^d \tonde{\nabla^j u_t(x)}^2 \sum_{k=0}^\infty R_\infty^{\eps,i,j,k}\\
		\eqref{eq:sum-R-i-j-infty}\Longrightarrow\quad	&\approx \sum_{j=1}^d \tonde{\nabla^j u_t(x)}^2 \purple{\tonde{\car_{i=j}+\mathfrak r^{i,j}}}\fstop
	\end{align}
	 Observe that the above convergences hold for all $t>0$, $i=1,\ldots, d$, 	 and uniformly over $x\in \T_\eps^d$. This proves \eqref{eq:pr-mean-1} with $\mathfrak a=\purple{\mathfrak r^{i,i}}$, thus, concluding the proof of the proposition.
\end{proof}
\begin{remark}\label{remark:coefficient-a-b-c}
	The proof of Proposition \ref{pr:mean-1} --- in particular,  relations \eqref{eq:sum-R-i-j-infty} and \eqref{eq:candidate-a} --- allows us to express the value $\mathfrak a=\mathfrak a(d)\in (0,\frac12]$ in terms  of the limits $\mathfrak b,\mathfrak c\in (0,\frac12]$ in \eqref{eq:b-c-limits}, combined with the relation \eqref{eq:b-c-12}:
	\begin{equation}\label{eq:a}
		\mathfrak a\eqdef \frac{\mathfrak b+\mathfrak c}{1-\mathfrak b+\mathfrak c}=\frac{1-2(d-2)\,\mathfrak c}{1
			+2d\,\mathfrak c}\fstop
	\end{equation}
\end{remark}

\begin{proof}[Proof of Proposition \ref{pr:mean-2}] Recalling \eqref{eq:Pi-eps-0}, \eqref{eq:Pi-eps-k} and \eqref{eq:tilde-Pi-eps-k}, we have, for all $k\in \N$, 
	\begin{align}
		&\Pi_t^{\eps,i,k}(u_0^\eps)(x)-\tilde \Pi_t^{\eps,i,k}(u_0^\eps)(x)\\
		&=\int_{[0,t]_>^k}\dd s_1\cdots \dd s_{k-1}\sum_{y_1,\ldots, y_{k-1}\in \T_\eps^d}\sum_{j_1,\ldots, j_{k-1}=1}^d\tonde{\prod_{\ell=1}^{k-1}q_{s_{\ell-1}-s_\ell}^{\eps,j_{\ell-1},j_\ell}(y_{\ell-1},y_\ell)}\\
		&\qquad \times \int_0^{s_{k-1}}\dd s_k\sum_{y_k\in \T_\eps^d}\sum_{j_k=1}^d q_{s_{k-1}-s_k}^{\eps,j_{k-1},j_k}(y_{k-1},y_k)\set{\Pi_{s_k}^{\eps,j_k,0}(u_0^\eps)(y_k)-\Pi_t^{\eps,j_k,0}(u_0^\eps)(x)}\fstop
	\end{align}  Further, by $u_0\in \cC^2(\T^d)$, $u_0^\eps =u_0\restr{\T_\eps^d}$ and the functional CLT in \eqref{eq:FCLT}, we get
	\begin{equation}
		\abs{\Pi_{s_k}^{\eps,j,0}(u_0^\eps)(y_k)-\Pi_t^{\eps,j,0}(u_0^\eps)(x)}\le	\purple{2\mathfrak s_\eps}+C_1\abs{P_{s_k} \nabla^ju_0(y_k)-P_t \nabla^ju_0(x)}\le C_2\comma
	\end{equation}
	for some $C_1, C_2>0$ depending only on $u_0\in \cC^2(\T^d)$.
	Hence, the left-hand side of \eqref{eq:pr-mean-2} is bounded above by \purple{a term bounded above by $4\mathfrak s_k$, which, by \eqref{eq:FCLT},  tends to $0$ as $\eps \to 0$, plus}
	\begin{equation}\label{eq:final-K}
		C_1\sum_{k=1}^\infty  K_t^{\eps,k}(u_0)(x)\le 		 \tonde{C_1\sum_{k=1}^{\lceil \abs{\log \eps}\rceil^2}  K_t^{\eps,k}(u_0)(x)  } +C_2\,	 2^{-\lceil \abs{\log \eps}\rceil^2}\comma 	\end{equation}
	where
	\begin{align}
		&	K_t^{\eps,k}(u_0)(x)\eqdef \int_{[-\infty,t]_>^{k-1}}\dd s_1\cdots \dd s_{k-1} \sum_{y_1,\ldots, y_{k-1}\in \T_\eps^d}\sum_{j_0,\ldots, j_{k-1}=1}^d \tonde{\prod_{\ell=1}^{k-1} q_{s_{\ell-1}-s_\ell}^{\eps,j_{\ell-1},j_\ell}(y_{\ell-1},y_\ell)}\\
		&\qquad \times \int_{-\infty}^{s_{k-1}}\dd s_k \sum_{y_k\in \T_\eps^d}\sum_{j_k=1}^d q_{s_{k-1}-s_k}^{\eps,j_{k-1},j_k}(y_{k-1},y_k)\abs{P_{s_k\vee 0} \nabla^{j_k} u_0(y_k)-P_t \nabla^{j_k}u_0(x)}\fstop
		\label{eq:K-def}
	\end{align}
	
	From now on, the proof ingredients are similar to those of Lemma \ref{lemma:short-time} and Proposition \ref{pr:mean-1}. More in detail,  we show, as already done in Proposition \ref{pr:mean-1}, that the integrals appearing in the definition \eqref{eq:K-def} of $K_t^{\eps,k}(u_0)(x)$ may be restricted around their upper extremes, e.g., integrate $s_k$ over $[s_{k-1}-\eps t,s_{\purple{k-1}}]$ rather than on the whole $[0,s_{k-1}]$. Then, we conclude by	exploiting the space-time continuity of $ P_s\nabla^j u_0(x)$ and the localization arguments employed in the proof of Lemma \ref{lemma:short-time}.		We sketch	 this part of the proof below.
	
	By splitting the domain of integration of the first integral in \eqref{eq:K-def} into $\set{s_1<t-\eps t}$ and its complement, we get, for some $C_3=C_3(u_0)>0$, 
	\begin{equation}\label{eq:K-ub}
		K_t^{\eps,k}(u_0)(x)\le C_3\, L_t^{\eps,k} +  K_t^{\eps,k,1}(u_0)(x)\comma
	\end{equation}
	where, for all $m=1,\ldots, k$,
	\begin{align}
		& 	K_t^{\eps,k,m}(u_0)(x)\eqdef \int_{[-\infty,t]_>^k}\dd s_1\cdots \dd s_k \sum_{y_1,\ldots, y_k\in \T_\eps^d}\sum_{j_0,\ldots, j_k=1}^d \tonde{\prod_{\ell=1}^k q_{s_{\ell-1}-s_\ell}^{\eps,j_{\ell-1},j_\ell}(y_{\ell-1},y_\ell)}\\
		&\qquad \times \abs{P_{s_k} \nabla^{j_k} u_0(y_k)-P_t \nabla^{j_k}u_0(x)}\tonde{\prod_{\ell=1}^m\car_{[s_{\ell-1}-\eps t,s_{\ell-1}]}(s_\ell)}\comma
		\label{eq:K-def-2}
	\end{align}
	and 
	\begin{equation}\label{eq:L-t}
		L_t^{\eps,k}\eqdef 
		d\tonde{\int_{t\eps}^\infty \cQ_\eps(s)\, \dd s}\tonde{\int_0^\infty \cQ_\eps(s)\, \dd s}^{k-1}\le dc_3\,t^{-3/2}\eps^{3/2}\, 2^{-(k-1)}	\fstop
	\end{equation}
	Note that in \eqref{eq:L-t} we used Lemmas \ref{lemma:Q1} and \ref{lemma:Q3} to obtain the inequality.	 Moreover, with the notation in \eqref{eq:K-def-2}, we may define $K_t^{\eps,k,0}\eqdef K_t^{\eps,k}$.
	Applying this same strategy to $ K_t^{\eps,k,1}(u_0)(x)$, by splitting the integral with respect to $s_2\in (-\infty,s_1)$ into $\set{s_2<s_1-\eps t}$ and its complement, we get
	\begin{equation}
		K_t^{\eps,k,1}(u_0)(x)\le C_3\, L_t^{\eps,k} + K_t^{\eps,k,2}(u_0)(x)\fstop
	\end{equation}
	Thus, by iterating $k$ times the above inequality, \eqref{eq:K-ub} yields
	\begin{equation}\label{eq:K-t-ub}
		K_t^{\eps,k}(u_0)(x)\le C_3\, k\, L_t^{\eps,k} + K_t^{\eps,k,k}(u_0)(x)\fstop
	\end{equation}
	Since $k\le \lceil \abs{\log \eps}\rceil^2$,  we further get, for all $\rho \in (0,\frac12)$, 
	\begin{align}
		&K_t^{\eps,k,k}(u_0)(x)\le d2^{-k}\bigg\{\max_{j=1,\ldots, d}\sup_{\substack{r,s\ge 0\\\abs{r-s}<\eps \abs{\log \eps}^2}}\norm{P_s\nabla^j u_0-P_r \nabla^j u_0}_{L^\infty(\T^d)}\bigg\}\\
		&\qquad+ d2^{-k}\bigg\{\max_{j=1,\ldots, d}\sup_{s\ge 0}\sup_{\substack{x,y\in \T^d\\
				\abs{x-y}\le\rho}}\abs{P_s\nabla^j u_0(y)-P_s\nabla^j u_0(x)}\bigg\} \\
		&\qquad+ C_4\int_{[0,\eps \abs{\log \eps}^2t]_>^k}\dd s_1\cdots \dd s_k \sum_{\substack{y_1,\ldots, y_k\in \T_\eps^d\\
				\abs{y_k-x}>\rho}}\sum_{j_0,\ldots, j_k=1}^d \tonde{\prod_{\ell=1}^k q_{s_{\ell-1}-s_\ell}^{\eps,j_{\ell-1},j_\ell}(y_{\ell-1},y_\ell)}
		\\
		&\qquad \defeq d2^{-k}\set{\omega(\eps\abs{\log\eps}^2)+\vartheta(\rho)}+C_4\, I_{\eps\abs{\log \eps}^2t}^{\eps,k}(\eps,\rho)
		\fstop
		\label{eq:K-t-k-k-ub}
	\end{align} 
	Note that the time integral above coincides with $I_t^{\eps,k}(\rho_1,\rho_2)$ introduced in \eqref{eq:I} ---  and estimated in \eqref{eq:I-ub} --- with the following choices:
	\begin{equation}
		\rho_1=\eps\comma\qquad \rho_2=\rho\comma\qquad t=\eps\abs{\log \eps}^2 t\comma \qquad n=\lceil \abs{\log \eps}\rceil^2\fstop
	\end{equation}
	As a consequence, summing over  $k=1,\ldots, \lfloor \eps^{-1/2}\rfloor$, \eqref{eq:K-t-ub}, \eqref{eq:L-t},	 \eqref{eq:K-t-k-k-ub} and \eqref{eq:I-ub} yield
	\begin{align}
		&\sum_{k=1}^{\lceil \abs{\log \eps}\rceil^2} K_t^{\eps,k}(u_0)(x)\\
		&\qquad\le C_5 \tonde{t^{-3/2}\eps^{3/2} +\omega(\eps\abs{\log \eps}^2)+\vartheta(\rho)+\eps^{-(d+2)}\exp\tonde{-\frac{C'\rho}{\abs{\log \eps}^3\eps^{1/2}}} }\comma
	\end{align}
	for some $C_5>0$ depending only on $d\ge 1$ and $u_0\in \cC^2(\T^d)$.
	Taking, e.g., $\rho=\eps^{1/3}$, the above bound and the space-time continuity of $P_s\nabla^ju_0$ ensure that the left-hand side above vanishes as $\eps \to 0$, for all $t>0$. The desired claim now follows from \eqref{eq:final-K}. 
\end{proof}
\section{Proof of Theorem \ref{th:flu}}\label{sec:proof-th-1}
In this section, we prove the functional CLT for  the distribution-valued \textit{c\`{a}dl\`{a}g} process $(\cY_t^\eps)_{t\ge 0}$ defined in \eqref{eq:fluctuation-fields}. Let us recall from \eqref{eq:martingale-decomposition}--\eqref{eq:martingale-eps} the decomposition of $(\cY_t^\eps)_{t\ge 0}$ into a drift and a martingale term:
\begin{equation}\label{eq:decomposition-2}
	\cY_t^\eps(f)=\int_0^t \cY_s^\eps(\tfrac12 \Delta_\eps f)\, \dd s+\cM_t^\eps(f)\comma \qquad t \ge 0\comma f \in \cC^\infty(\T^d)\fstop
\end{equation}Observe that, by \eqref{eq:predictable-q-v} and the notation introduced in \eqref{eq:Gamma-eps}, $(\cU_t^\eps)_{t\ge 0}$ defined as
\begin{equation}\label{eq:mart-2}
	\cU_t^\eps(f)\eqdef \cM_t^\eps(f)^2-\varGamma_t^\eps(\tfrac12 \nabla^\eps f)\comma\qquad t \ge 0\comma f \in \cC^\infty(\T^d)\comma
\end{equation} 
is also a distribution-valued martingale (with respect to the natural filtration of $(\cY_t^\eps)_{t\ge 0}$). Recall that all these fields depend on the initial conditions \begin{equation}u_0\in \cC^2(\T^d)\qquad \text{and}\qquad u_0^\eps \equiv u_0\restr{\T_\eps^d}\in \R^{\T_\eps^d}\comma 
\end{equation} which we fix all throughout this section.

As already mentioned in Section \ref{sec:sketch-proof+LLN}, the proof of Theorem \ref{th:flu} may be divided into steps (i)--(iii): tightness of the sequence, continuity of the limit points, and characterization of the limit, respectively. As we will see, these three steps build on Theorem \ref{th:LLN} and  two main estimates, which we  present and prove in the next subsection. 

\subsection{Main estimates}
Define, for all $g\in \R^{\T_\eps}$, 
$
\norm{g}_{{\rm Lip}(\T_\eps^d)}\eqdef \max_{i=1,\ldots, d}\ttnorm{\nabla^{\eps,i}g}_{L^\infty(\T_\eps^d)}$.
This definition naturally extends to functions $f\in \cC^\infty(\T^d)$. 

\begin{lemma}[Second moments of fields]\label{lemma:second-moment-fields}
	For all $f\in \cC^\infty(\T^d)$ and $t\ge 0$, we have
	\begin{align}\label{eq:second-moment-field}
		\E\big[\cY_t^\eps(f)^2\big]&\le \norm{f}_{{\rm Lip}(\T_\eps^d)}^2\norm{u_0^\eps}_{L^\infty(\T_\eps^d)}^2
		\\
		\E\big[\cM_t^\eps(f)^2\big]&\le \norm{f }_{{\rm Lip}(\T_\eps^d)}^2 \norm{u_0^\eps}_{L^\infty(\T_\eps^d)}^2\fstop
		\label{eq:ub-pqv}
	\end{align}
\end{lemma}
\begin{proof}
	We start with the proof of \eqref{eq:second-moment-field}. By expanding the square and taking expectation, we obtain
	\begin{equation}\label{eq:square}
		\E\big[\cY_t^\eps(f)^2\big] = \theta_\eps^2\,\eps^{2d}\sum_{x,y\in \T_\eps^d}\set{\E\big[u_t^\eps(x)u_t^\eps(y)\big]-\E\big[u_t^\eps(x)\big]\E\big[u_t^\eps(y)\big]}f(x)f(y)\fstop
	\end{equation}
	Writing $u_t^\eps(x)$ as the mild solution of \eqref{eq:SDE-poisson-compensated}, we get
	\begin{align}
		u_t^\eps(x)&= P_t^\eps u_0^\eps(x)+\sum_{i=1}^d \frac{\eps^2}2\int_0^t P_{t-s}^\eps \nabla_*^{\eps,i}\tonde{\dd \bar N_s^{\eps,i}\,\nabla^{\eps,i}u_{s^-}^\eps}(x)\\
		&= P_t^\eps u_0^\eps(x) +\sum_{i=1}^d \frac{\eps^2}2 \int_0^t \sum_{z\in \T_\eps^d} {\nabla_*^{\eps,i}p_{t-s}^\eps(\emparg,z)(x)}\,{\nabla^{\eps,i}\,u_{s^-}^\eps(z)\, \dd \bar N_s^{\eps,i}(z) }\fstop
	\end{align}
	\purple{Remark that the notation $\nabla_*^{\eps,i}p_{t-s}^\eps(\emparg,z)(x)$ is to emphasize that the discrete derivative acts on the $x$-variable.}
	Since $\bar N^\eps=(\bar N_t^{\eps,i}(x))_{t\ge 0,\,x\in \T_\eps^d,\,i=1,\ldots, d}$ is a family of i.i.d.\ martingales satisfying $\E[(\bar N_t^{\eps,i}(x))^2]=\eps^{-2}\,t$, we obtain
	\begin{align}
		&\E\big[u_t^\eps(x)u_t^\eps(y)\big]=\E\big[u_t^\eps(x)\big]\E\big[u_t^\eps(y)\big] \\
		&\qquad+ \sum_{i=1}^d \frac{\eps^2}{4}\int_0^t \sum_{z\in \T_\eps^d}\E\big[\tonde{\nabla^{\eps,i}u_s^\eps(z)}^2\big]\tonde{\nabla_*^{\eps,i}p_{t-s}^\eps(\emparg,z)(x)\,\nabla_*^{\eps,i}p_{t-s}^\eps(\emparg,z)(y)} \dd s\fstop 
	\end{align}
	Inserting this identity into \eqref{eq:square} and recalling that $\theta_\eps=\eps^{-(d/2+1)}$, we get
	\begin{align}
		\E\big[\cY_t^\eps(f)^2\big]&= \frac{\eps^d}4\sum_{z\in \T_\eps^d}\sum_{i=1}^d\int_0^t \E\big[\tonde{\nabla^{\eps,i}u_s^\eps(z)}^2\big]  \tonde{\sum_{x\in \T_\eps^d}\tonde{\nabla_*^{\eps,i}p_{t-s}^\eps(\emparg,z)(x)} f(x)}^2 
		\dd s\\		&= \eps^d \sum_{z\in \T_\eps^d}\sum_{i=1}^d \int_0^t \E\big[\tonde{\nabla^{\eps,i}u_s^\eps(z)}^2\big] \tonde{\tfrac12P_{t-s}^\eps\nabla^{\eps,i} f(z)}^2 \dd s\fstop
		\label{eq:var-explicit}
	\end{align}
	By the estimate $\ttnorm{P_{t-s}^\eps \nabla^{\eps,i}f}_{L^\infty(\T_\eps^d)}\le \ttnorm{\nabla^{\eps,i}f}_{L^\infty(\T_\eps^d)}\le \norm{f}_{{\rm Lip}(\T_\eps^d)}$, 
	\begin{equation}\label{eq:final-second-moment-field}
		\E\big[\cY_t^\eps(f)^2\big]\le \frac12\norm{f}_{{\rm Lip}(\T_\eps^d)}^2\int_0^t \frac12\,\E\big[\norm{u_s^\eps}_{H^1(\T_\eps^d)}^2\big]\dd s\le \frac12\norm{f}_{{\rm Lip}(\T_\eps^d)}^2 \norm{u_0^\eps}_{L^2(\T_\eps^d)}^2\comma
	\end{equation}
	where the second inequality used Aldous-Lanoue identity \eqref{eq:AL-identity}. 
	This  yields \eqref{eq:second-moment-field}.	
	
	For what concerns \eqref{eq:ub-pqv}, since the process in  \eqref{eq:mart-2} is a mean-zero martingale, we get
	\begin{equation}
		\E\big[\cM_t^\eps(f)^2\big]=\E\big[\varGamma_t^\eps(\tfrac12\nabla^\eps f)\big]\fstop
	\end{equation} Noting $\varGamma_t^\eps(\frac12 \nabla f)$ equals  \eqref{eq:var-explicit} with $P_{t-s}^\eps \nabla^{\eps,i}f$ replaced by $\nabla^{\eps,i}f$ (cf.\ \eqref{eq:Gamma-eps}), 	the  argument leading from \eqref{eq:var-explicit} to \eqref{eq:final-second-moment-field} proves \eqref{eq:ub-pqv}.	
\end{proof}

\begin{lemma}[Size of jumps]\label{lemma:jump-sizes}For all $f\in \cC^\infty(\T^d)$, 
	we have
	\begin{equation}
		\E\big[\sup_{t\ge 0}\abs{\cY_t^\eps(f)-\cY_{t^-}^\eps(f)}\big]\le \eps^{d/2}\norm{f}_{{\rm Lip}(\T_\eps^d)}\norm{u_0^\eps}_{L^\infty(\T_\eps^d)}\fstop
	\end{equation}
\end{lemma}
\begin{proof}
	By the Poisson nature of the averaging dynamics, there is $\P$-a.s.\ at most one  update at the time. Suppose to observe an update  over the nearest neighbor vertices $x$ and $y=x+\eps e_i\in \T_\eps^d$ at time $t\ge 0$; then, 
	\begin{align}
		&\abs{\cY_t^\eps(f)-\cY_{t^-}^\eps(f)}\\
		&\qquad= \theta_\eps\, \eps^d \abs{\tfrac12\tonde{u_{t^-}^\eps(x)+u_{t^-}(y)}\tonde{f(x)+f(y)}-u_{t^-}^\eps(x)\, f(x)-u_{t^-}^\eps(y)\, f(y)}\\
		&\qquad= \theta_\eps\, \eps^d\abs{\tfrac12\tonde{u_{t^-}^\eps(x)-u_{t^-}^\eps(y)}\tonde{f(x)-f(y)}}\\
		&\qquad\le \theta_\eps\, \eps^{d+1}\norm{f}_{{\rm Lip}(\T_\eps^d)}\norm{u_{t^-}^\eps}_{L^\infty(\T_\eps^d)}\\
		&\qquad \le \theta_\eps\, \eps^{d+1}\norm{f}_{{\rm Lip}(\T_\eps^d)}\norm{u_0^\eps}_{L^\infty(\T_\eps^d)}\comma
	\end{align}
	where the first inequality used $\abs{f(x)-f(y)}\le	\purple{\eps} \ttnorm{f}_{{\rm Lip}(\T_\eps^d)}$, 	while the second one used the monotonicity $\norm{u_s^\eps}_{L^\infty(\T_\eps^d)}\le \norm{u_0^\eps}_{L^\infty(\T_\eps^d)}$. Recalling  $\theta_\eps = \eps^{-(d/2+1)}$ concludes the proof.	
\end{proof}

\subsection{Martingale convergence theorem}
In view of the martingale decomposition in \eqref{eq:decomposition-2}--\eqref{eq:mart-2} and the conclusions of Theorem \ref{th:LLN} and Lemma \ref{lemma:jump-sizes}, the classical martingale convergence theorem (see, e.g, \cite[Theorem 7.1.9, p.\ 339]{ethier_kurtz_1986_Markov}) applies to our case. For the reader's convenience, we collect its consequences in the following proposition.
\begin{proposition}\label{pr:mart-conv}  Recall  \eqref{eq:Gamma}.
	Then, for  all $f\in \cC^\infty(\T^d)$, there exists a unique  (in distribution) real-valued square-integrable continuous martingale $M^f=(M_t^f)_{t\ge 0}$  (with respect to its natural filtration) with $M_0^f=0$,  predictable quadratic variation equal to	 $\varGamma_t(\tfrac12 \nabla f)$, and Gaussian independent increments.	Moreover, recalling \eqref{eq:martingale-eps}, \eqref{eq:predictable-q-v}, and  \eqref{eq:Gamma-eps}, we have, for all $f\in \cC^\infty(\T^d)$,
	\begin{equation}\label{eq:mart-conv}
		(\cM_t^\eps(f))_{t\ge 0} \xRightarrow{\eps\to 0} M_t^f\comma\qquad \text{in}\ \cD([0,\infty);\R)\fstop
	\end{equation}
\end{proposition}
\begin{proof} Since $(\varGamma_t(\frac12 \nabla f))_{t\ge 0}$ is a continuous deterministic non-negative function, \cite[Theorem 7.1.1, p.\ 338]{ethier_kurtz_1986_Markov} proves the first assertion on the existence and uniqueness of the martingale $(M_t^f)_{t\ge 0}$.
	
	As for the claim in \eqref{eq:mart-conv}, it suffices to observe that the hypotheses of \cite[Theorem 7.1.4, p.\ 339]{ethier_kurtz_1986_Markov}, as well as conditions (1.16)--(1.19) therein, hold true in our case. Indeed, the predictable quadratic variations  $(\varGamma_t^\eps(\frac12\nabla^\eps f))_{t\ge 0}$ of $(\cM_t^\eps(f))_{t\ge 0}$ are continuous non-negative processes. Hence, the hypotheses and conditions (1.16) and (1.18) are fulfilled. Condition (1.17) therein follows from  \begin{equation}\label{eq:jumps}
		\abs{\cY_t^\eps(f)-\cY_{t^-}^\eps(f)}=\abs{\cM_t^\eps(f)-\cM_{t^-}^\eps(f)}\comma
	\end{equation} which holds because the drift term in \eqref{eq:decomposition-2} is continuous, and Lemma \ref{lemma:jump-sizes}. Finally, condition (1.19) is a consequence of the triangle inequality
	\begin{align}
		\abs{\varGamma_t^\eps(\tfrac12 \nabla^\eps f)-\varGamma_t(\tfrac12 \nabla f)}&\le \abs{\varGamma_t^\eps(\tfrac12 \nabla^\eps f)-\varGamma_t^\eps(\tfrac12 \nabla f)}+\abs{\varGamma_t^\eps(\tfrac12\nabla f)-\varGamma_t(\tfrac12 \nabla f)}\\
		& \le \ttnorm{f}_{{\rm Lip}(\T_\eps^d)}\max_{i=1,\ldots, d}\ttnorm{ \nabla^{\eps,i} f- \nabla^i f}_{L^\infty(\T_\eps^d)}\int_0^t \norm{u_s^\eps}_{H^1(\T_\eps^d)}^2\, \dd s\\
		&\qquad +
		\abs{\varGamma_t^\eps(\tfrac12\nabla f)-\varGamma_t(\tfrac12 \nabla f)} \comma
	\end{align}
	taking expectation,
	Aldous-Lanoue identity \eqref{eq:AL-identity}, and Theorem \ref{th:LLN}. This concludes the proof of the proposition.	
\end{proof}
\subsection{Tightness and continuity of limits} In view of the decomposition \eqref{eq:decomposition-2}, tightness and continuity of the limit points for $(\cY_t^\eps)_{t\ge 0}$ in $\cD([0,\infty);H^{-\alpha}(\T^d))$ is equivalent to the same property for the corresponding drift and martingale terms.
Since we  already established convergence for the martingale when tested against smooth test functions (Proposition \ref{pr:mart-conv}), the tightness proof for this term may be simplified. For this purpose, let us start by recalling from \cite{billingsley_convergence_1999} a useful characterization of tightness with continuous limit points for general \textit{c\`{a}dl\`{a}g} processes.

\begin{proposition}[Tightness \& continuous limits]\label{pr:tightness-general} Let $(\Xi,\norm{\emparg})$ be a Banach space. A sequence of $\Xi$-valued \emph{c\`{a}dl\`{a}g} processes $((Y_t^\eps)_{t\ge 0})_\eps$ is tight in $\cD([0,\infty);\Xi)$ and such that all limits are continuous if and only if the following three conditions hold true: for all $T>0$ and $\gamma> 0$, 
	\begin{equation}\label{eq:tightness-cond-1}
		\lim_{\zeta\to \infty}\limsup_{\eps \to 0}\P\bigg(\sup_{t\in [0,T]}\norm{Y_t^\eps}>\zeta\bigg)=0\comma
	\end{equation}
	\begin{equation}
		\label{eq:tightness-cond-2}
		\lim_{\delta \to 0}\limsup_{\eps \to 0}\P\bigg(\sup_{\substack{s,t\in [0,T]\\
				\abs{t-s}<\delta}}\norm{Y_t^\eps-Y_s^\eps}>\gamma\bigg)=0\comma
	\end{equation}
	and
	\begin{equation}
		\label{eq:tightness-cond-3}
		\lim_{\eps \to 0}\P\bigg(\sup_{t\in [0,T]}\norm{Y_t^\eps-Y_{t^-}^\eps}>\gamma\bigg)=0\fstop
	\end{equation}
\end{proposition}
\begin{proof}
	Just combine Theorems 13.2 and 13.4 in \cite{billingsley_convergence_1999} with the inequalities in (12.7)--(12.9) therein.
\end{proof}

\begin{remark}\label{remark:tightness}
	Thanks to the definition of the norm $\norm{\emparg}_{H^{-\alpha}(\T^d)}$ in \eqref{eq:H-norm},  Proposition \ref{pr:tightness-general}  simplifies if $(\Xi,\norm{\emparg})=(H^{-\alpha}(\T^d),\norm{\emparg}_{H^{-\alpha}(\T^d)})$, for some $\alpha >0$. Indeed, while the first condition in Proposition \ref{pr:tightness-general}, namely \eqref{eq:tightness-cond-1}, remains unchanged, the second condition  \eqref{eq:tightness-cond-2} may be replaced by
	\begin{equation}
		\lim_{\delta\to 0}\limsup_{\eps\to 0}\P\bigg(\sup_{\substack{s,t\in [0,T]\\
				\abs{t-s}<\delta}}\abs{Y_t^\eps(\phi_m)-Y_s^\eps(\phi_m)}>\gamma\bigg)=0\comma\qquad 	 m\in \Z^d \fstop
	\end{equation}
	An analogous simplification holds for the third condition \eqref{eq:tightness-cond-3}.
\end{remark}

We now have all we need to prove the desired claim for the sequence $((\cY_t^\eps)_{t\ge 0})_\eps$. 

\begin{proposition}\label{pr:tightness-field}
	The sequence $((\cY_t^\eps)_{t\ge 0})_\eps$ is tight in $\cD([0,\infty);H^{-\alpha}(\T^d))$, for all $\alpha > 3+d/2$, and all limit points are continuous.
\end{proposition}
\begin{proof}
	By \eqref{eq:decomposition-2}, it suffices to verify the conditions in Proposition \ref{pr:tightness-general} for the drift and martingale terms separately. Recall \eqref{eq:infty-norm}, and fix $T>0$, $\gamma>0$ all throughout the proof.

	We start with the drift term, for which we adopt the following shorthand notation:  
	\begin{equation}
		\cA_t^\eps(f)\eqdef \int_0^t \cY_s^\eps(\tfrac12 \Delta_\eps f)\, \dd s\comma\qquad t \ge 0\comma f\in \cC^\infty(\T^d)\fstop
	\end{equation}
	As for the first condition, namely \eqref{eq:tightness-cond-1}, Cauchy-Schwarz inequality yields
	\begin{align}
		\limsup_{\eps\to 0}	\E\big[\sup_{t\in [0,T]}\norm{\cA_t^\eps}_{H^{-\alpha}(\T^d)}^2\big]&	\le T \int_0^T \sum_{m\in \Z^d}(1+|m|^2)^{-\alpha}\limsup_{\eps\to 0} \E\big[\abs{\cY_s^\eps(\tfrac12 \Delta_\eps\phi_m)}^2\big]\dd s	\\
		&\le (2\pi)^4\,T^2\norm{u_0}_\infty^2\sum_{m\in \Z^d}(1+|m|^2)^{-(\alpha-3)}\comma
	\end{align}
	where the second inequality used \eqref{eq:second-moment-field} in Lemma \ref{lemma:second-moment-fields} and $\limsup_{\eps\to0}\norm{\Delta \phi_m}_{{\rm Lip}(\T_\eps^d)}^2\le (2\pi)^4(1+|m|^2)^3$. Since $\alpha>3+d/2$, the right-hand side above is finite. Therefore, Markov inequality yields \eqref{eq:tightness-cond-1} for the drift.
	For the second condition \eqref{eq:tightness-cond-2}, a similar argument yields, for all $m\in \Z^d$, 
	\begin{equation}
		\limsup_{\eps\to 0}\E\big[\sup_{\substack{s,t\in [0,T]\\
				\abs{t-s}<\delta}}\abs{\cA_t^\eps(\phi_m)-\cA_s^\eps(\phi_m)}\big]\le \delta\, T\,(2\pi)^4(1+|m|^2)^3\norm{u_0}_\infty^2\comma
	\end{equation}
	which vanishes as $\delta \to 0$. By Markov inequality and Remark \ref{remark:tightness}, this suffices to prove \eqref{eq:tightness-cond-2}. The third condition \eqref{eq:tightness-cond-3} is trivially satisfied since the drift is continuous.
	
	Turning to the martingale term $(\cM_t^\eps)_{t\ge 0}$, we have
	\begin{align}
		\E\big[\sup_{t\in [0,T]}\norm{\cM_t^\eps}_{H^{-\alpha}(\T^d)}^2\big]&\le \sum_{m\in \Z^d}(1+|m|^2)^{-\alpha}\, \E\big[\sup_{t\in [0,T]}\abs{\cM_t^\eps(\phi_m)}^2\big]\\
		&\le \sum_{m\in \Z^d}(1+|m|^2)^{-\alpha}\, \E\big[|\cM_T^\eps(\phi_m)|^2\big]\\
		&\le (2\pi)^2\norm{u_0}_\infty^2 \sum_{m\in \Z^d} (1+|m|^2)^{-(\alpha-1)}\comma
	\end{align}
	where  the second step used Doob inequality, while  the third one used \eqref{eq:ub-pqv} in Lemma \ref{lemma:second-moment-fields}.  The right-hand side above is finite because $\alpha >1+d/2$. This proves the first condition \eqref{eq:tightness-cond-1}. Now observe that Proposition \ref{pr:mart-conv}, \eqref{eq:jumps}, and Lemma \ref{lemma:jump-sizes}  ensure that, for all $m\in \Z^d$, the sequence $((\cM_t^\eps(\phi_m))_{t\ge 0})_\eps$ satisfies all three conditions in Proposition \ref{pr:tightness-general} with $(\Xi,\norm{\emparg})=(\C,\abs{\emparg})$. Remark \ref{remark:tightness} concludes  the proof.
\end{proof}

\subsection{Final step}

We now have all we need to prove Theorem \ref{th:flu}. Indeed, the decomposition of $(\cY_t^\eps)_{t\ge 0}$ in \eqref{eq:decomposition-2}, and  Propositions \ref{pr:mart-conv} and \ref{pr:tightness-field} ensure that any limit process, say $(\cY_t)_{t\ge 0}$, have paths in $\cC([0,\infty);H^{-\alpha}(\T^d))$, for all $\alpha >3+d/2$, and solve the  martingale problem \eqref{eq:mart-problem} in $\cC([0,\infty);H^{-\alpha}(\T^d))$ with $\cY_0=0$. Moreover, the martingale $(\cM_t)_{t\ge 0}$ in \eqref{eq:mart-problem} satisfies
\begin{equation}
	(\cM_t(f))_{t\ge 0} = (M_t^f)_{t\ge 0}\comma\qquad f \in \cC^\infty(\T^d)\comma
\end{equation}
where $(M_t^f)_{t\ge 0}$ is defined in Proposition \ref{pr:mart-conv}, and the above identity is meant in distribution.

The proof of Theorem \ref{th:flu} ends as soon as we show that such a limit martingale problem has a unique solution. For this purpose, 	we introduce 
$	\cS(\T^d)\eqdef \bigcap_{\alpha\in \R}H^\alpha(\T^d)$ and $\cS(\T^d)'\eqdef \bigcup_{\alpha\in \R}H^\alpha(\T^d)$.
Since the embedding $H^\alpha(\T^d)\hookrightarrow H^\beta(\T^d)$ is Hilbert-Schmidt for all $\alpha >\beta +d/2$, $(\cS(\T^d),L^2(\T^d),\cS(\T^d)')$ defines a countably Hilbert nuclear triple.  Hence, it suffices to establish uniqueness of solutions in the larger space $\cC([0,\infty);\cS(\T^d)')$. By the Gaussian nature of the problem, the latter  is precisely covered by Holley-Stroock theory \cite{holley_generalized_1978} (see also \cite[\S11.4]{kipnis_scaling_1999} or \cite[Section C.5]{jara2018nonequilibrium}).

\begin{appendix}
\section*{Proofs of Lemmas \ref{lemma:Q1}, \ref{lemma:Q2} and \ref{lemma:Q3}}
\label{sec:app}

In this appendix, we prove all results contained in Lemmas \ref{lemma:Q1}, \ref{lemma:Q2} and \ref{lemma:Q3} concerned with the quantity first introduced in \eqref{eq:B}, which we recall here for the reader's convenience: for all $t\ge 0$, $x,y \in \T_\eps^d$ and $i,j=1,\ldots, d$,
\begin{equation}\label{eq:q-eps}
	q_t^{\eps,i,j}(x,y)
	\eqdef \frac{\eps^{-2}}4\big(p_t^\eps(x-\eps e_j,y)+p_t^\eps(x+\eps e_i,y)-p_t^\eps(x+\eps e_i-\eps e_j,y)-p_t^\eps(x,y)\big)^2\fstop
\end{equation}
All throughout this section, we will exploit the invariance and product structure of the random walk $(X_t^\eps)_{t\ge 0}$ on $\T_\eps^d$: for all $t\ge 0$ and $x, y \in \T_\eps^d$ with $x=(x^i)_{i=1,\ldots, d}$, 
\begin{equation}\label{eq:pi-eps}
	p_t^\eps(x,y)=p_t^\eps(y,x)=p_t^\eps(0,x-y)\comma\qquad p_t^\eps(0,x)= \prod_{i=1}^d \pi_t^\eps(x^i)\comma
\end{equation}
where $\pi_t^\eps$ denotes the distribution of the continuous-time random walk on $\T_\eps$ (i.e., one-dimensional), started from the origin and with  nearest-neighbor jump rates equal to $\frac12\eps^{-2}$.

As already mentioned at the beginning of Section \ref{sec:mean-proofs}, Lemma \ref{lemma:Q2} essentially generalizes Lemma \ref{lemma:Q1}. For this reason, we find more convenient to prove these two lemmas together.
\begin{proof}[Proofs of Lemmas \ref{lemma:Q1} and \ref{lemma:Q2}] We start by expressing $q^\eps$ in \eqref{eq:q-eps} in terms of $\pi^\eps$, introduced in \eqref{eq:pi-eps}: for all $t\ge 0$, $x=(x^i)_{i=1,\ldots, d}\in \T_\eps^d$, and $i,j=1,\ldots, d$ with $i\neq j$,
	\begin{align}
		q_t^{\eps,i,j}(0,x)&=\Bigg(\prod_{\substack{\ell=1\\
				\ell \neq i,j}}^d  \pi_t^\eps(x^\ell)\Bigg)^2 
		\frac{\eps^{-2}}{4} 
		\tonde{\pi_t^\eps(x^i+\eps)-\pi_t^\eps(x^i)}^2\tonde{\pi_t^\eps(x^j-\eps)-\pi_t^\eps(x^j)}^2\\
		&= \Bigg(\prod_{\substack{\ell=1\\
				\ell \neq i,j}}^d \pi_t^\eps(x^\ell)\Bigg)^2 4\eps^2
		\tonde{\tfrac12 \nabla^\eps \pi_t^\eps(x^i)}^2\tonde{\tfrac12\nabla_*^\eps \pi_t^\eps(x^j)}^2
		\comma
	\end{align} 
	while, for $i=j=1,\ldots, d$, 
	\begin{align}
		q_t^{\eps,i,i}(0,x)&= \Bigg(\prod_{\substack{\ell=1\\
				\ell \neq i}}^d \pi_t^\eps(x^\ell)\Bigg)^2	\frac{\eps^{-2}}{4}\tonde{\pi_t^\eps(x^i+\eps)+\pi_t^\eps(x^i-\eps)-2\pi_t^\eps(x^i)}^2\\
		&= 	\Bigg(\prod_{\substack{\ell=1\\
				\ell \neq i}}^d \pi_t^\eps(x^\ell)\Bigg)^2 \eps^2 \tonde{\tfrac12 \Delta_\eps \pi_t^\eps(x^i)}^2\comma
	\end{align} 
	where all $\eps$-gradients and corresponding laplacians here are one-dimensional. Define, for all $t\ge0$,
	\begin{align}\label{eq:R-eps}
		\begin{aligned}
			\cR_\eps(t)&\eqdef \eps^{-1}\norm{\pi_t^\eps}^2_{L^2(\T_\eps)}=\pi_{2t}^\eps(0) \comma\\
			\cS_\eps(t)&\eqdef
			2 \norm{\tfrac12\nabla^\eps \pi_t^\eps}_{L^2(\T_\eps)}^2 =2\norm{\tfrac12\nabla^\eps_* \pi_t^\eps}_{L^2(\T_\eps)}^2\comma\\
			\cT_\eps(t) &\eqdef \eps\norm{\tfrac12 \Delta_\eps \pi_t^\eps}_{L^2(\T_\eps)}^2\fstop
		\end{aligned}
	\end{align}
	Since $\Delta_\eps=\nabla^\eps\nabla_*^\eps=\nabla_*^\eps\nabla^\eps$ and $\frac{\dd}{\dd t}\pi_t^\eps= \frac12\Delta_\eps \pi_t^\eps$, we have \begin{equation}\label{eq:derivatives}
		\cS_\eps(t)=-\frac{\eps}2\,\cR_\eps'(t)\comma\qquad \cT_\eps(t)=-\frac{\eps}2\,\cS_\eps'(t)\fstop\end{equation}
	Now, summing over $x\in \T_\eps^d$, we obtain
	\begin{equation}\label{eq:Q-i-j-R-S-T}
		\cQ_\eps^{i,j}(t)\eqdef \sum_{x\in \T_\eps^d}q_t^{\eps,i,j}(0,x) = \begin{dcases}
			\cR_\eps(t)^{d-2}\,	\cS_\eps(t)^2 &\text{if}\ i\neq j\ (d\ge 2)\\
			\cR_\eps(t)^{d-1}\,\cT_\eps(t) &\text{if}\ i=j\ (d\ge 1)\fstop
		\end{dcases}
	\end{equation}
	The form of the right-hand side above proves the first claim in Lemma \ref{lemma:Q2}. 
	
	Further summing over $j=1,\ldots, d$, we get
	\begin{align}\label{eq:Q-R-S}
		\begin{aligned}
			\cQ_\eps(t)\eqdef	\sum_{j=1}^d \cQ_\eps^{i,j}(t)&= {\cR_\eps(t)}^{d-1}\,\cT_\eps(t)+\tonde{d-1}{\cR_\eps(t)}^{d-2}\,\cS_\eps(t)^2\\
			&=-\frac{\eps}2\big(\cR_\eps(t)^{d-1}\cS_\eps(t)\big)'\comma
		\end{aligned}
	\end{align}
	where the second step follows from \eqref{eq:derivatives}. Hence, since $\cR_\eps(0)=1$,  $\cS_\eps(0)=\eps^{-1}$, and $\cS_\eps(t)\to 0$ as $t\to \infty$, we get
	\begin{equation}\label{eq:Q-12}
		\int_0^\infty \cQ_\eps(t)\, \dd t=-\frac{\eps}2\int_0^\infty \big(\cR_\eps(t)^{d-1}\cS_\eps(t)\big)'\, \dd t = \frac{\eps}2\big(\cR_\eps(0)^{d-1}\cS_\eps(0)\big)=\frac12\fstop
	\end{equation}
	This is precisely the main claim in Lemma \ref{lemma:Q1} (see also its reformulation in Remark \ref{remark:Q-Q-i-j}). This settles the analysis when $d=1$, since $\cQ_\eps^{i,j}=\car_{i=j=1}\,\cQ_\eps$ only in one dimension.
	
	In order to prove the final claim in Lemma \ref{lemma:Q2} on 
	\begin{equation}\label{eq:b-c-appendix}
		\mathfrak b_\eps =\int_0^\infty \cR_\eps(t)^{d-1}\, \cT_\eps(t)\, \dd t\comma\qquad \mathfrak c_\eps =\int_0^\infty \cR_\eps(t)^{d-2}\, \cS_\eps(t)^2\, \dd t\comma
	\end{equation}
	and the existence of their limits (see \eqref{eq:b-c-eps}--\eqref{eq:b-c-limits}), fix $d\ge 2$ and first observe that, since $\cR_\eps(t)$, $\cS_\eps(t)$ and $\cQ_\eps(t)$ are all strictly positive for all $t\ge  0$, \eqref{eq:Q-i-j-R-S-T} and \eqref{eq:Q-12} ensure that $\mathfrak b_\eps, \mathfrak c_\eps \in (0,\frac12)$. Furthermore, by the first identity in \eqref{eq:b-c-12}, it suffices to show that either $\mathfrak b_\eps$ or $\mathfrak c_\eps$ converges to a value in $(0,\frac12)$. Let us focus on $\mathfrak c_\eps$.
	
	Recall \eqref{eq:R-eps}. 	 From Laplace inversion formulas (see, e.g., \cite[Eq.\ (2.3)]{cox_coalescing_1989}), we know
	\begin{equation}
		\cR_\eps(t)
		= \eps\sum_{j=0}^{ \eps^{-1}-1} \exp\set{-2\,\psi(\eps j)\, \eps^{-2}\,t}\comma \quad \text{with}\ z\in \T\mapsto \psi(z)\eqdef 1-\cos(2\pi z)\fstop
	\end{equation}
	This and the first identity in \eqref{eq:derivatives} yield
	\begin{align}
		\mathfrak c_\eps
		&= \int_0^\infty \eps^{d-2}\sum_{j_1,\ldots, j_d=0}^{\eps^{-1}-1}\psi(\eps j_1)\, \psi(\eps j_2) \exp\set{-2\tonde{\psi(\eps j_1)+\ldots+\psi(\eps j_d)}\eps^{-2}\,t}\dd t\\
		&= \frac{\eps^d}{2}\sum_{j_1,\ldots, j_d\purple{=0}}^{\eps^{-1}-1}\frac{\psi(\eps j_1)\, \psi(\eps j_2)}{\psi(\eps j_1)+\ldots + \psi(\eps j_d)}\comma
	\end{align}
	which converges, as $\eps \to 0$, to
	\begin{equation}\label{eq:c-explicit}
		\mathfrak c\eqdef \frac12 \int_{\T^d} \frac{\psi(x^1)\,\psi(x^2)}{\psi(x^1)+\ldots + \psi(x^d)}\, \dd x\comma\quad \text{with}\ x=(x^i)_{i=1,\ldots, d}\in \T^d\fstop
	\end{equation}
	Since the above integral is strictly positive, this completes the proof of Lemma \ref{lemma:Q2}.\end{proof}

\begin{remark}[Properties of $\mathfrak b$ and $\mathfrak c$]\label{remark:coefficients-b-c}
	The above proof also reveals  the behaviors of the coefficients  $\mathfrak b$ and $\mathfrak c$ as functions of the dimension $d\ge 1$. 
	\begin{enumerate}
		\item \emph{Monotonicity}. Since $\cR_\eps(t)\in (0,1)$  and $\cS_\eps(t),\cT_\eps(t)>0$ for all $t>0$, \eqref{eq:Q-i-j-R-S-T} and \eqref{eq:b-c-appendix} state that both $\mathfrak b_\eps$ and $\mathfrak c_\eps$ --- and, thus, their limits $\mathfrak b$ and $\mathfrak c$, too --- are strictly decreasing with $d\ge 1$. 
		\item \emph{Bounds}. By \eqref{eq:b-c-12} and $\mathfrak b\ge 0$, we have the upper bound $\mathfrak c\le \frac1{2\tonde{d-1}}$. Since $\psi(z)\in [0,2]$ and $\int_\T \psi(z)\, \dd z=1$, the integral in \eqref{eq:c-explicit} yields the lower bound $\mathfrak c\ge \frac1{4d}$.
		\item \emph{2D}. When $d=2$, the value of the integral in \eqref{eq:c-explicit} is explicit:
		\begin{equation}\label{eq:coefficient-c-d=2}
			\mathfrak c=\frac12 \int_\T\int_\T\frac{\psi(x^1)\,\psi(x^2)}{\psi(x^1)+\psi(x^2)}\, \dd x^1\dd x^2 = \frac12-\frac1\pi\comma\qquad \mathfrak b=\frac1\pi\comma
		\end{equation}
		where for the second identity we used \eqref{eq:b-c-12}.
	\end{enumerate} 
\end{remark}

\begin{remark}[Formula for $\mathfrak a$]\label{remark:formula-a}
	By combining \eqref{eq:a} and \eqref{eq:c-explicit}, we get, for all $d\ge2$, 
	\begin{equation}
		\mathfrak a=\mathfrak a(d)= \frac{1-(d-2)\int_{\T^d}\frac{\psi(x^1)\,\psi(x^2)}{\psi(x^1)+\ldots +\psi(x^d)}\,\dd x}{1+d\int_{\T^d}\frac{\psi(x^1)\,\psi(x^2)}{\psi(x^1)+\ldots +\psi(x^d)}\,\dd x}\comma\quad \text{with}\ x=(x^i)_{i=1,\ldots, d}\in \T^d\fstop
	\end{equation}
	
\end{remark}
We conclude the appendix with the short proof of Lemma \ref{lemma:Q3}.
\begin{proof}[Proof of Lemma \ref{lemma:Q3}]
	By \eqref{eq:Q-R-S} and arguing as for \eqref{eq:Q-12}, we have
	\begin{equation}
		\int_T^\infty \cQ_\eps(t)\, \dd t= \frac{\eps}2\,\big(\cR_\eps(T)^{d-1}\, \cS_\eps(T)\big)\le \frac{\eps}2\, \cS_\eps(T)\comma\qquad T>0\comma
	\end{equation}
	where for the inequality we used $\cR_\eps\le 1$. The desired estimate in \eqref{eq:lemma-Q3} is a consequence of the fact that $\cS_\eps(T)$ is, up to normalization, the Dirichlet form of the diffusively-rescaled continuous-time random walk on $\T_\eps^n$ with $n=1$, which is well-known to have the power law behavior  $T^{-(n+1/2)}=T^{-3/2}$ for $T\in (\eps^2,1)$. For a simple proof of this fact, see, e.g.,  \cite[Eq.\ (5.9)]{banerjee2020rates}.	
\end{proof}
\end{appendix}
\begin{acks}[Acknowledgments]
	The author thanks Assaf Shapira and Matteo Quattropani for fruitful discussions, and the two anonymous referees for their constructive comments.
\end{acks}

\begin{funding}
	The author is a member of GNAMPA, INdAM, and  the PRIN project TESEO (2022HSSYPN), and acknowledges financial support by \textquotedblleft Microgrants 2022\textquotedblright, funded by Regione FVG, legge LR 2/2011.
\end{funding}

\bibliographystyle{alpha}

\begin{thebibliography}{BRAS06}
	
	\bibitem[AD18]{armstrong_dario_elliptic_2018}
	Scott Armstrong and Paul Dario.
	  Elliptic regularity and quantitative homogenization on percolation
	clusters.
	  {\em Comm. Pure Appl. Math.}, 71(9):1717--1849, 2018.
	
	\bibitem[AKM19]{armstrong_kuusi_mourrat_quantitative_2019}
	Scott Armstrong, Tuomo Kuusi, and Jean-Christophe Mourrat.
	  {\em Quantitative stochastic homogenization and large-scale
		regularity}, volume 352 of {\em Grundlehren der mathematischen
		Wissenschaften}.
	  Springer, Cham, 2019.
	
	\bibitem[AL12]{aldous_lecture_2012}
	David Aldous and Daniel Lanoue.
	  A lecture on the averaging process.
	  {\em Probab. Surv.}, 9:90--102, 2012.
	
	\bibitem[Ald13]{aldous2013interacting}
	David Aldous.
	  Interacting particle systems as stochastic social dynamics.
	  {\em Bernoulli}, 19(4):1122--1149, 2013.
	
	\bibitem[BB21]{banerjee2020rates}
	Sayan Banerjee and Krzysztof Burdzy.
	  Rates of convergence to equilibrium for potlatch and smoothing
	processes.
	  {\em Ann. Probab.}, 49(3):1129--1163, 2021.
	
	\bibitem[BGPS06]{boyd_et_al_randomized_2006}
	Stephen Boyd, Arpita Ghosh, Balaji Prabhakar, and Devavrat Shah.
	 Randomized gossip algorithms.
	  {\em IEEE Trans. Inform. Theory}, 52(6):2508--2530, 2006.
	
	\bibitem[Bil99]{billingsley_convergence_1999}
	Patrick Billingsley.
	  {\em Convergence of probability measures}.
	  Wiley Series in Probability and Statistics: Probability and
	Statistics. John Wiley \& Sons, Inc., New York, second edition, 1999.
	  A Wiley-Interscience Publication.
	
	\bibitem[BR18]{biskup_rodriguez_limit_2018}
	Marek Biskup and Pierre-Fran\c{c}ois Rodriguez.
	  Limit theory for random walks in degenerate time-dependent random
	environments.
	  {\em J. Funct. Anal.}, 274(4):985--1046, 2018.
	
	\bibitem[BRAS06]{balazs_rassoul-agha_seppalainen_random_2006}
	M\'{a}rton Bal\'{a}zs, Firas Rassoul-Agha, and Timo Sepp\"{a}l\"{a}inen.
	  The random average process and random walk in a space-time random
	environment in one dimension.
	  {\em Comm. Math. Phys.}, 266(2):499--545, 2006.
	
	\bibitem[CDSZ22]{chatterjee2020phase}
	Sourav Chatterjee, Persi Diaconis, Allan Sly, and Lingfu Zhang.
	  A phase transition for repeated averages.
	  {\em Ann. Probab.}, 50(1):1--17, 2022.
	
	\bibitem[CF11]{como_scaling2011}
	Giacomo Como and Fabio Fagnani.
	  Scaling limits for continuous opinion dynamics systems.
	  {\em Ann. Appl. Probab.}, 21(4):1537--1567, 2011.
	
	\bibitem[CGQ24]{clozeau_gloria_qi_quantitative_2024}
	Nicolas Clozeau, Antoine Gloria, and Siguang Qi.
	  Quantitative homogenization for log-normal coefficients.
	  {\em arXiv:2403.00168}, 2024.
	
	\bibitem[Cox89]{cox_coalescing_1989}
	J.~T. Cox.
	  Coalescing random walks and voter model consensus times on the torus
	in {${\bf Z}^d$}.
	  {\em Ann. Probab.}, 17(4):1333--1366, 1989.
	
	\bibitem[CQS23]{caputo_quattropani_sau_cutoff_2023}
	Pietro Caputo, Matteo Quattropani, and Federico Sau.
	  Cutoff for the averaging process on the hypercube and complete
	bipartite graphs.
	  {\em Electron. J. Probab.}, 28:Paper No. 100, 31, 2023.
	
	\bibitem[Dar21]{dario_optimal_2021}
	Paul Dario.
	  Optimal corrector estimates on percolation cluster.
	  {\em Ann. Appl. Probab.}, 31(1):377--431, 2021.
	
	\bibitem[DG20a]{duerinckx_gloria_multiscale_concentration_2020}
	Mitia Duerinckx and Antoine Gloria.
	  Multiscale functional inequalities in probability: concentration
	properties.
	  {\em ALEA Lat. Am. J. Probab. Math. Stat.}, 17(1):133--157, 2020.
	
	\bibitem[DG20b]{duerinckx_gloria_multiscale_constructive_2020}
	Mitia Duerinckx and Antoine Gloria.
	  Multiscale functional inequalities in probability: constructive
	approach.
	  {\em Ann. H. Lebesgue}, 3:825--872, 2020.
	
	\bibitem[DG22]{duerinckx_gloria_quantitative_2022}
	Mitia Duerinckx and Antoine Gloria.
	  Quantitative homogenization theory for random suspensions in steady
	{S}tokes flow.
	  {\em J. \'{E}c. polytech. Math.}, 9:1183--1244, 2022.
	
	\bibitem[DGO20]{duerinckx_gloria_otto_structure_2020}
	Mitia Duerinckx, Antoine Gloria, and Felix Otto.
	  The structure of fluctuations in stochastic homogenization.
	  {\em Comm. Math. Phys.}, 377(1):259--306, 2020.
	
	\bibitem[DMP91]{de_masi_mathematical_1991}
	Anna De~Masi and Errico Presutti.
	  {\em Mathematical methods for hydrodynamic limits}, volume 1501 of
	{\em Lecture Notes in Mathematics}.
	  Springer-Verlag, Berlin, 1991.
	
	\bibitem[DO20]{duerinckx_otto_higher_2020}
	Mitia Duerinckx and Felix Otto.
	  Higher-order pathwise theory of fluctuations in stochastic
	homogenization.
	  {\em Stoch. Partial Differ. Equ. Anal. Comput.}, 8(3):625--692, 2020.
	
	\bibitem[DSC96]{diaconis1996nash}
	P.~Diaconis and L.~Saloff-Coste.
	  Nash inequalities for finite {M}arkov chains.
	  {\em J. Theoret. Probab.}, 9(2):459--510, 1996.
	
	\bibitem[DSPS24]{delloschiavo_portinale_sau_scaling_2021}
	Lorenzo Dello~Schiavo, Lorenzo Portinale, and Federico Sau.
	  Scaling limits of random walks, harmonic profiles, and stationary
	non-equilibrium states in {L}ipschitz domains. 
	  {\em Ann. Appl. Probab.}, 34(2):1789--1845, 2024.
	
	\bibitem[EK86]{ethier_kurtz_1986_Markov}
	Stewart~N. Ethier and Thomas~G. Kurtz.
	  {\em Markov processes. Characterization and convergence}.
	  Wiley Series in Probability and Mathematical Statistics: Probability
	and Mathematical Statistics. John Wiley \& Sons, Inc., New York, 1986.
	
	\bibitem[FF98]{ferrari_fontes_fluctuations_1998}
	P.~A. Ferrari and L.~R.~G. Fontes.
	  Fluctuations of a surface submitted to a random average process.
	  {\em Electron. J. Probab.}, 3:no. 6, 34, 1998.
	
	\bibitem[FMZ24]{fontes_machado_zuaznabar_scaling_2023}
	Luiz~Renato Fontes, Mariela~Pent{\'o}n Machado, and Leonel Zuazn{\'a}bar.
	  Scaling limit of an equilibrium surface under the random average
	process.
	  {\em Electron. J. Probab}. 29:1--28 2024.
	
	\bibitem[FS86]{fabes_new1986}
	E.~B. Fabes and D.~W. Stroock.
	  A new proof of {M}oser's parabolic {H}arnack inequality using the old
	ideas of {N}ash.
	  {\em Arch. Rational Mech. Anal.}, 96(4):327--338, 1986.
	
	\bibitem[GM16]{gu_mourrat_scaling_2016}
	Yu~Gu and Jean-Christophe Mourrat.
	  Scaling limit of fluctuations in stochastic homogenization.
	  {\em Multiscale Model. Simul.}, 14(1):452--481, 2016.
	
	\bibitem[GQ24]{gloria_qi_quantitative_2024}
	Antoine Gloria and Siguang Qi.
	  Quantitative homogenization for log-normal coefficients via malliavin
	calculus: the one-dimensional case.
	  {\em arXiv:2402.19182}, 2024.
	
	\bibitem[HH14]{haggstrom2014further}
	Olle H\"{a}ggstr\"{o}m and Timo Hirscher.
	  Further results on consensus formation in the {D}effuant model.
	  {\em Electron. J. Probab.}, 19:no. 19, 26, 2014.
	
	\bibitem[HS78]{holley_generalized_1978}
	Richard~A. Holley and Daniel~W. Stroock.
	  Generalized {O}rnstein-{U}hlenbeck processes and infinite particle
	branching {B}rownian motions.
	  {\em Publ. Res. Inst. Math. Sci.}, 14(3):741--788, 1978.
	
	\bibitem[JL23]{jara_landim_stochastic_2023}
	Milton Jara and Claudio Landim.
	  The stochastic heat equation as the limit of a stirring dynamics
	perturbed by a voter model.
	  {\em Ann. Appl. Probab.}, 33(6A):4163--4209, 2023.
	
	\bibitem[JM18]{jara2018nonequilibrium}
	Milton Jara and Ot{\'a}vio Menezes.
	  Non-equilibrium fluctuations of interacting particle systems.
	  {\em arXiv:1810.09526}, 2018.
	
	\bibitem[JO22]{josien_otto_annealed_2022}
	Marc Josien and Felix Otto.
	  The annealed {C}alder\'{o}n-{Z}ygmund estimate as convenient tool in
	quantitative stochastic homogenization.
	  {\em J. Funct. Anal.}, 283(7):Paper No. 109594, 74, 2022.
	
	\bibitem[KL99]{kipnis_scaling_1999}
	Claude Kipnis and Claudio Landim.
	  {\em Scaling limits of interacting particle systems}, volume 320 of
	{\em Grundlehren der Mathematischen Wissenschaften}.
	  Springer-Verlag, Berlin, 1999.
	
	\bibitem[Koz79]{kozlov_averaging_1979}
	S.~M. Kozlov.
	  The averaging of random operators.
	  {\em Mat. Sb. (N.S.)}, 109(151)(2):188--202, 327, 1979.
	
	\bibitem[Lan12]{lanchier_critical_2012}
	N.~Lanchier.
	  The critical value of the {D}effuant model equals one half.
	  {\em ALEA Lat. Am. J. Probab. Math. Stat.}, 9(2):383--402, 2012.
	
	\bibitem[LP11]{last_penrose_poisson_2011}
	G\"{u}nter Last and Mathew~D. Penrose.
	  Poisson process {F}ock space representation, chaos expansion and
	covariance inequalities.
	  {\em Probab. Theory Related Fields}, 150(3-4):663--690, 2011.
	
	\bibitem[LP18]{last_penrose_lectures_2018}
	G\"{u}nter Last and Mathew Penrose.
	  {\em Lectures on the {P}oisson process}, volume~7 of {\em Institute
		of Mathematical Statistics Textbooks}.
	  Cambridge University Press, Cambridge, 2018.
	
	\bibitem[MSW24]{movassagh_repeated2022}
	Ramis Movassagh, Mario Szegedy, and Guanyang Wang.
	  Repeated averages on graphs.
	  {\em Ann. Appl. Probab.}, 34(4): 3781--3819, 2024.
	
	\bibitem[Nas58]{nash_continuity1958}
	J.~Nash.
	  Continuity of solutions of parabolic and elliptic equations.
	  {\em Amer. J. Math.}, 80:931--954, 1958.
	
	\bibitem[PV81]{papanicolau_varadhan_boundary_1981}
	G.~C. Papanicolaou and S.~R.~S. Varadhan.
	  Boundary value problems with rapidly oscillating random coefficients.
	  In {\em Random fields, {V}ol. {I}, {II} ({E}sztergom, 1979)},
	volume~27 of {\em Colloq. Math. Soc. J\'{a}nos Bolyai}, pages 835--873.
	North-Holland, Amsterdam-New York, 1981.
	
	\bibitem[QS23]{quattropani2021mixing}
	Matteo Quattropani and Federico Sau.
	  Mixing of the averaging process and its discrete dual on
	finite-dimensional geometries.
	  {\em Ann. Appl. Probab.}, 33(2):936--971, 2023.
	
	\bibitem[Rho08]{rhodes_homogenization_2008}
	R\'{e}mi Rhodes.
	  On homogenization of space-time dependent and degenerate random
	flows. {II}.
	  {\em Ann. Inst. Henri Poincar\'{e} Probab. Stat.}, 44(4):673--692,
	2008.
	
	\bibitem[Sau24]{sau_concentration_2023}
	Federico Sau.
	  Concentration and local smoothness of the averaging process.
	  {\em Electron. J. Probab.}, 29:1--26, 2024.
	
\end{thebibliography}

\end{document}